\documentclass{article}
\usepackage{graphicx}
\usepackage{amssymb, amsmath, amsthm,mathtools}
\usepackage{epstopdf}

\usepackage{bbm}
\usepackage[margin=2.5cm]{geometry}
\usepackage{layout, color}
\usepackage{bbm}
\usepackage{enumerate}
\usepackage{authblk}

\usepackage{algorithm}
\usepackage{algpseudocode}

\usepackage{setspace}

\usepackage{hyperref}

\usepackage{url}
\usepackage{subfig}
\usepackage{mdwlist}
\usepackage{multirow}
\usepackage{amsthm}
\usepackage{color}
\usepackage[title]{appendix}

\newtheorem{theorem}{Theorem}[section]
\newtheorem{proposition}[theorem]{Proposition}
\newtheorem{assump}{Assumption}

\newtheorem{lemma}[theorem]{Lemma}
\newtheorem{definition}[theorem]{Definition}

\theoremstyle{remark}
\newtheorem{remark}{Remark}

\newtheorem{conjecture*}{Conjecture}

\theoremstyle{plain}

\usepackage[dvipsnames]{xcolor}

\title{Robust Inference of Manifold Density and Geometry \\ by Doubly Stochastic Scaling}
\author{ Boris Landa${^{1,*}}$~~~~Xiuyuan Cheng${^{2}}$\\
\small{${^1}$Program in Applied Mathematics, Yale University}\\
\small{${^2}$Department of Mathematics, Duke University}\\
\small{${^*}$Corresponding author. Email: boris.landa@yale.edu}
}
\date{}

\begin{document}

\maketitle

\begin{abstract}
The Gaussian kernel and its traditional normalizations (e.g., row-stochastic) are popular approaches for assessing similarities between data points. Yet, they can be inaccurate under high-dimensional noise, especially if the noise magnitude varies considerably across the data, e.g., under heteroskedasticity or outliers. In this work, we investigate a more robust alternative -- the doubly stochastic normalization of the Gaussian kernel. We consider a setting where points are sampled from an unknown density on a low-dimensional manifold embedded in high-dimensional space and corrupted by possibly strong, non-identically distributed, sub-Gaussian noise. We establish that the doubly stochastic affinity matrix and its scaling factors concentrate around certain population forms, and provide corresponding finite-sample probabilistic error bounds. We then utilize these results to develop several tools for robust inference under general high-dimensional noise. First, we derive a robust density estimator that reliably infers the underlying sampling density and can substantially outperform the standard kernel density estimator under heteroskedasticity and outliers. Second, we obtain estimators for the pointwise noise magnitudes, the pointwise signal magnitudes, and the pairwise Euclidean distances between clean data points. Lastly, we derive robust graph Laplacian normalizations that accurately approximate various manifold Laplacians, including the Laplace Beltrami operator, improving over traditional normalizations in noisy settings. We exemplify our results in simulations and on real single-cell RNA-sequencing data. For the latter, we show that in contrast to traditional methods, our approach is robust to variability in technical noise levels across cell types.
\end{abstract}

\section{Introduction}
\subsection{Traditional normalizations of the Gaussian kernel}
Many popular techniques for clustering, manifold learning, visualization, and semi-supervised learning, begin by learning similarities between observations. The learned similarities then form an affinity matrix that describes a weighted graph, which is further processed and analyzed according to the required task. A popular approach to construct an affinity matrix from the data is to evaluate the Gaussian kernel with pairwise Euclidean distances.
Specifically, letting $y_1,\ldots,y_n\in\mathbb{R}^m$ be a collection of given data points, we define the Gaussian kernel $\mathcal{K}_\epsilon: \mathbb{R}^m\times \mathbb{R}^m \rightarrow \mathbb{R}_+$ and the resulting kernel matrix $K\in\mathbb{R}^{n\times n}$ as
\begin{equation}
    {K}_{i,j} = 
    \begin{dcases}
    {\mathcal{K}_{\epsilon}({y}_i,{y}_j)}, & i \neq j, \\
    0, & i=j, 
    \end{dcases},
    \qquad\qquad \mathcal{K}_{\epsilon}(y_i,y_j) = \operatorname{exp}\left\{ -\frac{ \Vert y_i - y_j \Vert_2^2}{\epsilon} \right\}, \label{eq: Gaussian kernel K def}
\end{equation}
for $i,j=1,\ldots,n$, where $\epsilon>0$ is a tunable bandwidth parameter. Here we adopt the version of the kernel matrix $K$ whose main diagonal is zeroed out, namely with no self-loops in the graph described by $K$. This choice will be further motivated from the viewpoint of noise robustness in Section~\ref{sec: introduction - noise robustness}. 

The kernel matrix $K$ is often normalized to attain certain favorable properties. For instance, a popular choice is to divide each row of $K$ by its sum to make it a transition probability matrix. Generally, a family of normalizations that underlies many methods can be expressed by ${P}^{(\alpha)}\in\mathbb{R}^{n\times n}$ or $\hat{P}^{(\alpha)}\in\mathbb{R}^{n\times n}$ given by
\begin{equation}
    \hat{P}^{(\alpha)}_{i,j} = \frac{{P}^{(\alpha)}_{i,j}}{\sum_{j=1}^n {P}^{(\alpha)}_{i,j}}, \qquad\qquad {P}^{(\alpha)} = D^{-\alpha} K D^{-\alpha}, \qquad\qquad D_{i,i} = \sum_{j=1}^n K_{i,j}, \label{eq: traditional graph Laplacian normalizations}
\end{equation}
where $\alpha \in [0,1]$ is a parameter of the normalization, and $D\in\mathbb{R}^{n\times n}$ is a diagonal matrix whose main diagonal holds the degrees of the nodes in the weighted graph represented by $K$. If $\alpha=0$, then $\hat{P}^{(\alpha)}$ describes the popular row-stochastic normalization $D^{-1} K$, and if $\alpha = 0.5$, then $P^{(\alpha)}$ describes its symmetric variant $D^{-1/2} K D^{-1/2}$. These normalizations have been utilized and extensively studied in the context of clustering~\cite{shi2000normalized,ng2002spectral,sarkar2015role,von2007tutorial,fortunato2010community,trillos2021geometric}, non-linear dimensionality reduction (or manifold learning)~\cite{belkin2003laplacian,coifman2006diffusionMaps,nadler2006diffusion,maaten2008visualizing}, image denoising~\cite{buades2005non,pang2017graph,meyer2014perturbation,landa2018steerable,singer2009diffusion}, and graph-based signal processing and supervised-learning~\cite{shuman2013emerging,coifman2006diffusionWavelets,hammond2011wavelets,defferrard2016convolutional,bronstein2017geometric}. 

An important theoretical aspect of various normalizations is the convergence of the corresponding graph Laplacian to a differential operator as $n\rightarrow \infty$ and $\epsilon\rightarrow 0$, typically under the assumption that the points $y_1,\ldots,y_n$ are sampled from a low-dimensional Riemannian manifold $\mathcal{M}$ embedded in $\mathbb{R}^m$; see~\cite{singer2006graph,hein2005graphs,trillos2020error,calder2019improved,dunson2021spectral,cheng2021eigen} and references therein.
The family of normalizations in~\eqref{eq: traditional graph Laplacian normalizations} was proposed in the Diffusion Maps paper~\cite{coifman2006diffusionMaps}, where it was shown that a population analogue (i.e., a continuous surrogate in the limit $n\rightarrow\infty$) of the graph Laplacian $I_n - \hat{P}^{(\alpha)}$ converges to a certain differential operator parametrized by $\alpha$; see Section~\ref{sec: graph Laplacian estimation} for more details. This operator is particularly appealing from the viewpoint of diffusion on the manifold $\mathcal{M}$; in the case of $\alpha=1$, this operator is precisely the Laplace-Beltrami operator~\cite{grigor2006heat}, which determines the solution to the heat equation and encodes the intrinsic geometry of the manifold regardless of the sampling density. Note that the parameter $\alpha$ in~\eqref{eq: traditional graph Laplacian normalizations} determines the entrywise power of the degree matrix $D$, whose diagonal entries are $\sum_{j=1, \; j\neq i}^n \mathcal{K}_\epsilon (y_i,y_j)$, which can be interpreted as a kernel density estimator evaluated at the sample points $y_1,\ldots,y_n$~\cite{parzen1962estimation,rosenblatt1956remarks}. Indeed, $\alpha$ controls the influence of the sampling density on the resulting affinity matrix and its spectral behavior~\cite{coifman2006diffusionMaps}. 

\subsection{The doubly stochastic normalization}
An alternative normalization of $K$ that is not covered by $\hat{P}^{(\alpha)}$ or $P^{(\alpha)}$ of~\eqref{eq: traditional graph Laplacian normalizations} is the doubly stochastic normalization~\cite{zass2005unifying,zass2007doubly,marshall2019manifold,wormell2021spectral}
\begin{equation}
    W_{i,j} = d_i {K}_{i,j} d_j, \qquad\qquad \sum_{j=1}^{n} W_{i,j} = 1, \label{eq:discrete scaling eq}
\end{equation}
for $i=1,\ldots,n$, where $d_1,\ldots,d_n > 0$ are the \textit{scaling factors} of ${K}$. Since the resulting $W$ is symmetric and stochastic (i.e., each row sums to $1$), it is also doubly stochastic. The problem of finding $\mathbf{d} = [d_1,\ldots,d_n] > 0$ is known as \textit{matrix scaling}, and has been extensively studied in the literature; see~\cite{bapat1997nonnegative,idel2016review} and references therein.
In the case of~\eqref{eq:discrete scaling eq}, the scaling factors are guaranteed to exist and are unique for $n>2$ since ${K}$ is fully indecomposable (see Lemma 1 in~\cite{landa2021doubly}). Although the scaling factors $\mathbf{d}$ do not admit a closed-form solution, they can be found numerically, e.g., by the classical Sinkhorn-Knopp algorithm~\cite{sinkhorn1967concerning}, convex optimization~\cite{allen2017much}, or algorithms specialized for symmetric matrices~\cite{knight2014symmetry,wormell2021spectral}. The doubly stochastic normalization is also closely related to entropic optimal transport~\cite{cuturi2013sinkhorn,peyre2019computational}. In particular, $W_{i,j}$ is the optimal transport plan between $y_i$ and $y_j$ according to the squared Euclidean distance loss with an entropic regularization term weighted by $\epsilon$ (see Proposition 2 in~\cite{landa2021doubly}).  

Doubly stochastic affinity matrices proved useful for various tasks such as clustering~\cite{zass2007doubly,beauchemin2015affinity,wang2012improving,lim2020doubly,ah2022learning,douik2018riemannian,chen2022robust}, manifold learning~\cite{marshall2019manifold,wormell2021spectral,cheng2022bi}, image denoising~\cite{milanfar2013symmetrizing}, and graph matching~\cite{coifman2022common}, often exhibiting more stable behavior and outperforming traditional normalizations. We note that requiring an affinity matrix to be doubly stochastic is appealing from a geometric perspective since the heat kernel on a compact Riemannian manifold -- an operator describing affinities between points according to the intrinsic geometry -- is always doubly stochastic~\cite{grigor2006heat}.

In the context of manifold learning, the doubly stochastic normalization was recently investigated in~\cite{marshall2019manifold,wormell2021spectral,cheng2022bi}. Specifically,~\cite{marshall2019manifold} analyzed a population setting where the scaling equation~\eqref{eq:discrete scaling eq} is replaced with a more general family of integral equations parametrized by $\alpha$ (see eq.~\eqref{eq:integral scaling eq with density} in Section~\ref{sec: main results} for the special case of $\alpha=0.5$). It was shown that the limiting differential operator can be made the same as for the traditional normalization $\hat{P}^{(\alpha)}$ from~\eqref{eq: traditional graph Laplacian normalizations} as $\epsilon \rightarrow 0$ for any $\alpha \in [0,1]$. In a different direction, \cite{wormell2021spectral} focused on the case where the manifold is a torus, and derived spectral convergence rates for $W$ as $n\rightarrow\infty$ and $\epsilon\rightarrow 0$, showing that the doubly stochastic normalization admits an improved bias error term compared to the traditional normalization in~\eqref{eq: traditional graph Laplacian normalizations} when $\alpha=0.5$. Lastly,~\cite{cheng2022bi} investigated a regularized version of the scaling equation~\eqref{eq:discrete scaling eq} where the scaling factors are lower-bounded, and established operator convergence rates of the corresponding graph Laplacian to the same operator as for $\hat{P}^{(0.5)}$ as $n\rightarrow\infty$ and $\epsilon\rightarrow 0$ on general manifolds.
We note that theoretical investigation of the doubly stochastic normalization is typically more challenging than that of more traditional normalizations, particularly due to the implicit and nonlinear nature of the scaling equation~\eqref{eq:discrete scaling eq} and its population analogue (which is a nonlinear integral equation). 

\subsection{The influence of noise} \label{sec: introduction - noise robustness}
Modern experimental procedures such as single-cell RNA-sequencing (scRNA-seq)~\cite{brennecke2013accounting,jia2017accounting,kim2015characterizing}, cryo-electron microscopy~\cite{baxter2009determination,scheres2012bayesian,henderson2013avoiding}, and calcium imaging~\cite{malik2011denoising,rupasinghe2021direct,lohani2020dual}, to name a few, produce large datasets of high-dimensional observations often corrupted by strong noise. In such cases, the classical theoretical setup where data points reside on, or near, a low-dimensional manifold can be highly inaccurate. To model noisy data, we consider
\begin{equation}
    {y}_i = {x}_i + \eta_i, \label{eq: noisy data} 
\end{equation}
where $x_1,\ldots,x_n$ are the underlying clean observations and $\eta_1,\ldots,\eta_n\in\mathbb{R}^m$ are independent noise random vectors with zero means (to be specified in detail later on). 

For identically distributed noise vectors $\{\eta_i\}_{i=1}^n$, the influence of noise on pairwise Euclidean distances and the entries of a kernel matrix was derived in~\cite{el2010information}. Specifically, if the noise magnitudes $\Vert \eta_i \Vert_2^2$ concentrate well in high dimension around a global constant $c$, it was shown that $\Vert {y}_i - {y}_j \Vert_2^2 \approx \Vert {x}_i - {x}_j \Vert_2^2 + 2c$ for all $i\neq j$. Consequently, the noisy Gaussian kernel $\mathcal{K}_\epsilon({y}_i,{y}_j)$ is biased for $i\neq j$ by a global multiplicative factor. Since this factor does not exist on the main diagonal of the Gaussian kernel matrix, it is advantageous to zero it out (as done in $K$ from~\eqref{eq: Gaussian kernel K def}). This way, the multiplicative factor cancels out automatically through the traditional normalization $\hat{P}^{(\alpha)}$ (for any $\alpha$) or $P^{(\alpha)}$ with $\alpha=0.5$; see~\cite{el2016graph} for more details. 

In many applications, the noise characteristics vary considerably across the data due to heteroskedasticity and outliers. In particular, heteroskedastic noise is prevalent in applications involving count or nonnegative data, typically modeled by, e.g., Poisson, binomial, negative binomial, or gamma distributions, whose variances inherently depend on their means, which can vary substantially across the data. Notable examples for such data are network traffic analysis~\cite{shen2005analysis}, photon imaging~\cite{salmon2014poisson}, document topic modeling~\cite{wallach2006topic}, scRNA-seq~\cite{hafemeister2019normalization}, and high-throughput chromosome conformation capture~\cite{johanson2018genome}, among many others. Heteroskedastic noise also arises in natural image processing due to spatial pixel clipping~\cite{foi2009clipped} and in experimental procedures where conditions vary during data acquisition, such as in spectrophotometry and atmospheric data collection~\cite{cochran1977statistically,tamuz2005correcting}. Besides natural heteroskedasticity, experimental data often include outlier observations with abnormal noise distributions due to, e.g., abrupt deformations or technical errors during acquisition and storage. Consequently, to better understand the possible advantages of doubly stochastic normalization, it is important to investigate it under general non-identically distributed noise that supports heteroskedasticity and outliers.

If the noise vectors $\{\eta_i\}_{i=1}^n$ are not identically distributed or if $\Vert \eta_i \Vert_2^2$ do not concentrate well around a global constant, then the noisy Euclidean distances $\Vert {y}_i - {y}_j \Vert_2^2$ can be corrupted in a nontrivial way. In such cases, as demonstrated in~\cite{landa2021doubly}, the Gaussian kernel and several of its traditional normalizations can behave unexpectedly and incorrectly assess the similarities between data points.
On the other hand,~\cite{landa2021doubly} also shows that the doubly stochastic normalization is robust to non-identically distributed high-dimensional noise.
Specifically, under suitable conditions on the noise, and if $\epsilon$ and $n$ are fixed while $m$ is growing, ${W}$ converges pointwise in probability to its clean counterpart, even if the noise magnitudes $\Vert \eta_i \Vert_2^2$ are comparable to the signal magnitudes $\Vert x_i \Vert_2^2$ and fluctuate considerably. While this result highlights an important advantage of the doubly stochastic normalization, it does not account for the sample size $n$ nor the bandwidth $\epsilon$. Hence, it remains unclear what is the population interpretation of the doubly stochastic normalization under noise in terms of the sampling density and the underlying geometry.

\subsection{Our results and contributions}
In this work, we consider a setting where $x_1,\ldots,x_n$ are sampled from a low-dimensional manifold embedded in $\mathbb{R}^m$, and $\eta_1,\ldots,\eta_n$ are sampled from non-identically distributed sub-Gaussian noise that allows for heteroskedasticity and outliers. Our analysis is carried out in a high-dimensional regime in which the noisy Euclidean distances satisfy $\Vert y_i - y_j\Vert_2^2 = \Vert x_i - x_j\Vert_2^2 + \Vert \eta_i \Vert_2^2 + \Vert \eta_j \Vert_2^2 + {o}(1)$, where $\Vert \eta_i \Vert_2^2$ are unknown and can be large, and the ${o}(1)$ term is vanishing as $m \rightarrow \infty$ but is explicitly accounted for. 
Our main contributions are two-fold. First, we characterize the pointwise behavior of the scaling factors $\mathbf{d}$ and the scaled matrix $W$ in terms of the quantities in our setup for large $m,n$ and small $\epsilon$; see Section~\ref{sec: main results}. Second, we build on these results to infer various quantities of interest from the noisy data and provide robust normalizations analogous to~\eqref{eq: traditional graph Laplacian normalizations} with appropriate theoretical justification; see Section~\ref{sec: application to inference}. In addition, in section~\ref{sec: scRNA-seq example} we demonstrate our results on real single-cell RNA-sequencing data, and in Section~\ref{sec: discussion} we discuss future research directions. All proofs are deferred to the appendix. Below is a detailed account of our results and contributions.

In Section~\ref{sec: main results} we begin by considering the setting of fixed $\epsilon$ and large $m,n$. We establish that ${d}_i$ and $W_{i,j}$ concentrate around certain quantities that depend explicitly on the clean Gaussian kernel $\mathcal{K}_\epsilon(x_i,x_j)$, the noise magnitudes $\Vert \eta_i \Vert_2^2$, and the solution to an integral equation that is the population analogue of~\eqref{eq:discrete scaling eq}. The associated error term is described via a probabilistic bound that is explicit in $m$, $n$, and the sub-Gaussian norm of the noise; see Theorem~\ref{thm:scaling factors asym expression}. Importantly, this result allows the noise magnitudes $\Vert \eta_i \Vert_2^2$ to be large and possibly diverge (stochastically) as $m,n \rightarrow \infty$, while the probabilistic errors in $\mathbf{d}$ and $W$ are vanishing. Therefore, the doubly stochastic scaling is robust to the entry-wise perturbations of the noise in large samples and high dimension simultaneously. Next, we turn to analyze the solutions to the aforementioned integral equation (see Equation~\eqref{eq:integral scaling eq with density}) for small $\epsilon$. In particular, we prove a first-order approximation in $\epsilon$ that depends on the sampling density and the manifold geometry; see Theorem~\ref{thm: scaling function convergence}.
Overall, our results in Section~\ref{sec: main results} show that for small $\epsilon$ and sufficiently large $m$ and $n$, the noisy doubly stochastic affinity matrix $W_{i,j}$ approximates the clean Gaussian kernel $\mathcal{K}_\epsilon(x_i,x_j)$ up to a global constant and a multiplicative bias term that depends inversely on the square root of the sampling densities at $x_i$ and $x_j$, but not on the noise magnitudes $\Vert \eta_i \Vert_2^2$. This is made possible by the scaling factors $d_i$, which ``absorb'' the noise magnitudes $\Vert \eta_i \Vert_2^2$, thereby correcting the Euclidean distances in the noisy Gaussian kernel. In particular, the scaling factor $d_i$ depends exponentially on the noise magnitude $\Vert \eta_i \Vert_2^2$, and inversely on the square root of the sampling density at $x_i$; see Equation~\eqref{eq: W_ij and d_i limiting form} in Section~\ref{sec: main results}. To the best of our knowledge, these results are new even when no noise is present, as they describe the sample-to-population pointwise behavior of $W$ and the scaling factors $\mathbf{d}$.

In Section~\ref{sec: application to inference} we proceed by developing several tools for robust inference. First, we construct a robust density estimator by applying a nonlinearity to $W$ and establish its convergence to the true density up to a global constant under appropriate conditions; see Equation~\ref{eq:density estimator formula} and Theorem~\ref{thm: density estimator}. We demonstrate that this approach can provide accurate density estimates on a manifold under strong heteroskedastic noise and outliers, whereas the standard kernel density estimator $D_{i,i}$ from~\eqref{eq: traditional graph Laplacian normalizations} fails; see Figures~\ref{fig: density estimation}--\ref{fig: density estimation error vs epsilon}. Second, we show that the scaling factors $\mathbf{d}$ and our robust density estimator can be combined to recover the noise magnitudes $\Vert \eta_i \Vert_2^2$, the signal magnitudes $\Vert x_i \Vert_2^2$, and the clean Euclidean distances $\Vert {x}_i - {x}_j \Vert_2^2$ up to small perturbations; see Equations~\eqref{eq: noise magnitude estimator},~\eqref{eq: signal magnitude estimate and distance correction formulas} and Proposition~\ref{prop: noise magnitude est}. We demonstrate that these tools can be useful for detecting outliers, assessing the local quality of data, and identifying near neighbors more accurately; see Figures~\ref{fig: recovering noise magnitudes and signal magnitudes} and~\ref{fig: recovering Euclidean distances}.
Third, by utilizing our robust density estimator and the doubly stochastic matrix $W$, we provide a family of normalizations that is a robust analogue of the traditional normalizations in~\eqref{eq: traditional graph Laplacian normalizations}, and establish convergence of the corresponding graph Laplacians to the appropriate family of differential operators; see Equations~\eqref{eq: robust graph Laplacian normalizations},~\eqref{eq: differential operator T} and Theorem~\ref{thm: Laplacian estimator}. These normalizations can be used to obtain a more robust version of the Diffusion Maps method~\cite{coifman2006diffusionMaps}. In particular, we demonstrate that in the case of $\alpha=1$ and high-dimensional heteroskedastic noise, our approach provides a much more accurate approximation to the Laplace Beltrami operator than the traditional normalization of~\eqref{eq: traditional graph Laplacian normalizations}; see Figure~\ref{fig: Laplacian estimation example}. Overall, our results in Section~\ref{sec: application to inference} show that it is possible to recover the sampling density and the manifold geometry under general high-dimensional noise, even when the noise magnitudes are non-negligible and vary substantially. Importantly, our results show that this recovery is possible even when the ambient dimension grows slowly with the sample size, e.g., $m\propto n^{0.001}$.

Lastly, in Section~\ref{sec: scRNA-seq example} we exemplify the tools derived in Section~\ref{sec: application to inference} on experimental single-cell RNA-sequencing (scRNA-seq) data with cell type annotations. First, we show that our general-purpose noise magnitude estimator (derived in Section~\ref{sec: recovering noise magnitudes}) agrees well with a popular model for explaining scRNA-seq data; see Figure~\ref{fig: estimated noise vs. Poisson noise scRNAseq}. Second, we show that our robust analogue of $\hat{P}^{(\alpha)}$ from~\eqref{eq: traditional graph Laplacian normalizations} (derived in Section~\ref{sec: graph Laplacian estimation}) describes a more accurate and stable random walk behavior with respect to the ground truth cell types; see Figures~\ref{fig: average probability to leave cell type scRNAseq} and~\ref{fig: worst-case probability to leave cell type scRNAseq}. The reason for this advantage is that different cell types have different levels of technical noise, which are automatically accounted for by our proposed robust normalization.

 \begin{table}[t]
 \small
   \caption{
\label{tab:notations}
Common symbols and notation} 
 \begin{center}
 \begin{tabular}{  p{1.1cm}  p{8.5cm}   }
 \hline
  ${\mathcal{M}}$ 	& $d$-dimensional manifold embedded in $\mathbb{R}^m$  	  \\
  $d$			& Intrinsic dimension of $\mathcal{M}$ \\
  $m$		& Dimension of the ambient space \\
  $n$ 			& Number of data points 	\\
 $q(x)$			& Sampling density at $x\in{\mathcal{M}}$  \\
 $x_i$      		& Clean data points on $\mathcal{M}$		\\
 $y_i$      		& Noisy data points in $\mathbb{R}^m$		\\
 $\eta_i$      		& Noise vectors in $\mathbb{R}^m$		\\
 $\mathcal{K}_{\epsilon}(\cdot,\cdot)$      		& Gaussian kernel		\\
 $\epsilon$ 	&  Kernel bandwidth parameter 		 \\ 
 $ \rho_{\epsilon}(x)$		& Function solving the integral eq.~\eqref{eq:integral scaling eq with density} at $x\in\mathcal{M}$ \\
 $K$ 	&  Noisy $n\times n$ Gaussian kernel matrix 		 \\ 
 $W$ 	&  Noisy $n\times n$ doubly stochastic affinity matrix 		 \\ 
 $\mathbf{d}$ 	&  Vector of $n$ scaling factors solving eq.~\eqref{eq:discrete scaling eq} 		 \\ 
  $E$ 	&  Maximal sub-Gaussian norm of the noise 		 \\ 
   $\Delta_{\mathcal{M}}$ 	&  The (negative) Laplace-Beltrami operator on $\mathcal{M}$ 		 \\ 
   $d\mu(x)$ 	&  Volume form of $\mathcal{M}$ at $x\in\mathcal{M}$	 \\ 
  \hline
\end{tabular}
\end{center}
\end{table} \label{table: notation}

\section{Large sample behavior of doubly stochastic scaling under high-dimensional noise} \label{sec: main results}
We consider the setting where the clean points $x_1,\ldots,x_n$ are sampled independently from a probability measure $d\nu$ supported on a $d$-dimensional Riemannian manifold $\mathcal{M} \subset \mathbb{R}^m$. In particular, $d\nu = q(x) d\mu(x)$, where $d\mu(x)$ is the volume form of $\mathcal{M}$ at $x \in \mathcal{M}$ and $q(x)$ is a positive and continuous probability density function on $\mathcal{M}$. We further make the following assumption on $\mathcal{M}$.
\begin{assump} \label{assump: manifold}
$\mathcal{M}$ is compact, smooth, with no boundary, and satisfies
$\Vert x \Vert_2\leq 1$ for all $x\in\mathcal{M}$. 
\end{assump}

A random vector $\eta\in\mathbb{R}^{m}$ is called sub-Gaussian if  $\langle \eta , y \rangle $ is a sub-Gaussian random variable~\cite{vershynin2018high} for any $y\in\mathbb{R}^m$, where $\langle \cdot , \cdot \rangle $ is the standard scalar product. For each $x\in\mathcal{M}$, let $\eta(x) \in \mathbb{R}^m$ be a sub-Gaussian random vector with zero mean. Given the clean points $x_1,\ldots,x_n$, the noise vectors $\eta_1,\ldots,\eta_{n}$ are sampled independently from $\eta(x_1),\ldots,\eta(x_{n})$, respectively. 
Therefore, each clean point $x_i$ is first sampled independently from $\mathcal{M}$ according to the density function $q(x)$, and then each noisy observation $y_i$ is produced by~\eqref{eq: noisy data} according to the realization of the random vector $\eta(x_i)$. 
Let $\Vert \eta (x) \Vert_{\psi_2}$ be the sub-Gaussian norm of $\eta (x)$~\cite{vershynin2018high}, given by $\Vert \eta (x) \Vert_{\psi_2} = \sup_{\Vert y \Vert_2 = 1} \Vert \langle \eta(x), y\rangle \Vert_{\Psi_2}$, where $\Vert \cdot \Vert_{\Psi_2}$ in the right-hand side is the sub-Gaussian norm of a random variable~\cite{vershynin2018high}. To control the magnitude of the noise, we make the following assumption.
\begin{assump} \label{assump: noise magnitude}
$E := \max_{x\in\mathcal{M}} \Vert \eta(x) \Vert_{\psi_2} \leq C/(m^{1/4} \sqrt{\log m})$ for a constant $C>0$.
\end{assump}
For instance, Assumption~\ref{assump: noise magnitude} holds if $\eta(x)$ is a multivariate normal with covariance $\Sigma(x) \in \mathbb{R}^{m\times m}$ satisfying $\Vert \Sigma(x) \Vert_2 \leq C/(m^{1/4} \sqrt{\log m})$ for all $x\in \mathcal{M}$. Observe that this includes the special case $\Sigma(x) = I_m/\sqrt{m}$, where $I_m \in \mathbb{R}^{m\times m}$ is the identity matrix. In this case, the noise magnitude at $x\in\mathcal{M}$ is $\mathbb{E}\Vert \eta(x) \Vert_2^2 = \operatorname{Tr}\{\Sigma(x)\} = 1$, which is equal or grater to the magnitude of the clean point $\Vert x \Vert_2^2$ (by Assumption~\ref{assump: manifold}). Moreover, in the more extreme case of $\Sigma(x) = I_m/(m^{1/4} \sqrt{\log m})$, the noise magnitude is $\mathbb{E}\Vert \eta(x) \Vert_2^2 = \sqrt{m}/\log m$, which is growing in $m$ and can be much larger than $\Vert x \Vert_2^2$. Note that the distribution of the noise can vary across $x\in \mathcal{M}$. In particular, we can have regions of $\mathcal{M}$ where $\mathbb{E}\Vert \eta(x) \Vert_2^2$ is very large and others where it is very small, allowing for considerable noise heteroskedasticity. Even if $\eta(x)$ is identically distributed across $x\in\mathcal{M}$, the norm $\Vert \eta(x) \Vert_2$ is allowed to have a heavy tail that prohibits $\Vert \eta(x) \Vert_2$ from concentrating around a global constant. For instance, we can take $\eta(x)$ to be the zero vector with probability $p\in (0,1)$ and sampled uniformly from a bounded subset of $\mathbb{R}^m$ with probability $1-p$. This is a useful model for describing outliers, in which case the magnitudes $\Vert \eta(x_i) \Vert_2^2$ can fluctuate substantially over $i=1,\ldots,n$. Lastly, we emphasize that the coordinates of $\eta(x)$ need not be independent nor identically distributed. Figure~\ref{fig: prototypical noise models} in Section~\ref{sec: application to inference} provides a two-dimensional visualization of prototypical noise models covered in our setting. 

To analyze the doubly stochastic scaling in large sample size and high dimension, we require the dimension to be at least some fractional power of the sample size, that is
\begin{assump} \label{assump: dimensions}
$m\geq n^\gamma$ for a constant $\gamma > 0$.
\end{assump}
We note that the requirement $m \geq n^\gamma$ can be modified to include an arbitrary constant $c>0$, namely $m \geq c n^\gamma$. This constant was set to $1$ for simplicity.

We treat the quantities $d$, $\gamma$, $C$, $q(x)$, and the geometry of $\mathcal{M}$ (e.g., curvature, reach) as fixed and independent of $m$, $n$, $\epsilon$, and $E$, which can vary. To fix the geometry of $\mathcal{M}$ and make it independent of $m$, one can consider a manifold that is first embedded in $\mathbb{R}^r$ for fixed $r>0$, and then embedded in $\mathbb{R}^m$ for any $m>r$ via a rigid transformation (i.e., a composition of rotations, translations, and reflections). Rigid transformations preserve the pairwise Euclidean distances $\{\Vert x-y\Vert_2\}_{x,y\in\mathcal{M}}$, thereby making all geometric properties of $\mathcal{M}$ (both intrinsic and extrinsic) independent of $m$.

To state our results, we introduce the positive function $\rho_{\epsilon}:\mathcal{M}\rightarrow \mathbb{R}_+$ that solves the integral equation
\begin{equation}
    \frac{1}{(\pi \epsilon)^{d/2}}\int_\mathcal{M} \rho_{\epsilon}(x) \mathcal{K}_{\epsilon}(x,y) \rho_{\epsilon}(y) q(y) d\mu(y) = 1, \label{eq:integral scaling eq with density}
\end{equation}
for all $x\in\mathcal{M}$. The integral equation~\eqref{eq:integral scaling eq with density} and the scaling function $\rho_{\epsilon}(x)$ can be interpreted as population analogues of the scaling equation~\eqref{eq:discrete scaling eq} and the scaling factors $\mathbf{d}$, respectively. Due to the compactness of $\mathcal{M}$ and the positivity and continuity of $\mathcal{K}_{\epsilon}(x,y)$ and $q(x)$, the results in~\cite{borwein1994entropy} (see in particular Theorem 5.2 and Corollaries 4.12 and 4.19 therein) guarantee that the solution $\rho_{\epsilon}(x)$ exists and is a unique positive and continuous function on $\mathcal{M}$ (see also~\cite{knopp1968note}). Note that $\rho_{\epsilon}(x)$ depends on the kernel bandwidth parameter $\epsilon$. Table~\ref{table: notation} summarizes common symbols and notation used in the results of this section.

We now have the following theorem, which characterizes the scaled matrix $W$ and the scaling factors $\mathbf{d}$ for large sample size $n$ and high dimension $m$.
\begin{theorem} \label{thm:scaling factors asym expression}
Under Assumptions~\ref{assump: manifold},\ref{assump: noise magnitude},\ref{assump: dimensions}, there exist $t_0,m_0(\epsilon),n_0(\epsilon), C^{'}({\epsilon})>0$, 
such that for all $m>m_0(\epsilon)$ and $n>n_0(\epsilon)$, we have
\begin{align}
    W_{i,j} &= \frac{\rho_{\epsilon}(x_i) \mathcal{K}_{\epsilon}(x_i,x_j) \rho_{\epsilon}(x_j)}{(n-1)(\pi \epsilon)^{d/2}}  \left( 1 + \mathcal{E}_{i,j} \right), \label{eq: W_ij variance error} \\
    d_i &= \frac{\rho_{\epsilon}(x_i)}{\sqrt{(n-1)(\pi \epsilon)^{d/2}}} \operatorname{exp}\left( \frac{\Vert \eta_i \Vert_2^2}{\epsilon} \right) \left( 1 + \mathcal{E}_{i,i} \right), \label{eq: d_i variance error} 
\end{align}
for all $i,j=1,\ldots,n$, $i\neq j$, where $\rho_{\epsilon}(x)$ is the solution to~\eqref{eq:integral scaling eq with density}, and
\begin{equation}
    \max_{i,j} \left\vert \mathcal{E}_{i,j} \right\vert \leq {t} C^{'}({\epsilon}) \max \left\{ E\sqrt{\log m}, E^2 \sqrt{m\log m}, \sqrt{\frac{\log n}{n}}\right\}, \label{eq: variance error probabilistic bound}
\end{equation}
with probability at least $1-n^{-t}$, for any $t>t_0$.
\end{theorem}
Theorem~\ref{thm:scaling factors asym expression} provides explicit asymptotic expressions for the doubly stochastic matrix $W$ and the associated scaling factors $\mathbf{d}$ in terms of the quantities in our setup, as well as a high-probability bound on the relative pointwise errors $\mathcal{E}_{i,j}$ with explicit dependence on $m$, $n$, and $E$. We note that $t_0,m_0(\epsilon),n_0(\epsilon), C^{'}({\epsilon})$ in Theorem~\ref{thm:scaling factors asym expression} all depend on $d$, $\gamma$, $C$, $q(x)$, and on the geometry of $\mathcal{M}$, which are considered as fixed in our setup. The notation $m_0(\epsilon),n_0(\epsilon), C^{'}({\epsilon})$ means that these quantities additionally depend on $\epsilon$. 

Observe that under Assumption~\ref{assump: noise magnitude}, the quantities $E\sqrt{\log m}$ and $E^2 \sqrt{m\log m}$ appearing in~\eqref{eq: variance error probabilistic bound} always converge to zero as $m\rightarrow \infty$. Moreover, since $m$ is growing with $n$ by Assumption~\ref{assump: dimensions}, all three quantities $E\sqrt{\log m}$, $E^2 \sqrt{m\log m}$, and $\sqrt{\log n / n}$ converge to zero as $n\rightarrow \infty$. Consequently, Theorem~\ref{thm:scaling factors asym expression} implies that if we fix a bandwidth parameter $\epsilon$, then all pointwise errors $\mathcal{E}_{i,j}$ appearing in~\eqref{eq: W_ij variance error} and~\eqref{eq: d_i variance error} converge almost surely to zero as $n\rightarrow \infty$. 

According to~\eqref{eq: variance error probabilistic bound}, the convergence rate of $\max_{i,j} \left\vert \mathcal{E}_{i,j} \right\vert$ to zero is bounded by the largest among $E\sqrt{\log m}$, $E^2 \sqrt{m\log m}$, and $\sqrt{\log n / n}$. If no noise is present, i.e. $E = 0$, this rate is bounded by $\sqrt{\log n / n}$, which describes the sample-to-population convergence and is independent of the ambient dimension $m$. In fact, in this case, we do not actually require Assumptions~\ref{assump: noise magnitude} and~\ref{assump: dimensions}. On the other hand, if noise is present, then the convergence rate depends also on the maximal sub-Gaussian norm $E$ and on the ambient dimension $m$. Let us suppose for simplicity that $m \propto n$. In this case, as long as $E \leq C/\sqrt{m}$, then the convergence rate of $\max_{i,j} \left\vert \mathcal{E}_{i,j} \right\vert$ to zero is still bounded by $\sqrt{\log n / n}$. As discussed earlier, a simple example for this case is if $\eta(x)$ is multivariate normal with covariance $\Sigma(x)$ satisfying $\Vert \Sigma(x) \Vert_2 \leq C/\sqrt{m}$, which allows the magnitude of the noise $\mathbb{E}\Vert \eta(x) \Vert_2^2$ to be comparable to the signal magnitude $\Vert x \Vert_2^2$. If $m \propto n$ and $E$ decays more slowly than $1/\sqrt{m}$, then the bound on $\max_{i,j} \left\vert \mathcal{E}_{i,j} \right\vert$ becomes dominated by the term $E^2 \sqrt{m\log m}$, which converges to zero as $m,n\rightarrow \infty$ even though the noise magnitude $\mathbb{E}\Vert \eta(x) \Vert_2^2$ can possibly diverge (see the discussion following Assumption~\ref{assump: noise magnitude}). 

The proof of Theorem~\ref{thm:scaling factors asym expression} can be found in Appendix~\ref{appendix: proof of theorem for concentration of scaling factors and scaled matrix}, and relies on two main ingredients. The first is the decomposition~\sloppy$\mathcal{K}_{\epsilon}(y_i,y_j) = \operatorname{\exp}\{-\Vert y_i \Vert_2^2/\epsilon\}  \operatorname{\exp}\{2\langle y_i,y_j \rangle/\epsilon\} \operatorname{\exp}\{-\Vert y_j \Vert_2^2/\epsilon\}$, which can be viewed as diagonal scaling of the nonnegative matrix $(\operatorname{\exp}\{2\langle y_i,y_j \rangle/\epsilon\})$, together with the fact that the inner products $\langle y_i,y_j \rangle$ concentrate around their clean counterparts $\langle x_i,x_j \rangle$ for $i\neq j$ under high-dimensional sub-Gaussian noise; see Lemma~\ref{lem:noise scalar product concentration} in Appendix~\ref{appendix: supporting lemmas}. The second ingredient is a refined stability analysis of the scaling factors of a symmetric nonnegative matrix with zero main diagonal; see Lemma~\ref{lem:closeness of scaling factors} in Appendix~\ref{appendix: supporting lemmas}, which improves upon the analysis in~\cite{landa2021doubly}. These two ingredients are combined with a perturbation analysis of the Gaussian kernel and large-sample concentration arguments to prove the results in Theorem~\ref{thm:scaling factors asym expression}.

Theorem~\ref{thm:scaling factors asym expression} asserts that for sufficiently large $m$ and $n$, $W_{i,j}$ is close to the clean Gaussian kernel $\mathcal{K}_{\epsilon}(x_i,x_j)$ up to a constant factor and the multiplicative bias term $\rho_{\epsilon}(x_i) \rho_{\epsilon}(x_j)$. This bias term is determined by the scaling function $\rho_{\epsilon}(x)$ that solves~\eqref{eq:integral scaling eq with density}, which depends on the geometry of $\mathcal{M}$ and the density $q(x)$ in a non-trivial way and does not admit a closed form expression in general. Nonetheless, we can provide an explicit approximation to $\rho_{\epsilon}(x)$ when $\epsilon$ is small. To that end, we first assume that $q(x)$ is sufficiently smooth on $\mathcal{M}$, specifically
\begin{assump} \label{assump: smoothness}
$q \in \mathcal{C}^6(\mathcal{M})$.
\end{assump}
Let $p\geq 1$ and define $\Vert \cdot \Vert_{L^p(\mathcal{M},d\mu)}$ as the standard $L^p$ norm on $\mathcal{M}$ with measure $d\mu$, i.e., $\Vert f \Vert_{L^p(\mathcal{M},d\mu)} = (\int_\mathcal{M} |f_{\epsilon}(x)|^p d\mu)^{1/p}$ for any $f: \mathcal{M} \rightarrow \mathbb{R}$. In addition, we denote by $\Delta_\mathcal{M}\{f\}(x)$ the negative Laplace-Beltrami operator on $\mathcal{M}$ applied to $f$ and evaluated at $x\in\mathcal{M}$. We now have the following result.
\begin{theorem} \label{thm: scaling function convergence}
Under Assumptions~\ref{assump: manifold} and~\ref{assump: smoothness}, for any $p \in [1, 4/3)$, there exist constants $\epsilon_0, c_p>0$ and a function $\omega:\mathcal{M}\rightarrow \mathbb{R}$ that depends only on the geometry of $\mathcal{M}$, such that for all $\epsilon\leq \epsilon_0$,
\begin{equation}
    \left\Vert \rho_{\epsilon} - q^{-1/2} F_{\epsilon} \right\Vert_{L^{p}(\mathcal{M},d\mu)} \leq c_p \epsilon^2,
    \label{eq: L_p convergence}
\end{equation}
where $\rho_{\epsilon}(x)$ is the solution to~\eqref{eq:integral scaling eq with density}, and
\begin{equation}
    F_{\epsilon}(x) = 1 - \frac{\epsilon}{8}\left( \omega(x) - \frac{\Delta_\mathcal{M}\{q^{-1/2}\}(x)}{\sqrt{q(x)}}\right). \label{eq: F_eps def}
\end{equation}
If $d\leq 5$, then~\eqref{eq: L_p convergence} holds for any $p \in [1, 2]$, and moreover, there exist $\epsilon_0^{'}, c^{'}>0$, 
such that for all $\epsilon\leq \epsilon_0^{'}$ and $x\in\mathcal{M}$,
\begin{equation}
    \left\vert \rho_{\epsilon}(x) - q^{-1/2}(x) F_{\epsilon}(x) \right\vert
    \leq c^{'} \epsilon^{2-d/4}. \label{eq: pointwise convergence with d <= 2}
\end{equation}
\end{theorem}
The first part of Theorem~\ref{thm: scaling function convergence} provides an asymptotic approximation to $\rho_{\epsilon}(x)$ with an $L^p(\mathcal{M},d\mu)$ error of $\mathcal{O}(\epsilon^2)$ for $p<4/3$. This approximation is equal to $q^{-1/2}$ to zeroth-order with a first-order correction term that depends additionally on the manifold geometry and on the smoothness of the density. The second part of Theorem~\ref{thm: scaling function convergence} improves upon this result in the case of $d\leq 5$, where the convergence now is in $L^2(\mathcal{M},d\mu)$ with the same rate of $\mathcal{O}(\epsilon^2)$. Moreover, in this case, we have uniform pointwise convergence on $\mathcal{M}$ with a rate at least $\mathcal{O}(\epsilon^{2-d/4})$. If $d\leq 3$, then this result implies the first-order pointwise asymptotic approximation $\rho_{\epsilon}(x) \sim q^{-1/2}(x)F_{\epsilon}(x)$ uniformly for $x\in \mathcal{M}$. Otherwise, if $d=4$ or $d=5$,~\eqref{eq: pointwise convergence with d <= 2} only implies the zeroth-order pointwise asymptotic approximation $\rho_{\epsilon}(x)\sim q^{-1/2}(x)$ (since the error in the right-hand side of~\eqref{eq: pointwise convergence with d <= 2} becomes $\mathcal{O}(\epsilon)$ or larger). We note that the expression $q^{-1/2}F_\epsilon$ was also adopted in~\cite{cheng2022bi} to construct an approximate solution to~\eqref{eq:integral scaling eq with density}, yet our results here prove the convergence of $\rho_\epsilon$ to $q^{-1/2}F_\epsilon$, which did not appear previously. 

The proof of Theorem~\ref{thm: scaling function convergence} can be found in Appendix~\ref{appendix: proof of scaling function convergence} and is based on the following approach. First, we construct a certain covering of $\mathcal{M}$ to show that the measure of the set $\{x: \rho_{\epsilon}(x) > t\}$ is upper bounded by $c/t^2$ for some constant $c>0$ that depends only on the manifold $\mathcal{M}$ and the density $q$; see Lemma~\ref{lem: boundedness of rho} in Appendix~\ref{appendix: supporting lemmas}. Then, to establish~\eqref{eq: L_p convergence}, we make use of a technical manipulation of the integral equation~\eqref{eq:integral scaling eq with density} that relies on the aforementioned Lemma~\ref{lem: boundedness of rho}, the positive definiteness of the Gaussian kernel (as an integral operator), and the asymptotic expansion developed in~\cite{coifman2006diffusionMaps} (see also Lemma~\ref{lem: diffusion maps lemma} in Appendix~\ref{appendix: supporting lemmas}). 
In the special case of $d\leq 5$, the $L_p(\mathcal{M},d\mu)$ convergence in~\eqref{eq: L_p convergence} together with Holder's inequality allows us to refine the previous analysis and establish the remaining claims.

By combining Theorems~\ref{thm:scaling factors asym expression} and~\ref{thm: scaling function convergence}, we can describe the convergence of $d_i$ and $W_{i,j}$ to population forms that do not depend on the manifold $\mathcal{M}$ (to zeroth-order in $\epsilon$). In particular, if $d\leq 5$, we are guaranteed that in the asymptotic regime where $m,n\rightarrow \infty$ and $\epsilon \rightarrow 0$ sufficiently slowly, we have 
\begin{equation}
    d_i \sim \frac{1}{\sqrt{(n-1)(\pi \epsilon)^{d/2} q(x_i)}} \operatorname{exp}\left( \frac{\Vert \eta_i \Vert_2^2}{\epsilon} \right), \qquad\qquad 
    W_{i,j} \sim \frac{\mathcal{K}_{\epsilon}(x_i,x_j)}{(n-1)(\pi \epsilon)^{d/2}\sqrt{q(x_i) q(x_j)}}, \label{eq: W_ij and d_i limiting form}
\end{equation}
almost surely for all indices $i\neq j$. Hence, if the sampling density on $\mathcal{M}$ is uniform, i.e., $q(x)$ is a constant function, then $W_{i,j}$ approximates the clean Gaussian kernel $\mathcal{K}_{\epsilon}(x_i,x_j)$ for all $i\neq j$ up to a global constant, even if the noise magnitudes $\Vert \eta_i \Vert_2^2$ are large and fluctuate considerably. In this case, the variability of the scaling factors $d_i$ corresponds to the variability of $\Vert \eta_i \Vert_2^2$, where large values of $d_i$ correspond to strong noise, and vice versa. If the density is not uniform, then $W_{i,j}$ and $d_i$ are also affected by the variability of the density $q(x_i)$. Nonetheless, this effect can be removed by estimating the density and correcting $d_i$ and $W_{i,j}$ accordingly; see Sections~\ref{sec: robust density estimation} and~\ref{sec: recovering noise magnitudes} for more details. 

\section{Applications to inference of density and geometry} \label{sec: application to inference}
In this section, we utilize the doubly stochastic scaling~\eqref{eq:discrete scaling eq} and the results in the previous section to infer various quantities of interest from the noisy data. All numerical experiments described in this section use the scaling algorithm of~\cite{wormell2021spectral} to solve~\eqref{eq:discrete scaling eq} with a tolerance of $10^{-9}$.
To simplify the analysis and statements of the results presented in this section, we work under the following assumption that extends the pointwise first-order convergence of $\rho_{\epsilon}(x)$ in Theorem~\ref{thm: scaling function convergence} to arbitrary intrinsic dimension $d$; see Remark~\ref{remark: assumption on pointwise convergence of rho} below.
\begin{assump} \label{assump: pointwise convergence}
There exist $\beta \in (0,1]$, $\epsilon_0^{'}>0$, and $c^{'}>0$, 
such that for all $\epsilon\leq \epsilon_0^{'}$ and $x\in\mathcal{M}$, $\left\vert \rho_\epsilon(x) - q^{-1/2}(x) F_{\epsilon}(x) \right\vert
    \leq c^{'} \epsilon^{1+\beta}$.
\end{assump}

\begin{remark} \label{remark: assumption on pointwise convergence of rho}
Assumption~\ref{assump: pointwise convergence} requires that $\rho_{\epsilon}(x)$ (the solution to~\eqref{eq:integral scaling eq with density}) is approximated by $q^{-1/2}(x) F_{\epsilon}(x)$ uniformly on $\mathcal{M}$ with an error of $\mathcal{O}(\epsilon^{1+\beta})$ for some $\beta\in [0,1)$. According to Theorem~\ref{thm: scaling function convergence}, Assumption~\ref{assump: pointwise convergence} is immediately satisfied for any $d\leq 3$ (with $\beta = 1 - d/4$) under Assumptions~\ref{assump: manifold} and~\ref{assump: smoothness}. We conjecture that this property also holds in more general settings and for higher intrinsic dimensions, currently not covered by Theorem~\ref{thm: scaling function convergence}. Therefore, we rely on Assumption~\ref{assump: pointwise convergence} to simplify the presentation of our results in this section and state them in more generality for arbitrary intrinsic dimensions. We note that all numerical examples in this section were conducted in settings with $d=1$ that satisfy Assumption~\ref{assump: pointwise convergence}.
\end{remark}

\subsection{Robust manifold density estimation} \label{sec: robust density estimation}
Since the asymptotic expression of the doubly stochastic kernel $W_{i,j}$ in~\eqref{eq: W_ij and d_i limiting form} is invariant to the noise magnitudes $\Vert \eta_i \Vert_2^2$, it is natural to employ $W_{i,j}$ to infer the probability density $q(x_i)$. Recall that the standard Kernel Density Estimator (KDE) using the Gaussian kernel at $x_i$ is given by $D_{i,i}/(n-1) = \sum_{j=1, \; j\neq i}^n \mathcal{K}_{\epsilon}(x_i,x_j) / (n-1)$, which approximates $(\pi \epsilon)^{d/2} q(x_i)$ asymptotically for large $n$ and small $\epsilon$ (see~\cite{wu2022strong} and references therein). Clearly, we cannot directly replace the Gaussian kernel in the KDE with $W$ since $\sum_{j=1}^n W_{i,j} = 1 $. Instead, we propose to employ the nonlinearity $\sum_{j=1}^n [W_{i,j}]^s$ for $s > 0,\; s\neq 1$, where $[W_{i,j}]^s$ is the $s$'th power of $W_{i,j}$. Specifically, we define the {Doubly Stochastic Kernel Density Estimator (DS-KDE)} as
\begin{equation}
    \hat{q}_i = \frac{1}{n-1}\left(\sum_{j = 1}^{n}  [W_{i,j}]^s \right)^{{1}/{(1-s)}}, \label{eq:density estimator formula}
\end{equation}
for $i=1,\ldots,n$. We now have the following result.
\begin{theorem} \label{thm: density estimator} Fix $s>0$, $s\neq 1$. Under Assumptions~\ref{assump: manifold}--\ref{assump: pointwise convergence}, there exist $\epsilon_0,t_0,m_0(\epsilon),n_0(\epsilon), C^{'}({\epsilon})>0$, such that for all $\epsilon < \epsilon_0$, $m>m_0(\epsilon)$, $n>n_0(\epsilon)$, we have
\begin{equation} 
    \hat{q}_i =  ( {\pi \epsilon} )^{d/2}  s^{{d}/{(2(s-1))}} q(x_i) \left[ 1 + \mathcal{O}(\epsilon) +  \mathcal{E}^{(1)}_{i} \right], \label{eq: robust density estimator bias and variance errors}
\end{equation}
for all $i=1,\ldots,n$, where $\max_{i}\vert \mathcal{E}^{(1)}_{i} \vert$ is upper bounded by the right-hand side of~\eqref{eq: variance error probabilistic bound} with probability at least $1-n^{-t}$, for any $t>t_0$. 
\end{theorem}
We note that the quantities $\epsilon_0,t_0,m_0(\epsilon),n_0(\epsilon), C^{'}({\epsilon})$ appearing in Theorem~\ref{thm: density estimator} need not be the same as those in Theorem~\ref{thm:scaling factors asym expression}, and may additionally depend on $s$ and $\beta$, which are considered as fixed constants independent of $m$, $n$, $E$, and $\epsilon$. The proof of Theorem~\ref{thm: density estimator} can be found in Appendix~\ref{appendix: proof of robust density estimator}. 

Theorem~\ref{thm: density estimator} establishes that up to the constant factor $( {\pi \epsilon} )^{d/2}  s^{{d}/{(2(s-1))}}$, the DS-KDE $\hat{q}_i$ approximates the density $q(x_i)$ for all $i=1,\ldots,n$ with a bias error of $\mathcal{O}(\epsilon)$ and a variance error $\mathcal{E}^{(1)}_{i}$ that has the same behavior as $\mathcal{E}_{i,j}$ in~\eqref{eq: variance error probabilistic bound}. In particular, for sufficiently small $\epsilon$ and sufficiently large $m$ and $n$ (which depend also on $\epsilon$), the quantity $\hat{q}_i$ can approximate $( {\pi \epsilon} )^{d/2}  s^{{d}/{(2(s-1))}} q(x_i)$ with high probability up to an arbitrarily small relative error. Therefore, $\hat{q}_i$ can serve as a density estimator that is robust to the high-dimensional noise in our setup. 

We now demonstrate the advantage of the DS-KDE over the standard KDE via a toy example. We simulated $n=2000$ points from the unit circle in $\mathbb{R}^2$ and embedded them in $\mathbb{R}^m$ with $m=2000$ by applying a random orthogonal transformation. The angle of each clean point $x_i$, denoted by $\theta_i \in [0,2\pi)$, was sampled from $\mathcal{N}(0,0.16\pi^2)$ modulo $2\pi$, where $\mathcal{N}(\mu,\sigma^2)$ is the standard univariate normal distribution. The resulting sampling density $q(x)$ on the circle can be seen in Figure~\ref{fig: density est clean data}. We also depict the outputs of the standard KDE and the DS-KDE with $s=2$ and $\epsilon = 0.1$. It is evident that without noise, both estimators provide similarly accurate estimates of $q(x_i)$, noting that we normalized the standard KDE by $(\pi\epsilon)^{d/2}$ and the DS-KDE by  $({\pi \epsilon} )^{d/2}  s^{{d}/{(2(s-1))}}$.

\begin{figure} 
  \centering
  	{
  	\subfloat[Clean data] 
  	{
    \includegraphics[width=0.29\textwidth]{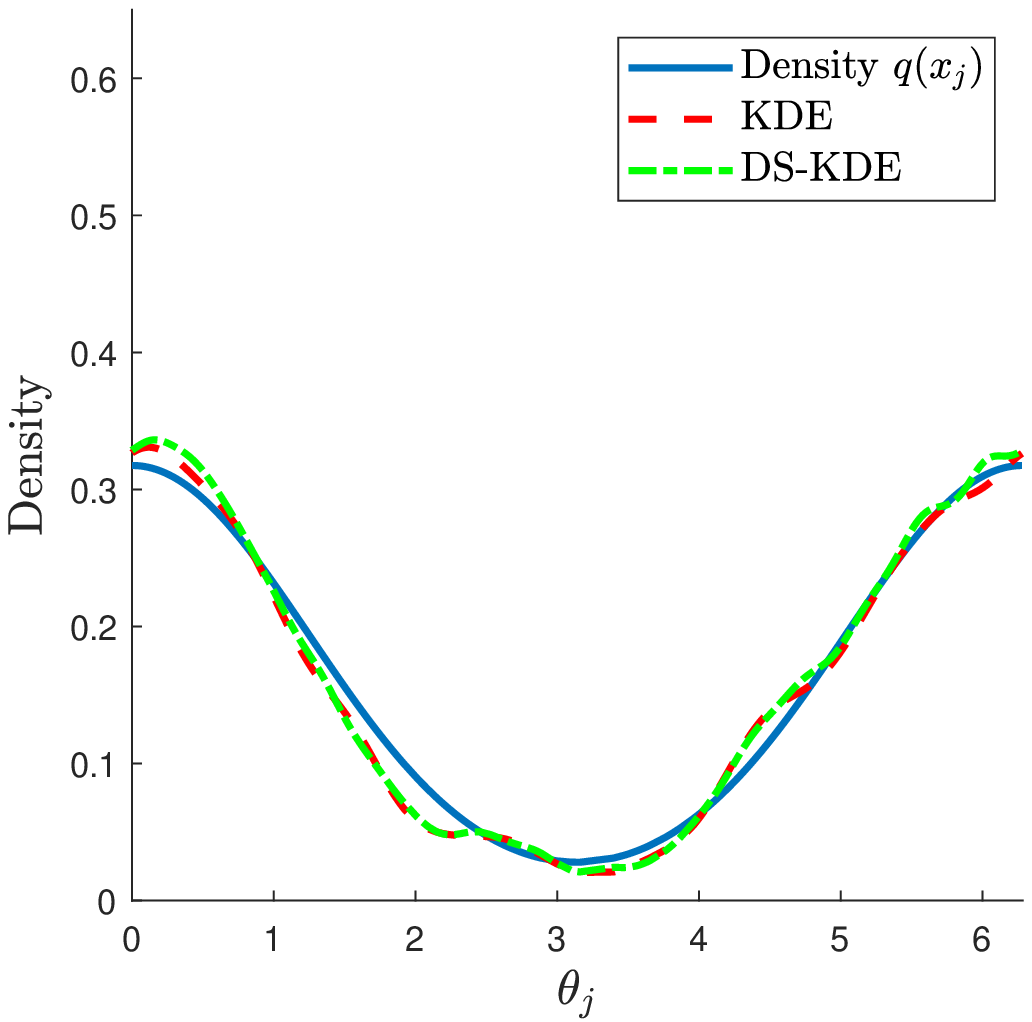}  \label{fig: density est clean data}
    } 
    \hspace{15pt}
    \subfloat[Varying noise]  
  	{
    \includegraphics[width=0.29\textwidth]{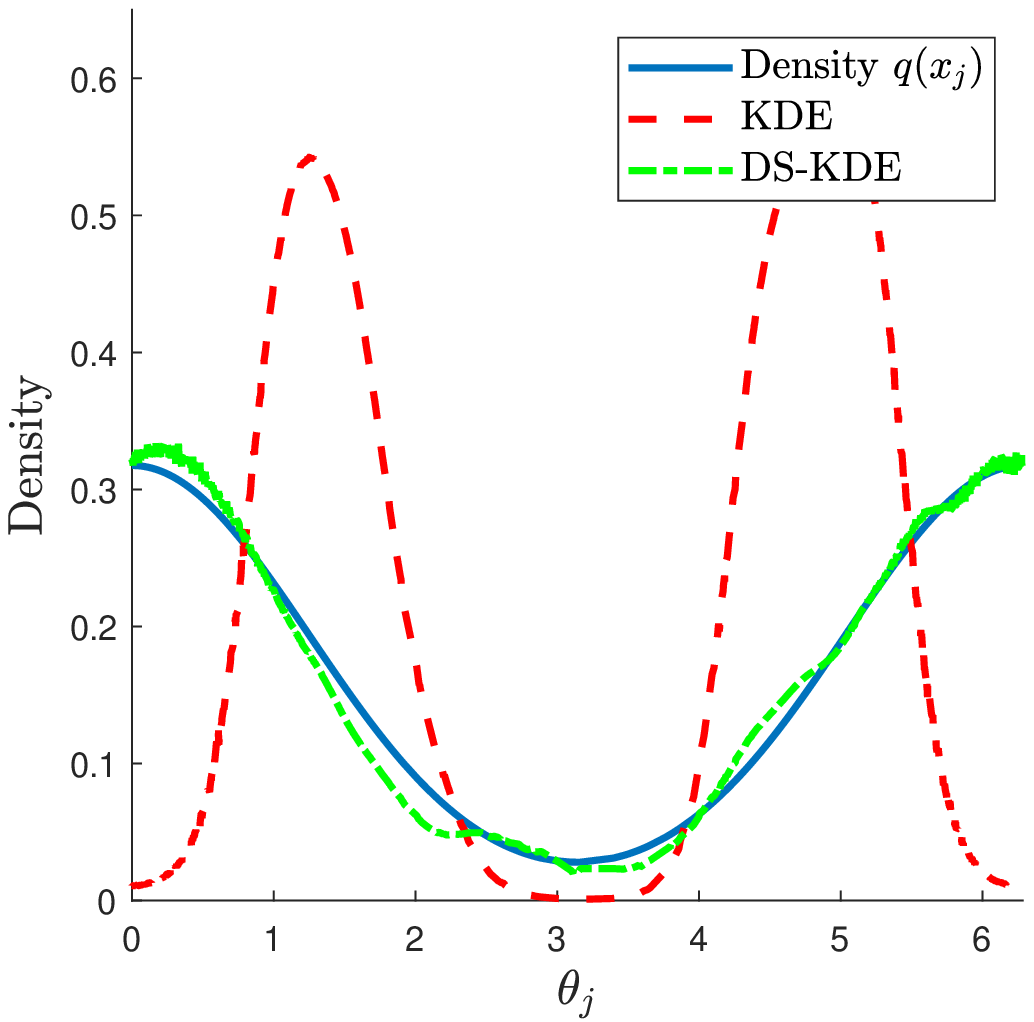} \label{fig: density est example varying noise}
    } 
    \hspace{15pt}
    \subfloat[Outlier-type noise]  
  	{
    \includegraphics[width=0.29\textwidth]{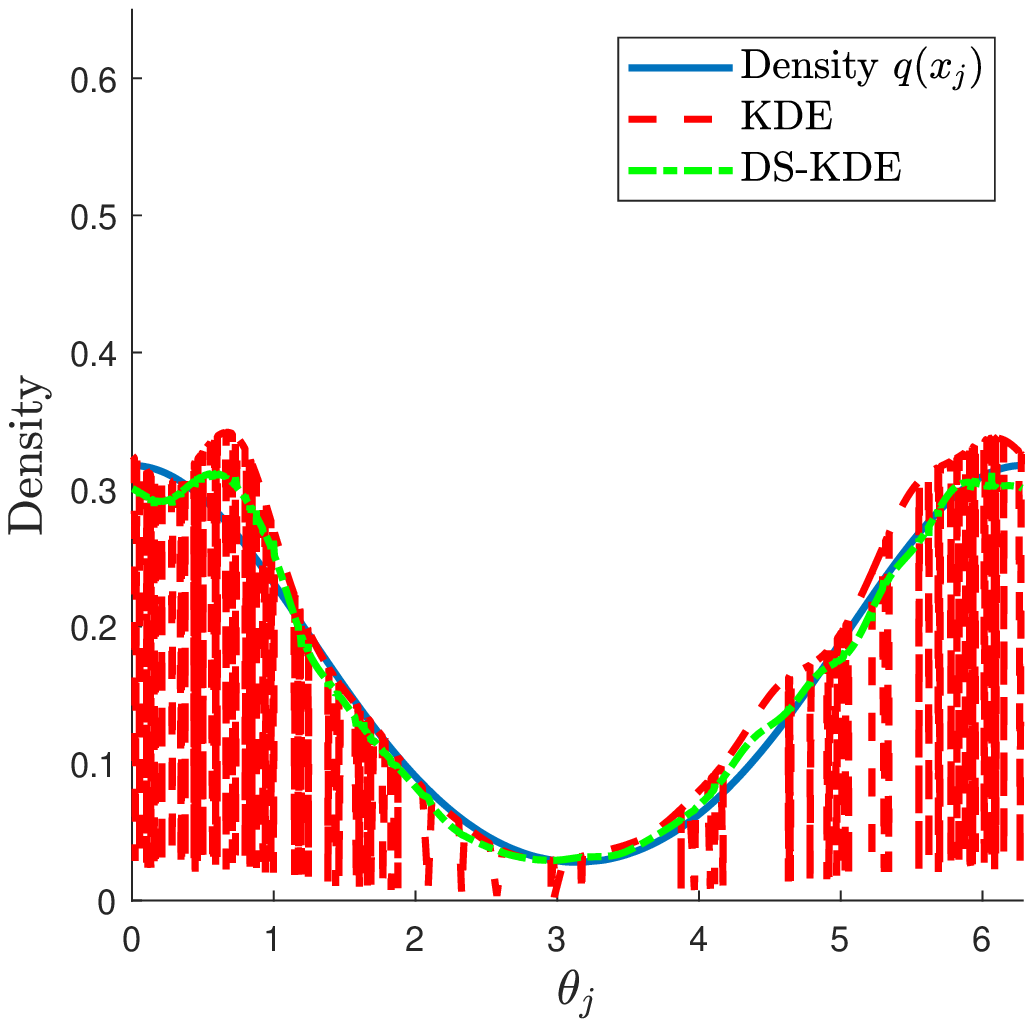} \label{fig: density est example outlier noise}
    }
    }
    \caption
    {Example of density estimation on the unit circle using DS-KDE from~\eqref{eq:density estimator formula} versus the standard KDE in clean and noisy scenarios, where $n=m=2000$, $s=2$, and $\epsilon=0.1$. The angles $\theta_i\in [0,2\pi)$ were sampled from $\mathcal{N}(0,0.16\pi^2)$ modulo $2\pi$. \textbf{Panel (a)}: Clean data.  \textbf{Panel (b)}: Heteroskedastic noise with smoothly varying magnitude; see Figure~\ref{fig: circle with varying noise}. \textbf{panel (c)}: Identically-distributed outlier type noise; see Figure~\ref{fig: circle with outlier noise}. 
    } \label{fig: density estimation}
    \end{figure} 
    
Next, we simulated two types of high-dimensional noise. First, we added noise $\eta_i$ sampled uniformly from a ball in $\mathbb{R}^m$ with radius $0.01 + 0.49(1+\cos (\theta_i))/2$, where $\theta_i$ is the angle of $x_i$ on the unit circle. Hence, the expected noise magnitude varies smoothly between $0.01$ and $0.5$ along the circle; see Figure~\ref{fig: circle with varying noise} for a two-dimensional visualization. Figure~\ref{fig: density est example varying noise} depicts the standard KDE as well as our robust density estimator $\hat{q}_i$ versus the true density $q(x_i)$. We observe that the standard KDE produces an estimate that is very different from the true density $q(x_i)$, and has more to do with the noise magnitudes $\Vert \eta_i \Vert_2^2$ in the data. On the other hand, the DS-KDE is robust to the magnitudes of the noise and produces an estimate that is nearly as accurate as in the clean case. For the second type of noise, we took each $\eta_i$ to be the zero vector with probability $p = 0.9$ and sampled it from a multivariate normal with covariance $I_m/(4m)$ with probability $1-p = 0.1$, thereby simulating identically-distributed outlier-type noise; see Figure~\ref{fig: circle with outlier noise} for a two-dimensional visualization. Figure~\ref{fig: density est example outlier noise} shows that in this case, the standard KDE suffers from pointwise drops in the estimated density. Essentially, these drops stem from the nonzero realizations of the noise, i.e., the ```outliers'', whose large noise magnitudes inflate the pairwise Euclidean distances. On the other hand, the DS-KDE produces an estimate that is invariant to the outliers and is very close to $q(x_i)$.

    \begin{figure} 
  \centering
  	{
  	\subfloat[Noise with smoothly varying magnitude]  
  	{
    \includegraphics[width=0.33\textwidth]{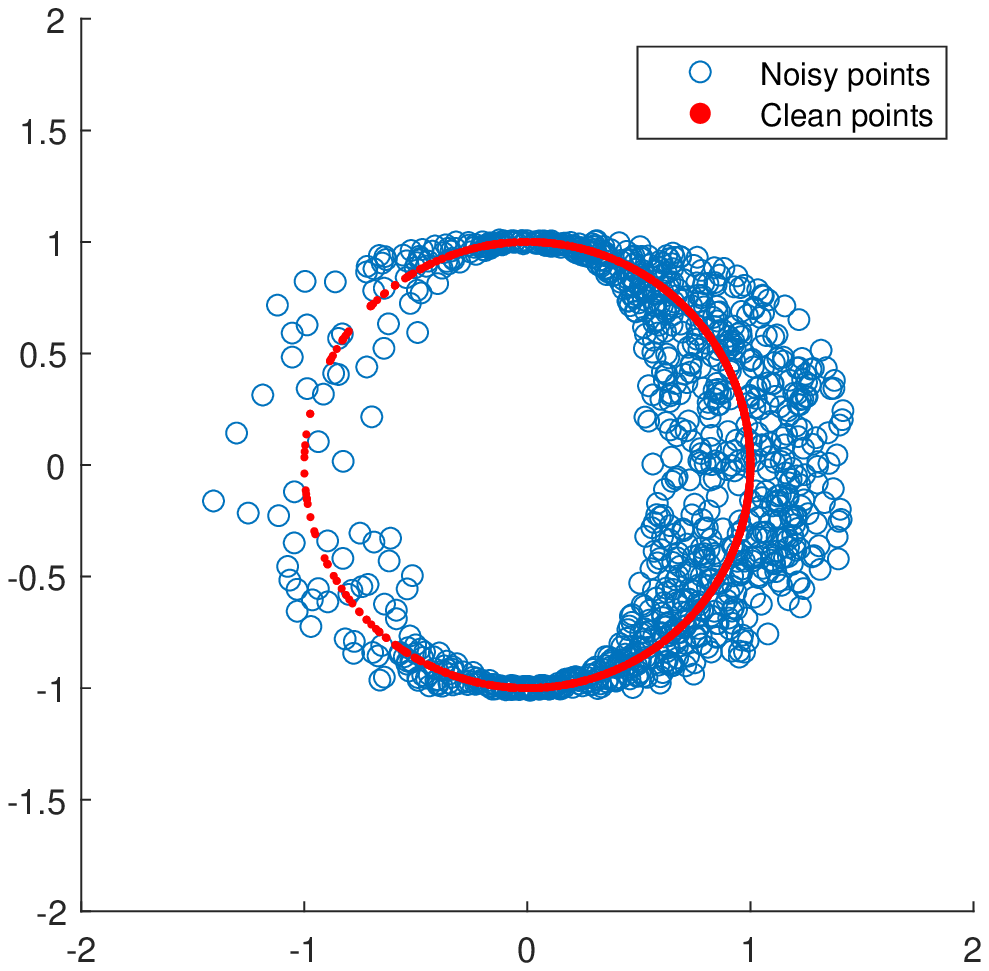} \label{fig: circle with varying noise}
    }
    \hspace{40pt}
    \subfloat[Outlier type noise]  
  	{
    \includegraphics[width=0.33\textwidth]{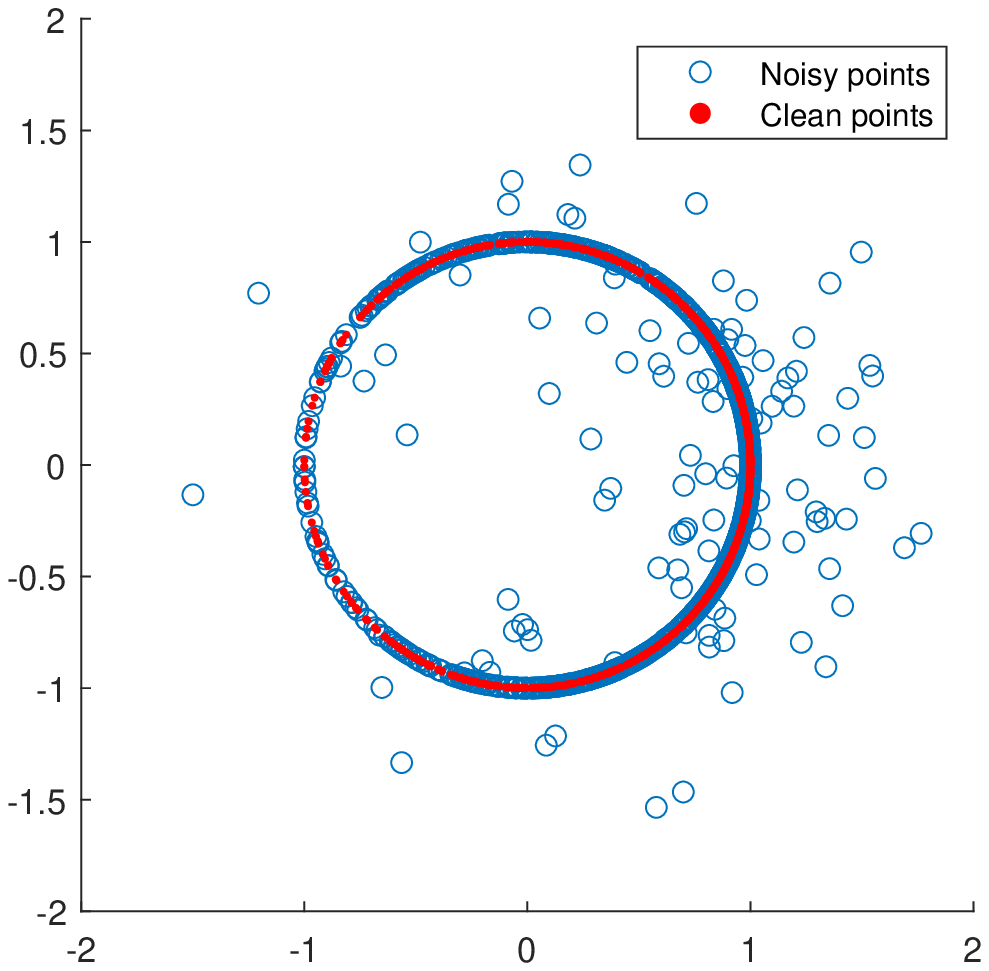} \label{fig: circle with outlier noise}
    }
    }
    \caption
    {Two-dimensional visualization of prototypical noise models used in our experiments. The clean data $x_i$ are sampled from the unit circle with non-uniform density as in Figure~\ref{fig: density est clean data}. \textbf{Panel (a)}: Heteroskedastic noise with smoothly varying magnitude, where $\eta_i$ is sampled uniformly from a ball in $\mathbb{R}^m$ with radius $0.01 + 0.49(1+\cos (\theta_i))/2$ ($\theta_i$ is the angle of $x_i$ on the unit circle). \textbf{Panel (b)}: Identically-distributed outlier-type noise, where $\eta_i$ is zero with probability $0.9$ and sampled from a multivariate normal with covariance $I_m/(4m)$ with probability $0.1$. 
    } \label{fig: prototypical noise models}
    \end{figure} 

It is interesting to point out that although the DS-KDE is undefined when $s=1$, the limiting case of $s\rightarrow 1$ is interpretable and can be implemented. In particular, according to~\eqref{eq:density estimator formula}, a direct calculation shows that
\begin{equation}
    \lim_{s\rightarrow 1} \hat{q}_i = \frac{1}{n-1}\operatorname{exp}\left\{ -\sum_{j=1,\; j\neq i}^n W_{i,j} \log(W_{i,j}) \right\}. \label{eq: density estimator s tends to 1}
\end{equation}
The right-hand side of~\eqref{eq: density estimator s tends to 1}, up to the factor $1/(n-1)$, is known as the {perplexity} of the $i$th row of $W$, where the expression inside the exponent in~\eqref{eq: density estimator s tends to 1} is the entropy. According to Theorem~\ref{thm: density estimator}, we expect the right-hand side of~\eqref{eq: density estimator s tends to 1} to approximate $( {\pi \epsilon} )^{d/2}  s^{{d}/{(2(s-1))}} q(x_i) \rightarrow (\pi e \epsilon)^{d/2} q(x_i)$ as $s\rightarrow 1$, which provides an explicit relation between the entropy of each row of the doubly stochastic kernel $W$ and the sampled density $q(x_i)$. 
We mention that Theorem~\ref{thm: density estimator} does not strictly cover the limit $s\rightarrow 1$ since the dependence of the bias and variance errors on $s$ is harder to track and is not made explicit. However, the numerical experiments described below suggest that the conclusions of Theorem~\ref{thm: density estimator} also hold for $s\rightarrow 1$, and that the performance of the density estimator in this case is comparable to other choices of $s$ over a range of bandwidth parameters $\epsilon$.

Figure~\ref{fig: density estimation error vs n} illustrates the maximal density estimation errors (over $i=1,\ldots,n$) for the standard KDE as well as the DS-KDE as functions of $n$, for $m=n$ and $s = 0.5$, $s=2$, and $s\rightarrow 1$, where $\epsilon= 0.1$. We used the same noise settings as for Figure~\ref{fig: prototypical noise models}, and the displayed errors were averaged over $50$ randomized trials. In the clean case, the KDE and the DS-KDE perform similarly, where all errors decrease with $n$ at a rate close to $n^{-1/2}$, which agrees with Theorem~\ref{thm: density estimator} and~\ref{eq: variance error probabilistic bound} up to a logarithmic factor. In both noisy cases, however, the KDE error saturates at a high level and does not decrease further, whereas the DS-KDE errors decrease roughly at the same rate as in the clean case. In particular, the DS-KDE errors for $n = 3000$ are over an order of magnitude smaller than the standard KDE error. 

\begin{figure} 
  \centering
  	{
  	\subfloat[Clean data]  
  	{
    \includegraphics[width=0.29\textwidth]{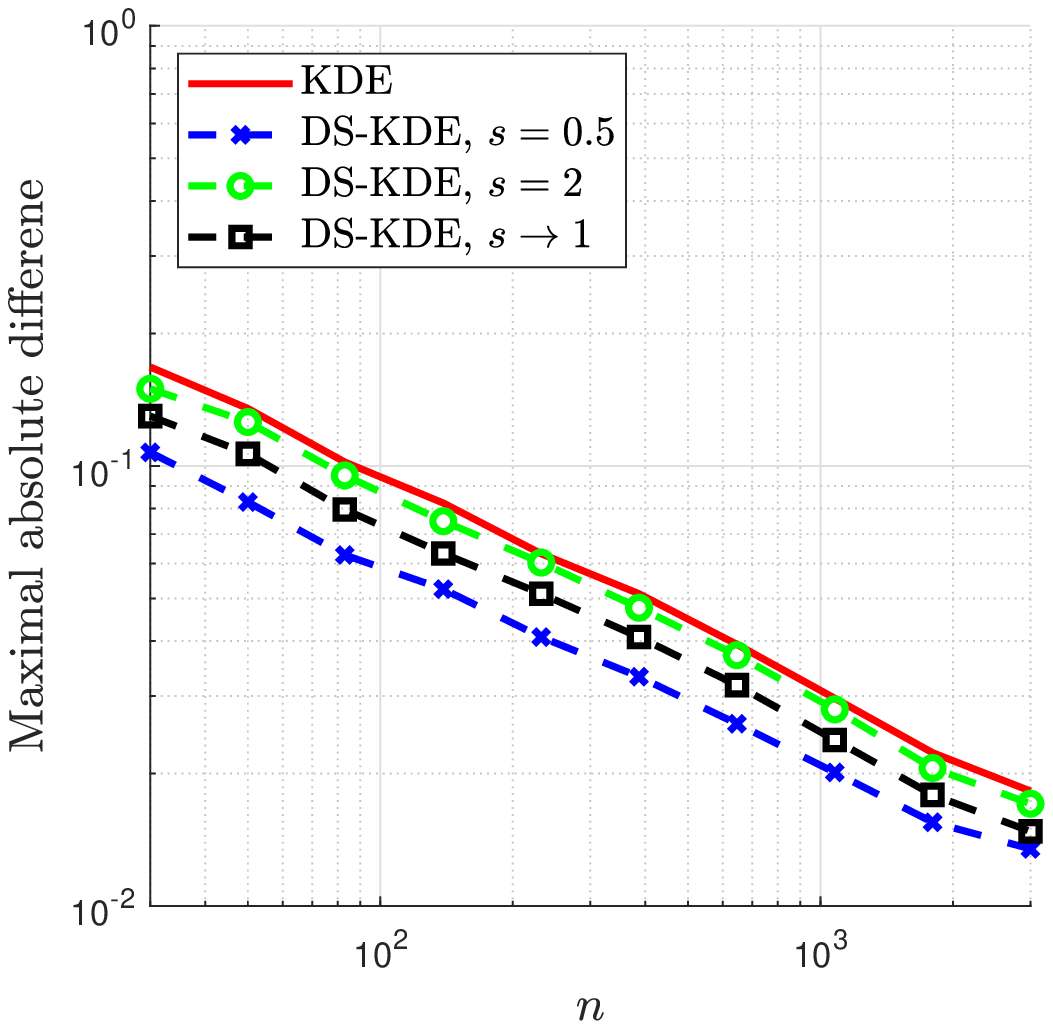} \label{fig: density est vs n clean}
    }
    \hspace{15pt}
    \subfloat[Varying noise]  
  	{
    \includegraphics[width=0.29\textwidth]{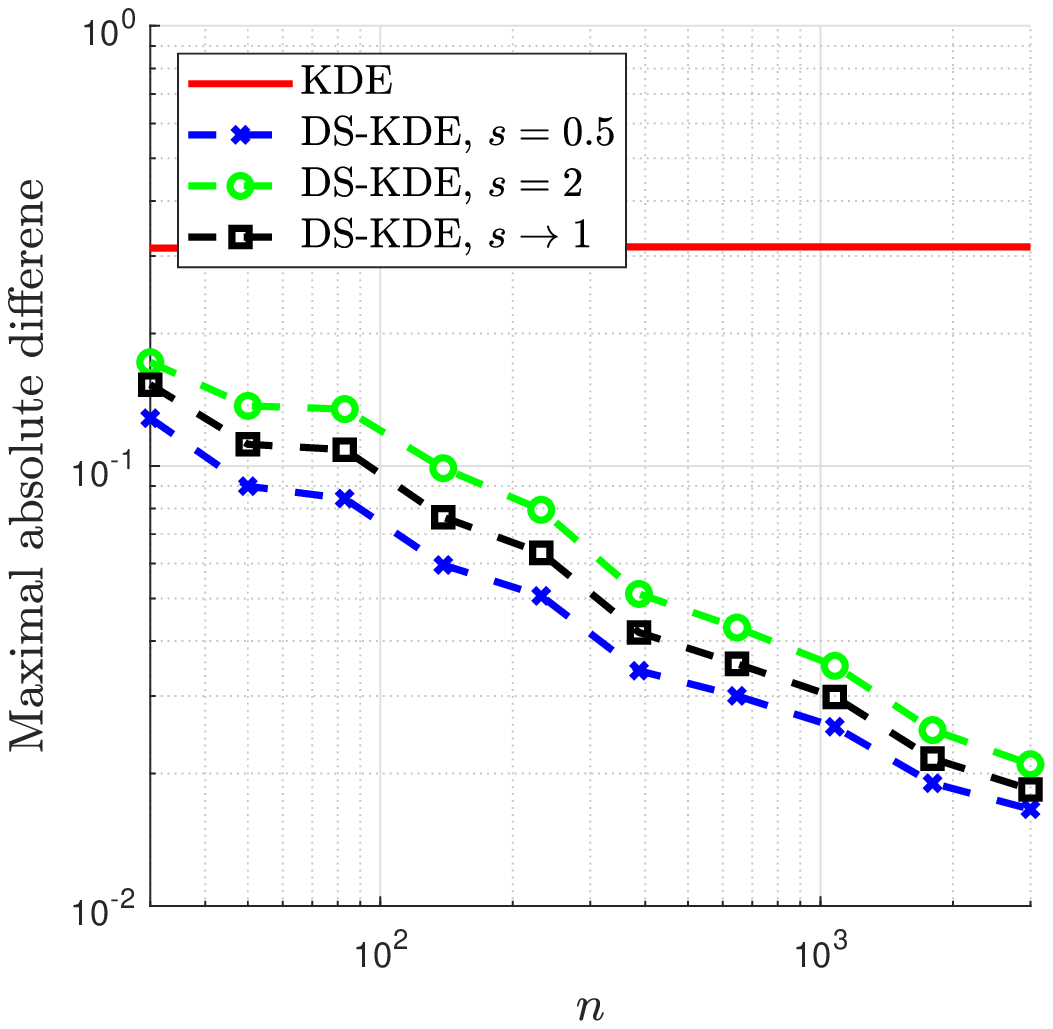} \label{fig: density est vs n varying noise}
    }
    \hspace{15pt}
    \subfloat[Outlier-type noise]  
  	{
    \includegraphics[width=0.29\textwidth]{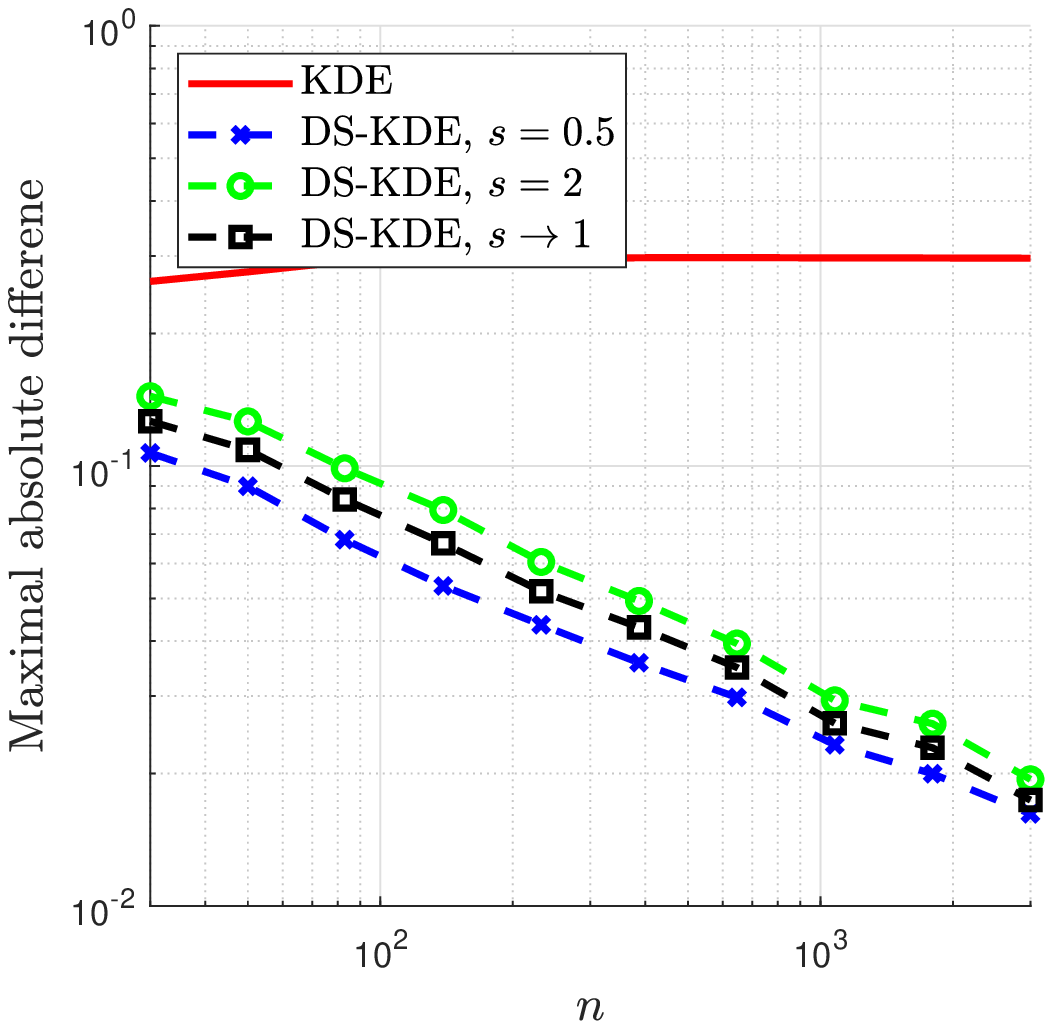} \label{fig: density est vs n outlier noise}
    }
    }
    \caption
    {Density estimation errors for the standard KDE and the DS-KDE from~\eqref{eq:density estimator formula} (normalized by the corresponding global constant) versus the number of samples $n$ and several values of $s$ for clean and noisy scenarios, where $m=n$. The density $q(x)$ and noise models are the same as for figures~\ref{fig: density est clean data},~\ref{fig: density est example varying noise}, and~\ref{fig: density est example outlier noise}, respectively; see also figures~\ref{fig: circle with varying noise} and~\ref{fig: circle with outlier noise}.
    } \label{fig: density estimation error vs n}
    \end{figure} 
    
In Figure~\ref{fig: density estimation error vs epsilon}, we depict the maximal density estimation errors for the standard KDE and DS-KDE~\eqref{eq:density estimator formula} as functions of $\epsilon$ for $s = 0.5$, $s=2$, and $s\rightarrow 1$, where $m=n=2000$. We used the same noise settings as for Figure~\ref{fig: prototypical noise models}, and the displayed errors were averaged over $10$ randomized trials.
We observe that in the clean case, all density estimators perform similarly well, attaining errors of about $0.02$ for the best values of $\epsilon$, with a small advantage to the DS-KDE with $s=2$. Yet, in the noisy scenarios, the standard KDE can only achieve an error of about $0.1$, which requires using a large bandwidth parameter, while the DS-KDE behaves similarly to the clean case and achieves significantly smaller errors. As expected from~Theorem~\ref{thm: density estimator}, we see the prototypical bias-variance trade-off in all noise scenarios, where the error of the DS-KDE is dominated by the bias term $\mathcal{O}(\epsilon)$ for large $\epsilon$, and dominated by the variance error $\max_i \vert \mathcal{E}_i^{(1)} \vert$ for small $\epsilon$. However, while the strong noise forces the standard KDE to use a large bandwidth $\epsilon$ (proportional to the magnitude of the noise) to achieve the smallest error in the bias-variance trade-off, the DS-KDE does not suffer from this issue and achieves small errors even when the bandwidth is much smaller than the noise magnitudes.

    \begin{figure} 
  \centering
  	{
  	\subfloat[Clean data]  
  	{
    \includegraphics[width=0.29\textwidth]{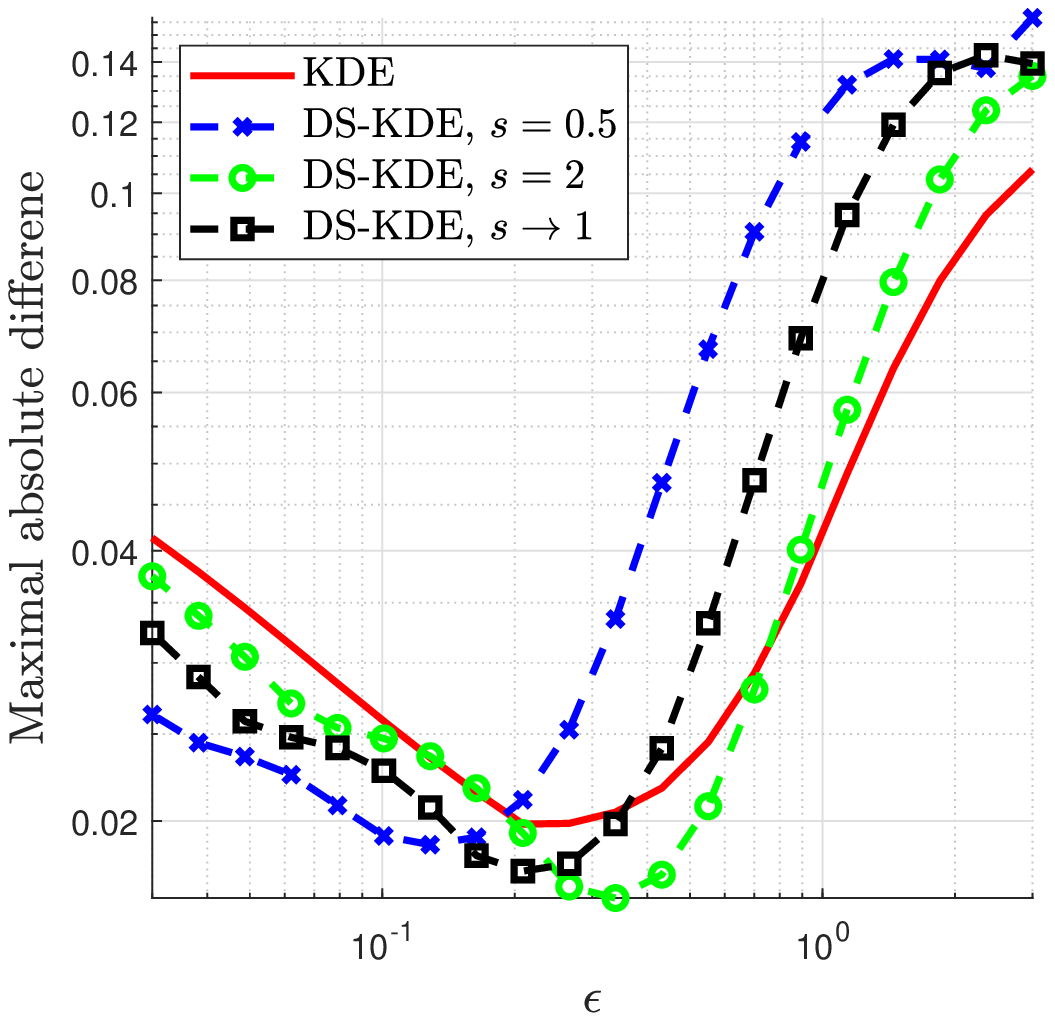} \label{fig: density est vs eps clean}
    }
    \hspace{15pt}
    \subfloat[Varying noise]  
  	{
    \includegraphics[width=0.29\textwidth]{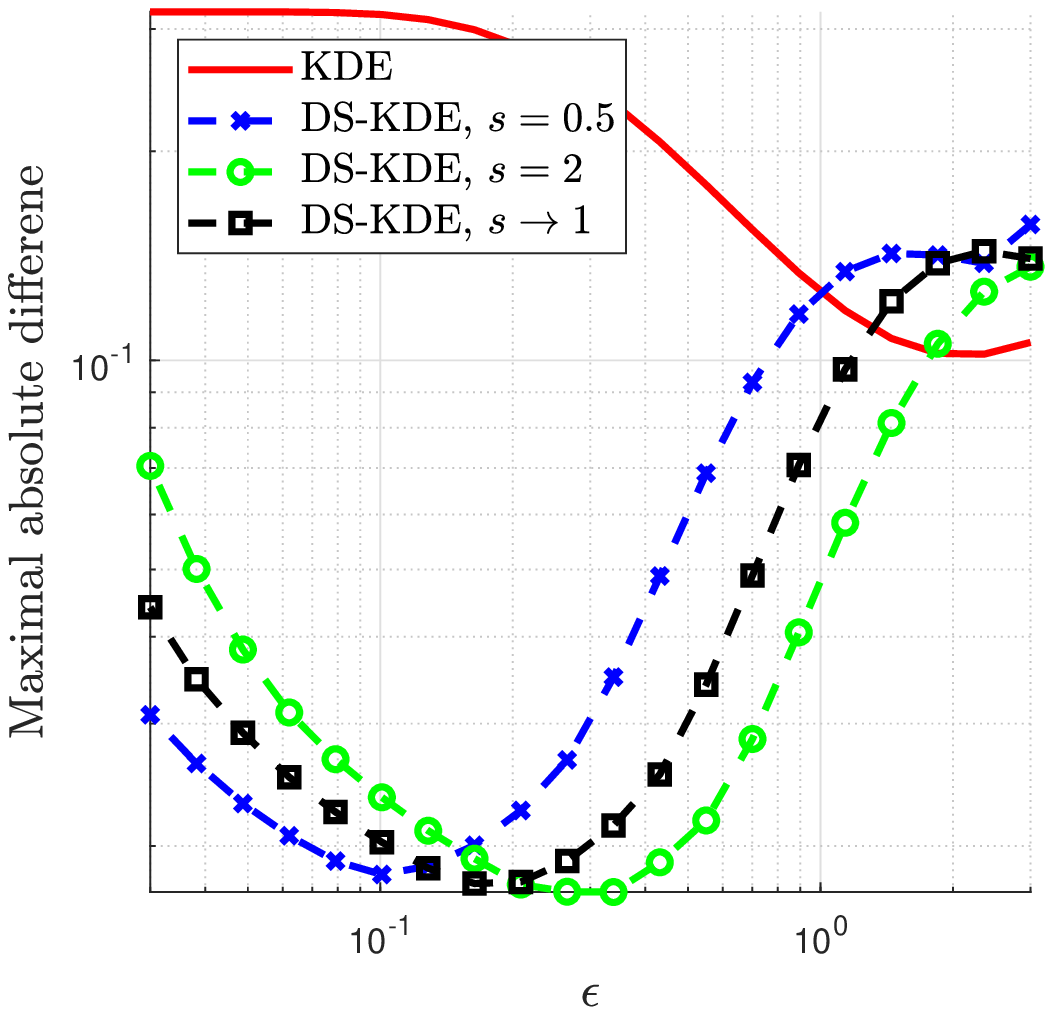} \label{fig: density est vs eps varying noise}
    }
    \hspace{15pt}
    \subfloat[Outlier-type noise]  
  	{
    \includegraphics[width=0.29\textwidth]{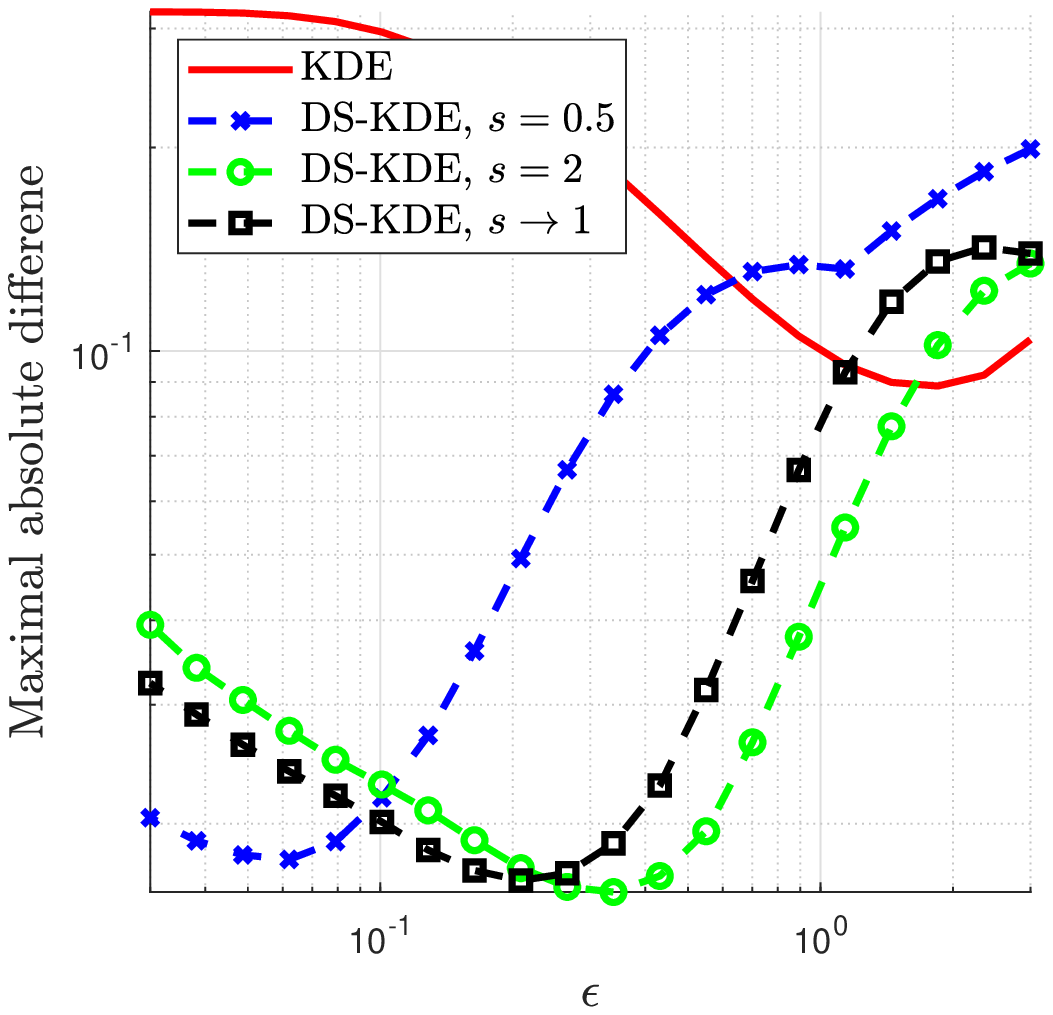} \label{fig: density est vs eps outlier noise}
    }
    }
    \caption
    {Density estimation errors for the standard KDE and the DS-KDE from~\eqref{eq:density estimator formula} (normalized by the corresponding global constant) versus the bandwidth parameter $\epsilon$ and several values of $s$ for clean and noisy scenarios, where $m=n=2000$. The density $q(x)$ and noise models are the same as for figures~\ref{fig: density est clean data},~\ref{fig: density est example varying noise}, and~\ref{fig: density est example outlier noise}, respectively; see also figures~\ref{fig: circle with varying noise} and~\ref{fig: circle with outlier noise}. 
    } \label{fig: density estimation error vs epsilon}
    \end{figure} 

\subsection{Recovering noise magnitudes, signal magnitudes, and Euclidean distances} \label{sec: recovering noise magnitudes}
According to the asymptotic expression for the scaling factors $d_i$ in~\eqref{eq: W_ij and d_i limiting form}, we can extract the noise magnitudes $\Vert \eta_i \Vert_2^2$ from $d_i$ (up to a global constant) if we know the density $q(x_i)$. Since we do not have access to $q(x_i)$ directly, we replace it with its estimate $\hat{q}_i$ from~\eqref{eq:density estimator formula} and define
\begin{align}
    \hat{N}_i = \epsilon \log \left( d_i \sqrt{(n-1)\hat{q}_i} \right),  \label{eq: noise magnitude estimator}
\end{align}
for $i=1,\ldots,n$, which serves as an estimator for the noise magnitude $\Vert \eta_i \Vert_2^2$. 
In our setup of high-dimensional noise (Assumptions~\ref{assump: noise magnitude} and~\ref{assump: dimensions}), we have $\Vert y_i \Vert_2^2 \approx \Vert {x}_i \Vert_2^2 + \Vert \eta_i \Vert_2^2$ and $\Vert y_i - y_j \Vert_2^2 \approx \Vert x_i - x_j \Vert_2^2 + \Vert \eta_i \Vert_2^2 + \Vert \eta_j \Vert_2^2$; see Lemma~\ref{lem:noise scalar product concentration} and the proof of Theorem~\ref{thm:scaling factors asym expression}. Hence, we can infer the signal magnitudes  $\Vert x_i \Vert_2^2$ and the pairwise Euclidean distances $\Vert x_i - x_j \Vert_2^2$ according to
\begin{equation}
    \hat{S}_i = \Vert y_i\Vert_2^2 - \hat{N}_i, \qquad\qquad
    \hat{D}_{i,j} = \Vert y_i - y_j \Vert_2^2 - \hat{N}_i - \hat{N}_j, \label{eq: signal magnitude estimate and distance correction formulas}
\end{equation}
respectively, for $i,j=1,\ldots,n$ with $j\neq i$. Equivalently, $\hat{D}_{i,j}$ from~\eqref{eq: signal magnitude estimate and distance correction formulas} can be derived directly from $W_{i,j}$ by canceling-out the term ${(q(x_i) q(x_j))^{-1/2}}$ appearing in~\eqref{eq: W_ij and d_i limiting form} via the density estimator $\hat{q}_i$, that is,  
\begin{equation}
    \hat{D}_{i,j} = -\epsilon \log \left\{(n-1) \sqrt{\hat{q}_i} W_{i,j} \sqrt{\hat{q}_j} \right\}.
\end{equation}
Therefore, the corrected distances $\hat{D}_{i,j}$ correspond to the similarities measured by the affinity matrix $\sqrt{\hat{q}_i} W_{i,j} \sqrt{\hat{q}_j}$, which approximates the clean Gaussian kernel (up to a global constant) according to Theorem~\ref{thm: density estimator} and the results in Section~\ref{sec: main results}.

We now have the following result, whose proof can be found in Appendix~\ref{appenidx: proof of noise magnitude estimation}. 
\begin{proposition} \label{prop: noise magnitude est}
Under Assumptions~\ref{assump: manifold}--\ref{assump: pointwise convergence}, there exist $\epsilon_0,t_0,m_0(\epsilon),n_0(\epsilon), C^{'}({\epsilon})>0$, such that for all $\epsilon < \epsilon_0$,  $m>m_0(\epsilon)$, $n>n_0(\epsilon)$, we have
\begin{align}
    \hat{N}_i &= \Vert \eta_i \Vert_2^2 + \epsilon\frac{d \log(s)}{4(s-1)} + \mathcal{O}(\epsilon^{2}) + \mathcal{E}^{(2)}_{i}, \label{eq: noise estimator bias and variance error} \\
    \hat{S}_i &= \Vert x_i \Vert_2^2 - \epsilon\frac{d \log(s)}{4(s-1)} + \mathcal{O}(\epsilon^{2}) + \mathcal{E}^{(3)}_{i}, \label{eq: signal estimator bias and variance error} \\
    \hat{D}_{i,j} &= \Vert x_i - x_j \Vert_2^2 - \epsilon \frac{d \log(s)}{2(s-1)} + \mathcal{O}(\epsilon^{2}) + \mathcal{E}^{(4)}_{i,j}, \label{eq: Euclidean distance correction}
\end{align}
for all $i,j=1,\ldots,n$, $i\neq j$, where $\max_i \vert \mathcal{E}^{(2)}_i \vert$, $\max_i \vert \mathcal{E}^{(3)}_i \vert$, and $\max_{i,j} \vert \mathcal{E}^{(4)}_{i,j} \vert$ are upper bounded by the right-hand side of~\eqref{eq: variance error probabilistic bound} with probability at least $1-n^{-t}$, for any $t>t_0$.
\end{proposition}

Proposition~\ref{prop: noise magnitude est} asserts that for sufficiently large $m,n$ and sufficiently small $\epsilon$, the quantities $\hat{N}_i$,  $\hat{S}_i$, and $\hat{D}_{i,j}$ can approximate $\Vert \eta_i \Vert_2^2$,  $\Vert x_i \Vert_2^2$, and $\Vert x_i - x_j \Vert_2^2$, respectively, up to arbitrarily small errors with high probability. According to~\eqref{eq: noise estimator bias and variance error},~\eqref{eq: signal estimator bias and variance error}, and~\eqref{eq: Euclidean distance correction}, the first error term in these approximations is a global constant that depends explicitly on $d$, $s$, and $\epsilon$, and thus can be removed if the intrinsic dimension $d$ is known or can be estimated. Alternatively, if one is only interested in ranking $\Vert \eta_i \Vert_2^2$,  $\Vert x_i \Vert_2^2$, or $\Vert x_i - x_j \Vert_2^2$, then the relevant bias error term is improved to $\mathcal{O}(\epsilon^{2})$ since ranking is unaffected by a global additive constant. For example, this is the case if one is interested in identifying the points with the largest or smallest noise magnitudes or determining the nearest neighbors of each point $x_i$. The variance error terms $\mathcal{E}^{(2)}$, $\mathcal{E}^{(3)}$, $\mathcal{E}^{(4)}$ have the same behavior as $\mathcal{E}$ from Theorem~\ref{thm:scaling factors asym expression} in Section~\ref{sec: main results}.

Note that the noise magnitude estimator $\hat{N}_i$ in~\eqref{eq: noise magnitude estimator} corrects for the effect of the variability of the density $q(x_i)$ on the scaling factors $\mathbf{d}$. However, one does not have to use $\hat{q}_i$ in~\eqref{eq: noise magnitude estimator}, and it can be replaced with the constant $1$. In such a case, we would still have an $\mathcal{O}(\epsilon)$ bias error term in each of~\eqref{eq: noise estimator bias and variance error},~\eqref{eq: signal estimator bias and variance error}, and~\eqref{eq: Euclidean distance correction}, but it would depend on $q(x_i)$ (and $q(x_j)$ in the case of~\eqref{eq: Euclidean distance correction}). Hence, the $\mathcal{O}(\epsilon)$ term would no longer be a global constant that does not influence ranking. Consequently, the main advantage of accounting for the density is to improve the effective bias error term from $\mathcal{O}(\epsilon)$ to $\mathcal{O}(\epsilon^2)$ under ranking.

We begin by demonstrating $\hat{N}_i$ and $\hat{S}_i$ via a toy example. We generated two centered circles in $\mathbb{R}^2$, one with radius $1$ and the other with radius $0.5$. We independently sampled $500$ points from the first circle and $500$ from the second circle according to the same (non-uniform) density used for Figure~\ref{fig: density est clean data}. We then embedded all points in $\mathbb{R}^m$ with $m=500$ by applying a random orthogonal transformation and added i.i.d outlier-type noise $\eta_i$ taken to be zero with probability $0.9$ and sampled from a multivariate normal with covariance $\sigma_i I_m/m$ with probability $0.1$, where $\sigma_i$ is sampled uniformly from $(0,1)$; see Figure~\ref{fig: two circles outlier noise} for a two-dimensional visualization of this setup. Figure~\ref{fig: two circles noisy signal magnitude} illustrates the noisy point magnitudes $\Vert y_i \Vert_2^2$, where the signal magnitudes are clearly intertwined with the noise. 
Figure~\ref{fig: two circles noise estimator} illustrates the noise magnitude estimator $\hat{N}_i$ from~\eqref{eq: noise magnitude estimator} with $s=2$ and $\epsilon=0.1$, for each index $i=1,\ldots,1000$. It is evident that $\hat{N}_i$ accurately infers the true noise magnitudes $\Vert \eta_i\Vert_2^2$, albeit a small upward shift due to the bias term $\epsilon d \log s /(4(s-1))$ from~\eqref{eq: noise estimator bias and variance error}. Importantly, $\hat{N}_i$ is invariant to the density $q(x_i)$ and the signal magnitudes $\Vert x_i \Vert_2^2$. Similarly, Figure~\ref{fig: two circles estimated signal} shows that $\hat{S}_i$ accurately recovers $\Vert x_i \Vert_2^2$ up to a small global shift and minor fluctuations. Overall, the doubly stochastic scaling allows us to decompose $\Vert y_i \Vert_2^2$ into signal and noise parts. In particular, $\hat{N}_i$ can be utilized to identify the noisy points in this setting, while $\hat{S}_i$ reveals that the clean data points can be partitioned into two groups with distinct magnitudes.

\begin{figure} 
  \centering
  	{
  	\subfloat[][Two-dimensional visualization]  
  	{
    \includegraphics[width=0.29\textwidth]{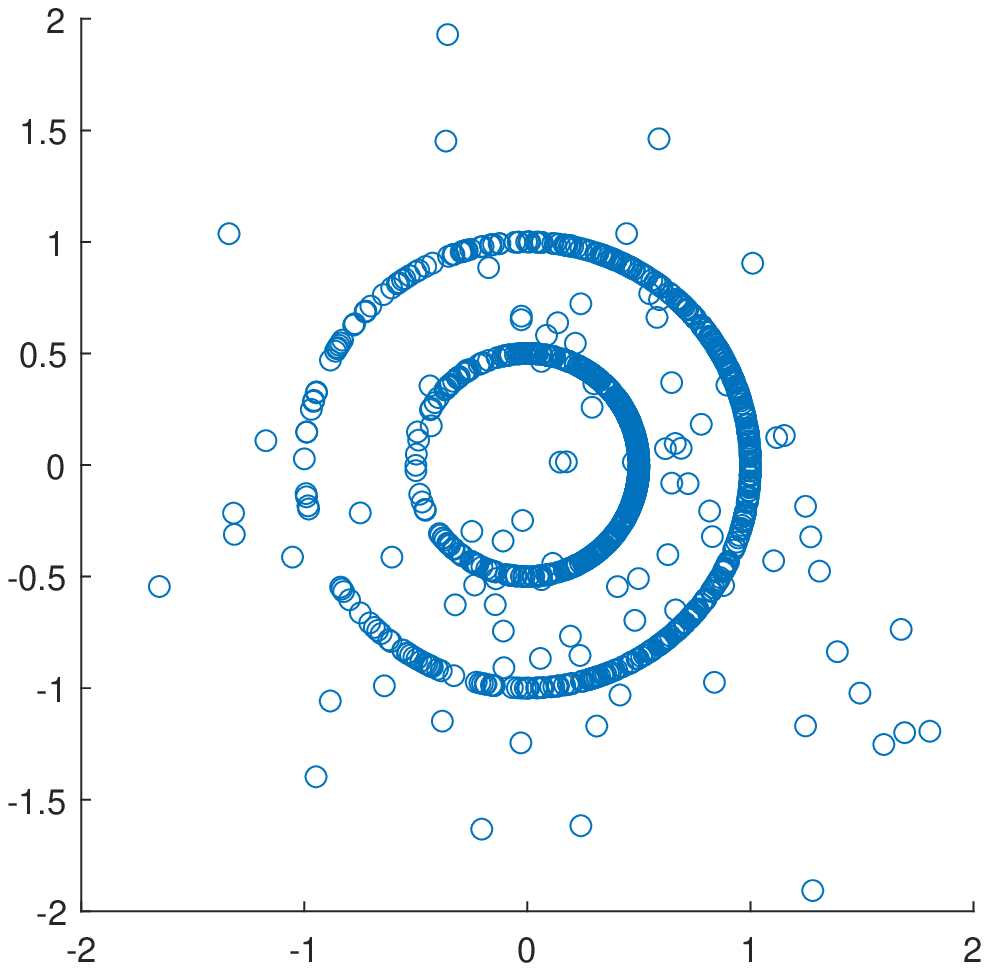} \label{fig: two circles outlier noise}
    }
    \hspace{30pt}
    \subfloat[][Magnitudes of noisy points $\Vert y_i \Vert_2^2$]
  	{
    \includegraphics[width=0.29\textwidth]{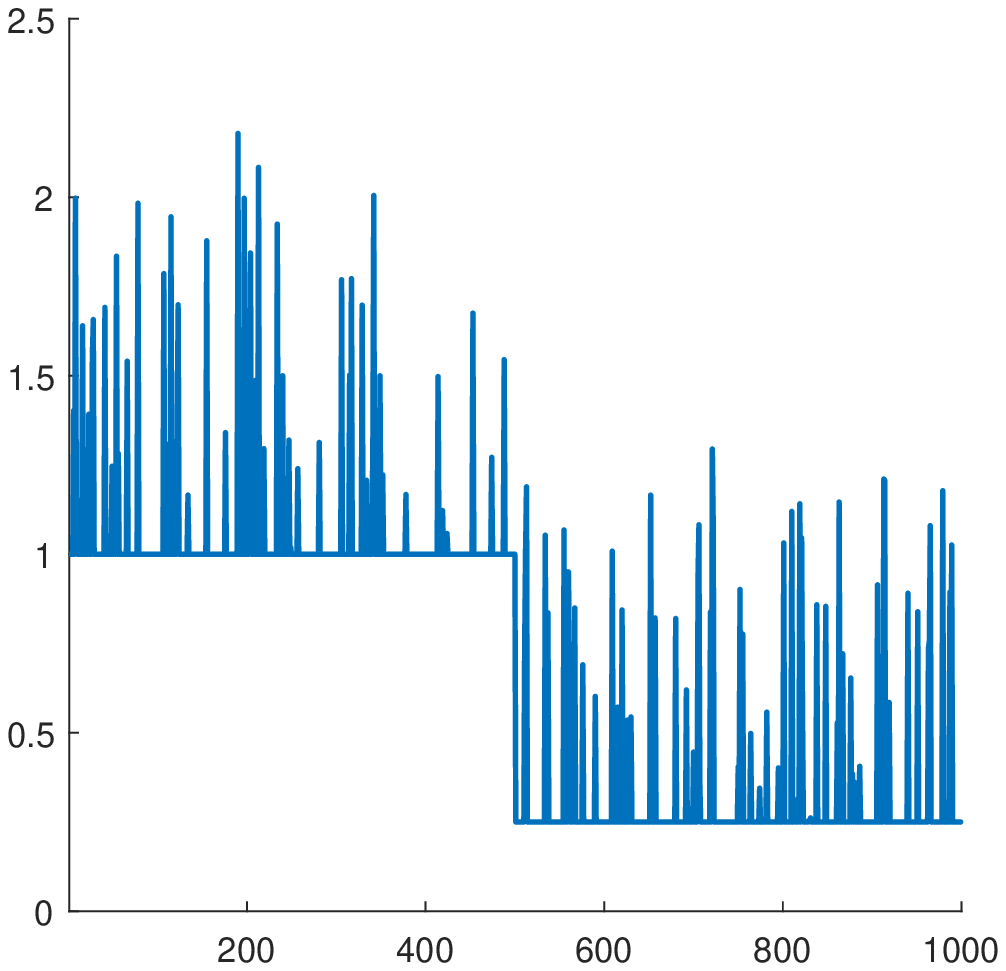} \label{fig: two circles noisy signal magnitude}
    }
    \\
    \vspace{10pt}
    \subfloat[][Estimated vs. true noise magnitude]  
  	{
    \includegraphics[width=0.29\textwidth]{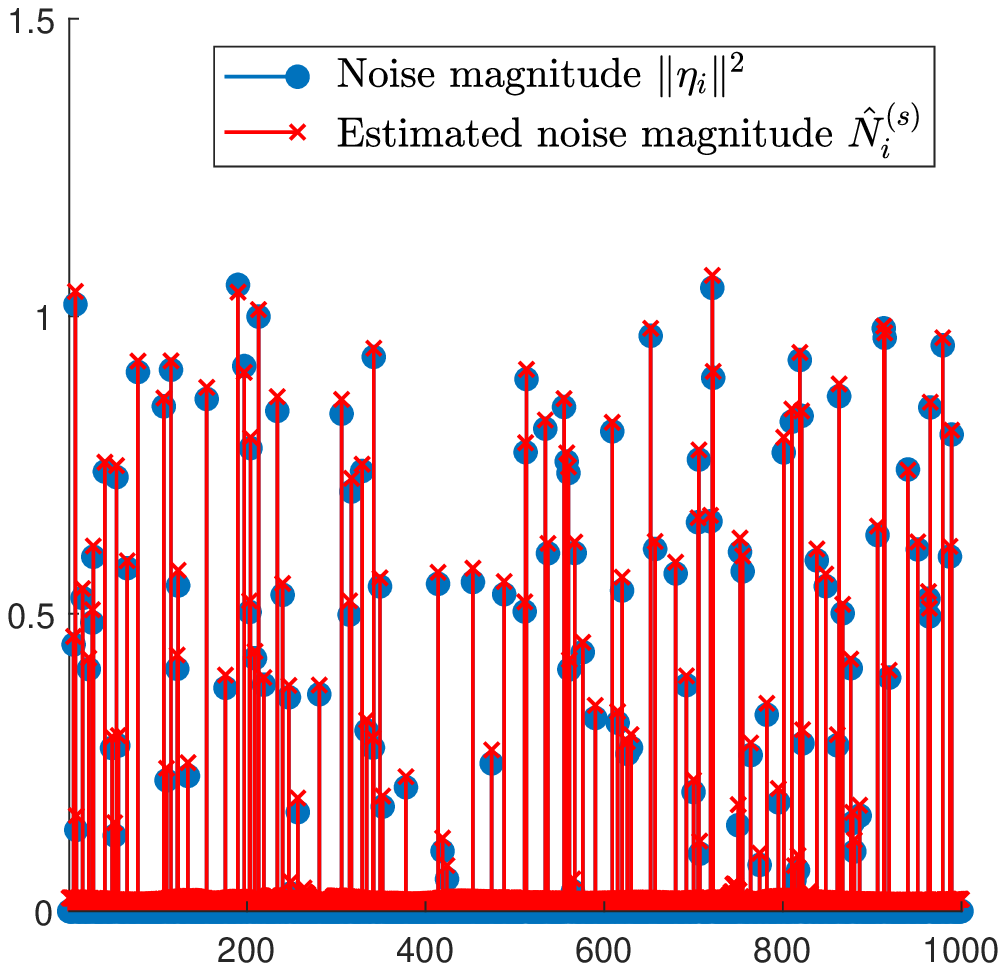} \label{fig: two circles noise estimator}
    }
    \hspace{30pt}
    \subfloat[][Estimated vs. true signal magnitude]  
  	{
    \includegraphics[width=0.29\textwidth]{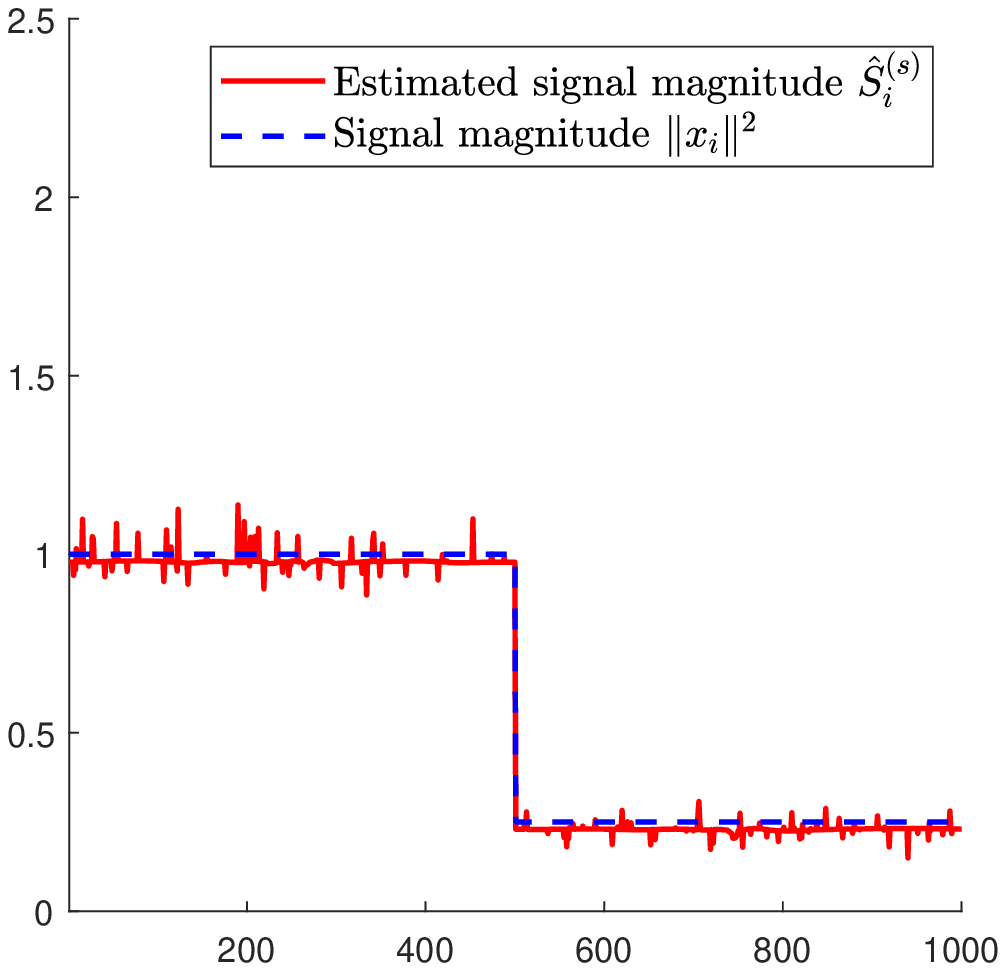} \label{fig: two circles estimated signal}
    }
    }
    \caption
    {Example of estimating the noise magnitudes $\Vert \eta_i \Vert_2^2$ and signal magnitudes $\Vert x_i\Vert_2^2$ using $\hat{N}_i$ and $\hat{S}_i$ from~\eqref{eq: noise magnitude estimator} and~\eqref{eq: Euclidean distance correction}, respectively (panels (c) and (d)), with $s=2$, $\epsilon=0.1$, for $i=1,\ldots,1000$. The noisy observations $y_i$ are sampled from two concentric circles, each with the non-uniform density of Figure~\ref{fig: density est clean data}, embedded in dimension $m=500$ and corrupted with probability $0.1$ by multivariate normal noise with covariance $\sigma_i I_m$, where $\sigma_i$ is sampled uniformly at random from $(0,1)$ (panels (a) and (b)).
    } \label{fig: recovering noise magnitudes and signal magnitudes}
\end{figure} 

Next, we demonstrate the advantage of correcting the noisy Euclidean distances $\Vert y_i - y_j \Vert_2^2$ via $\hat{D}$ from~\eqref{eq: Euclidean distance correction}. Figure~\ref{fig: noisy vs. clean distances} shows the noisy distances $\Vert y_i - y_j \Vert_2^2$ versus the clean distances $\Vert {x}_i - {x}_j \Vert_2^2$ in the setup of Figures~\ref{fig: density est example varying noise} and~\ref{fig: circle with varying noise}, where points are sampled from a circle and corrupted by noise with magnitude that is varying smoothly from $0.01$ to $0.5$ (see more details in the text of Section~\ref{sec: robust density estimation}). It is evident that the Euclidean distances are strongly corrupted by the variability of the noise magnitude and deviate substantially from the desired behavior, which is the dashed diagonal line (describing perfect correspondence). In Figure~\ref{fig: corrected vs. clean distances} we depict the corrected distances $\hat{D}_{i,j}$ computed with $s=2$ and $\epsilon=0.1$, which are much closer to the clean distances and concentrate well around a line with slope $1$ that is shifted slightly below the desired trend. This shift agrees almost perfectly with the bias term $-\epsilon {d \log(s)}/{(2(s-1))} \approx -0.035$ appearing in~\eqref{eq: Euclidean distance correction}. Lastly, for each $k=1,\ldots,50$, we computed the $k$ nearest neighbors of each $x_i$ according to the noisy distances $\Vert y_i - y_j \Vert_2^2$ and the corrected distances $\hat{D}_{i,j}$. For each $k$, Figure~\ref{fig: near neighbor identification accuracy} shows the proportion of these nearest neighbors that coincide with any of the $k$ true nearest neighbors according to the clean distances $\Vert x_i - x_j \Vert_2^2$, averaged over $i=1,\ldots,1000$. It is clear that the corrected distances allow for more accurate identification of near neighbors, with more than $80\%$ accuracy for $k=50$ while the noisy distances provide less than $60\%$ accuracy in that case. Note that the nearest neighbors of each point $x_i$ according to the corrected distances $\hat{D}_{i,j}$ correspond to the largest entries of $\sqrt{\hat{q}_i} W_{i,j} \sqrt{\hat{q}_j}$ in each row $i$. Hence, Figure~\ref{fig: near neighbor identification accuracy} also describes the advantage of the affinity matrix $\sqrt{\hat{q}_i} W_{i,j} \sqrt{\hat{q}_j}$ over the noisy Gaussian kernel ${K}_{i,j}$ in encoding similarities.

\begin{figure} 
  \centering
  	{
  	\subfloat[Noisy vs. clean distances]  
  	{
    \includegraphics[width=0.29\textwidth]{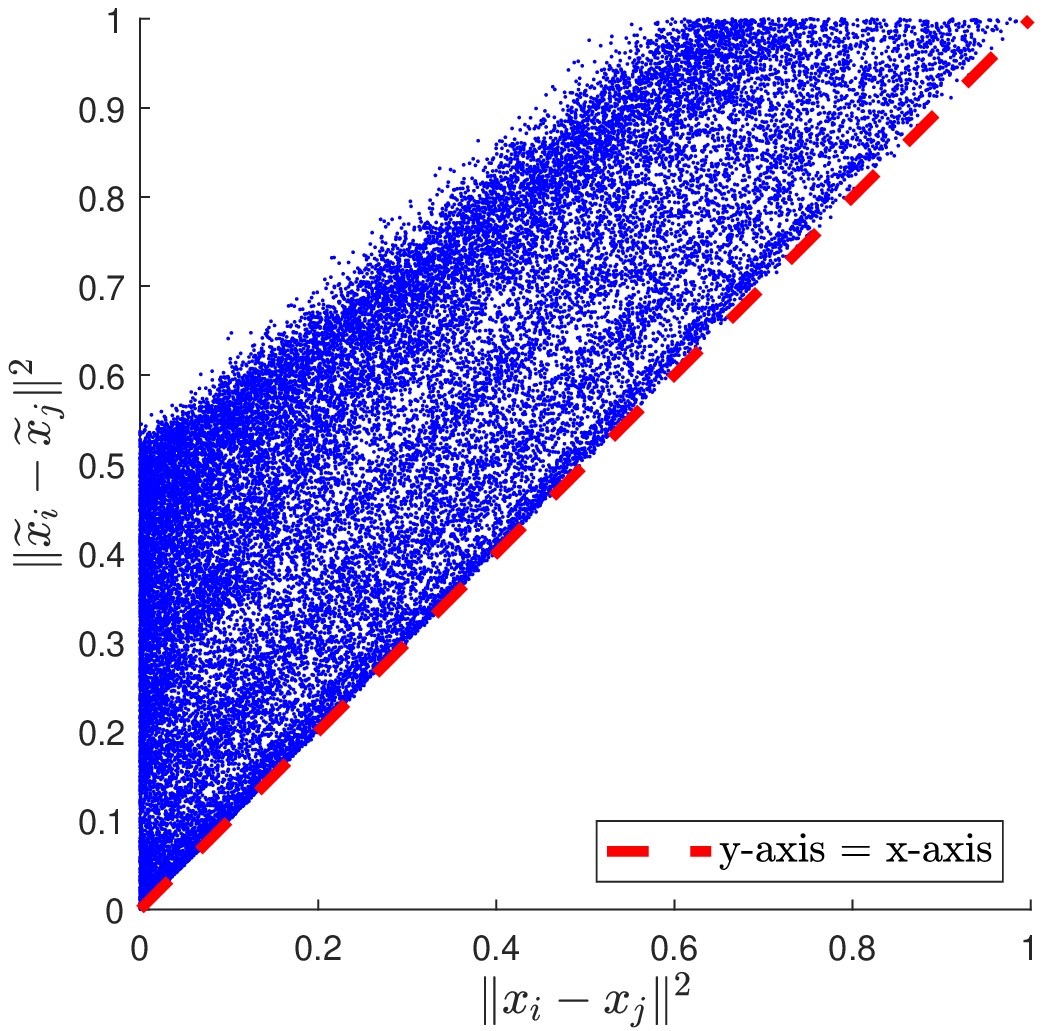} \label{fig: noisy vs. clean distances}
    }
    \hspace{15pt}
    \subfloat[Corrected vs. clean distances]  
  	{
    \includegraphics[width=0.29\textwidth]{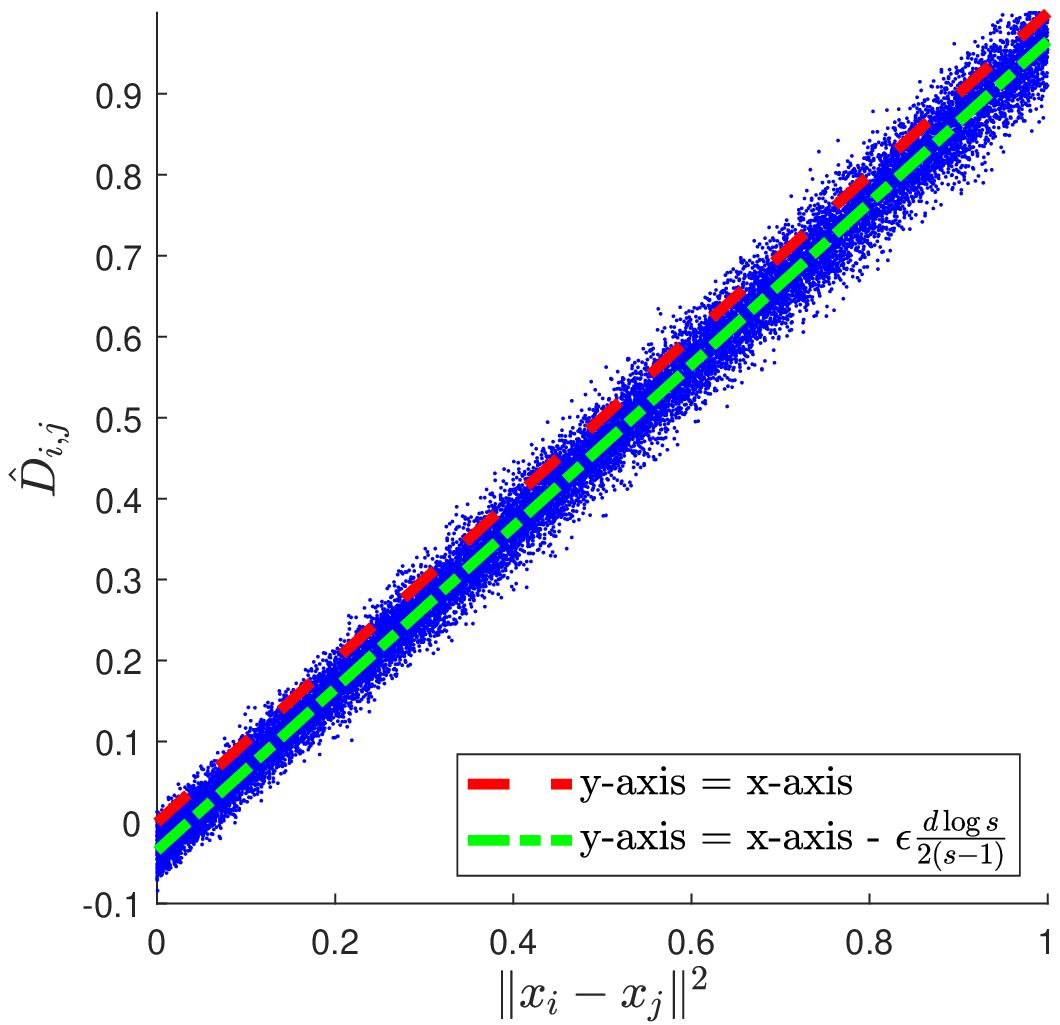} \label{fig: corrected vs. clean distances}
    }
    \hspace{15pt}
    \subfloat[Near neighbor detection]  
  	{
    \includegraphics[width=0.29\textwidth]{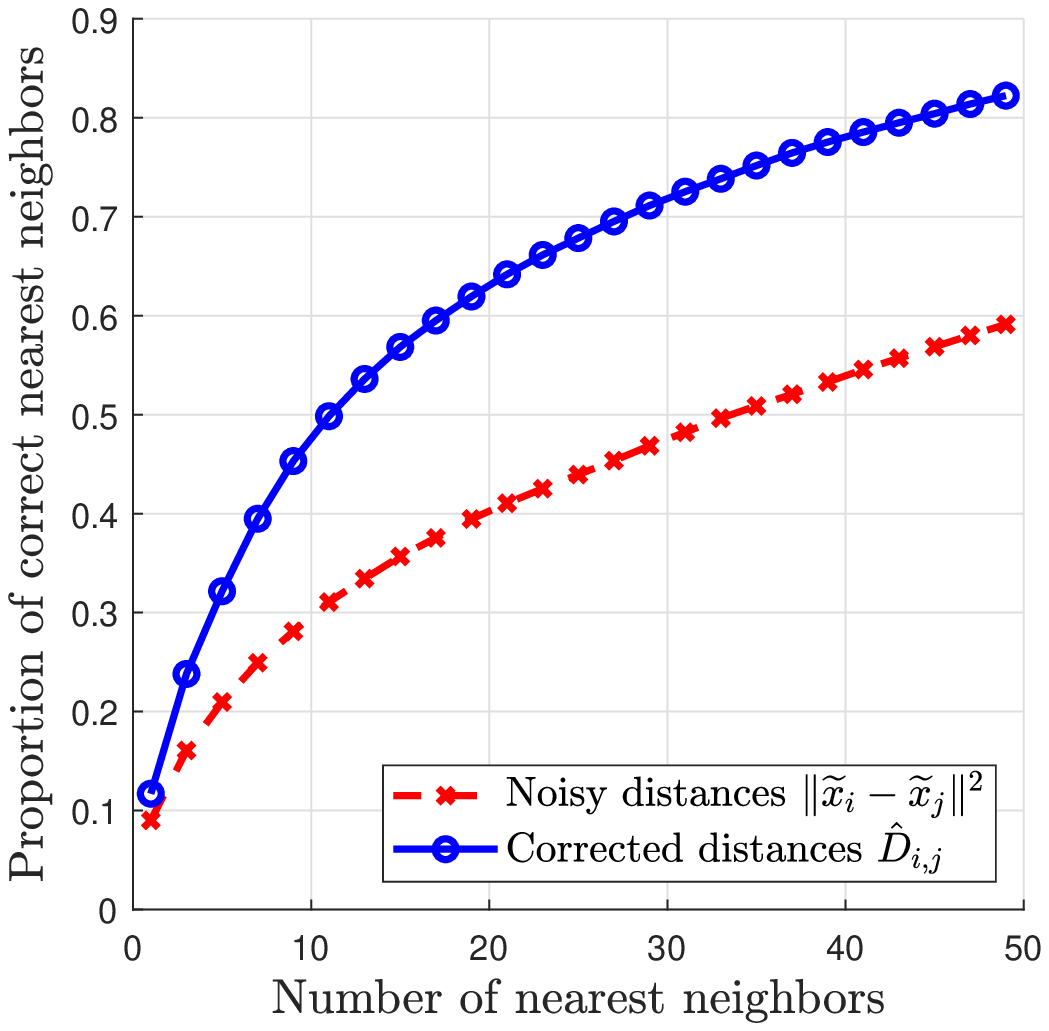} \label{fig: near neighbor identification accuracy}
    }
    }
    \caption
    {Influence of noise on pairwise Euclidean distances and on detection accuracy of the nearest neighbors, compared to the corrected distances $\hat{D}_{i,j}$ from~\eqref{eq: Euclidean distance correction} with $s=2$ and $\epsilon=0.1$.
    The data is sampled from a unit circle with non-uniform density, embedded in dimension $m=1000$ and corrupted by noise whose magnitude varies smoothly from $0.01$ to $0.5$; see Figures~\ref{fig: density est example varying noise} and~\ref{fig: circle with varying noise} and relevant text.
    } \label{fig: recovering Euclidean distances}
    \end{figure} 

\subsection{Robust weighted manifold Laplacian approximation} \label{sec: graph Laplacian estimation}
In what follows we construct a family of normalizations that is a robust analogue of~\eqref{eq: traditional graph Laplacian normalizations} and establish convergence to the associated family of differential operators (see~\cite{coifman2006diffusionMaps}).

Fix $\alpha\in [0,1]$, and define
\begin{equation}
    L^{(\alpha)} = \frac{4(I_{n} - \hat{W}^{(\alpha)})}{\epsilon}, \qquad \hat{W}^{(\alpha)}_{i,j} = \frac{\widetilde{W}^{(\alpha)}_{i,j}}{\sum_{j=1}^n \widetilde{W}^{(\alpha)}_{i,j}}, \qquad \widetilde{W}^{(\alpha)}_{i,j} = \frac{W_{i,j}}{\left[ \hat{q}_i \hat{q}_j \right]^{\alpha - 1/2}}
    , \label{eq: robust graph Laplacian normalizations}
\end{equation}
for all $i,j=1,\ldots,n$,
where $I_{n}$ is the $n\times n$ identity matrix, and $L^{(\alpha)}$ is an appropriately-normalized graph Laplacian for $\hat{W}^{(\alpha)}$. The formulas in~\eqref{eq: robust graph Laplacian normalizations} are equivalent to those in~\eqref{eq: traditional graph Laplacian normalizations} except that we utilize the robust density estimator $\hat{q}_i$ instead of the standard KDE and further account for the asymptotic approximation of $W$ in~\ref{eq: W_ij and d_i limiting form} (leading to the power $\alpha -0.5$ in the denominator of $\widetilde{W}^{(\alpha)}$ instead of the power $\alpha$ appearing in $P^{(\alpha)}$ of~\eqref{eq: traditional graph Laplacian normalizations}). Note that when $\alpha = 0.5$, no normalization of $W$ is actually performed since $\hat{W}^{(0.5)}_{i,j} = W_{i,j}$.

Next, we define the Schrodinger-type differential operator
\begin{equation}
    \{T^{(\alpha)}f\}(x) = \frac{\Delta_\mathcal{M}\{ f q^{1-\alpha} \} (x)}{[q(x)]^{1-\alpha}} - \frac{\Delta_\mathcal{M}\{q^{1-\alpha} \} (x)}{[q(x)]^{1-\alpha}} f(x), \label{eq: differential operator T}
\end{equation}
for any $f\in\mathcal{C}^2(\mathcal{M})$, where $\Delta_\mathcal{M}$ is the negative Laplace-Beltrami operator on $\mathcal{M}$. If the sampling density is uniform, i.e., $q(x)$ is a constant function, then $T^{(\alpha)}$ reduces to $\Delta_\mathcal{M}$ for any $\alpha$. Otherwise, $T^{(\alpha)}$ depends on the density $q(x)$, except for the special case of $\alpha = 1$, where the density vanishes and $T^{(\alpha)}$ again becomes $\Delta_\mathcal{M}$. When $\alpha = 0.5$, $T^{(\alpha)}$ is the Fokker-Planck operator describing Brownian motion via the Langevin equation~\cite{nadler2006diffusion}, and when $\alpha = 0$,  $T^{(\alpha)}$ describes the limiting operator of the popular random walk graph Laplacian. 

We now have the following result, whose proof can be found in Appendix~\ref{appendix: proof of robust graph Laplacian normalization}.
\begin{theorem} \label{thm: Laplacian estimator}
Fix $f\in\mathcal{C}^3(\mathcal{M})$.  Under Assumptions~\ref{assump: manifold}--\ref{assump: pointwise convergence}, there exist $\epsilon_0,t_0,m_0(\epsilon),n_0(\epsilon), C^{'}({\epsilon})>0$, such that for all $\epsilon < \epsilon_0$, $m>m_0(\epsilon)$, $n>n_0(\epsilon)$, we have
\begin{equation}
    \sum_{j=1}^n L^{(\alpha)}_{i,j} f(x_j) =  \{T^{(\alpha)}f\}(x_i) + \mathcal{O}(\epsilon^\beta) +  \mathcal{E}^{(5)}_i, \label{eq: robust graph Laplacian approximation bias and variance errors}
\end{equation}
for all $i=1,\ldots,n$, where $\max_i \vert \mathcal{E}_i^{(5)} \vert$ is upper bounded by the right-hand side of~\eqref{eq: variance error probabilistic bound} with probability at least $1-n^{-t}$, for any $t>t_0$.
\end{theorem} 
Theorem~\ref{thm: Laplacian estimator} shows that for sufficiently large $m,n$, and sufficiently small $\epsilon$, the matrix $L^{(\alpha)}$ can approximate the action of the operator $T^{(\alpha)}$ pointwise up to an arbitrarily small error with high probability. If $\alpha=1$, then $L^{(\alpha)}$ approximates the (negative) Laplace-Beltrami operator $\Delta_\mathcal{M}$, which encodes the intrinsic geometry of the manifold~\cite{coifman2006diffusionMaps} regardless of the sampling density. If $\alpha = 0.5$, then $L^{(\alpha)} = 4(I_n - W)/\epsilon$ approximates the Fokker-Planck operator, suggesting that the doubly stochastic Markov matrix $W$ simulates Langevin diffusion on $\mathcal{M}$, agreeing with the results in~\cite{marshall2019manifold,wormell2021spectral,cheng2022bi}. If $\alpha = 0$, we have $\widetilde{W}^{(\alpha)}_{i,j} = \sqrt{\hat{q}_{i}} W_{i,j} \sqrt{\hat{q}_{j}}$, which corrects for the influence of density on $W$ according to~\eqref{eq: W_ij and d_i limiting form}, approximating the clean Gaussian kernel up to a global constant. In this case, $L^{(\alpha)}$ approximates the same operator as the standard random walk graph Laplacian on the clean data.
Theorem~\ref{thm: Laplacian estimator} shows that the popular family of normalizations~\eqref{eq: traditional graph Laplacian normalizations} can be made robust to general high-dimensional noise via the doubly stochastic affinity matrix $W$ and our robust density estimator~\eqref{eq:density estimator formula}.  

In Figure~\ref{fig: Laplacian estimation example} we demonstrate the advantage of the robust graph Laplacian normalization~\eqref{eq: robust graph Laplacian normalizations} for $\alpha=1$ over the traditional normalization $4(I_n - \hat{P}^{(1)})/\epsilon$, where $\hat{P}^{(1)}$ is from~\eqref{eq: traditional graph Laplacian normalizations}.
We used the same setting as the one used for Figure~\ref{fig: density est clean data} and Figure~\ref{fig: density est example varying noise}, namely the unit circle with non-uniform density, where the sampled points are either clean or corrupted by smoothly varying noise; see more details in the corresponding text of Section~\ref{sec: robust density estimation}. To quantify the accuracy of the approximation in~\eqref{eq: robust graph Laplacian approximation bias and variance errors}, we used the test function $f(\theta(x)) = (\cos(\theta(x))+\sin(2\theta(x)))/5$, where $\theta(x)\in [0,2\pi)$ is the angle of a point $x$ on the circle. For $\alpha=1$, $T^{(\alpha)}$ reduces to the Laplace-Beltrami operator $\Delta_\mathcal{M}$, in which case $\{\Delta_\mathcal{M}(f)\}(x_i) = (-\cos(\theta_i)-4\sin(2\theta_i))/5$, where $\theta_i=\theta(x_i)$. We then computed the maximal absolute difference between $\sum_{j=1}^n L^{(\alpha)}_{i,j} f(x_j)$ and $\{\Delta_\mathcal{M}(f)\}(x_i)$ over $i=1,\ldots,n$, both for our robust normalization as well as the traditional one. Figure~\ref{fig: robust Laplacian vs n} shows these errors versus the sample size $n$, where $\epsilon=0.1$, $s=2$, $m=n$, and we averaged the errors over $30$ randomized trials. Figure~\ref{fig: robust Laplacian vs n} shows the same errors versus the bandwidth parameter $\epsilon$, where $m=n=5000$, $s=2$, and we averaged the results over $20$ randomized trials. It is evident that the robust and the traditional graph Laplacian normalizations perform nearly identically in the clean case, both across $n$ and across $\epsilon$, suggesting that the bias and variance errors in~\eqref{eq: robust graph Laplacian approximation bias and variance errors} match those for the traditional normalization, at least in our setting. On the other hand, the robust normalization performs much better in the noisy case whenever the error is dominated by the variance term, i.e., when $\epsilon$ is sufficiently small with respect to the sample size $n$, while having almost identical behavior when the error is dominated by the bias term. Consequently, the robust normalization achieves smaller errors for any fixed bandwidth in this scenario, and allows us to use a smaller optimally-tuned bandwidth to obtain a better approximation of the Laplace-Beltrami operator under noise.

\begin{figure} 
  \centering
  	{
  	\subfloat[Error versus $n$]  
  	{
    \includegraphics[width=0.35\textwidth]{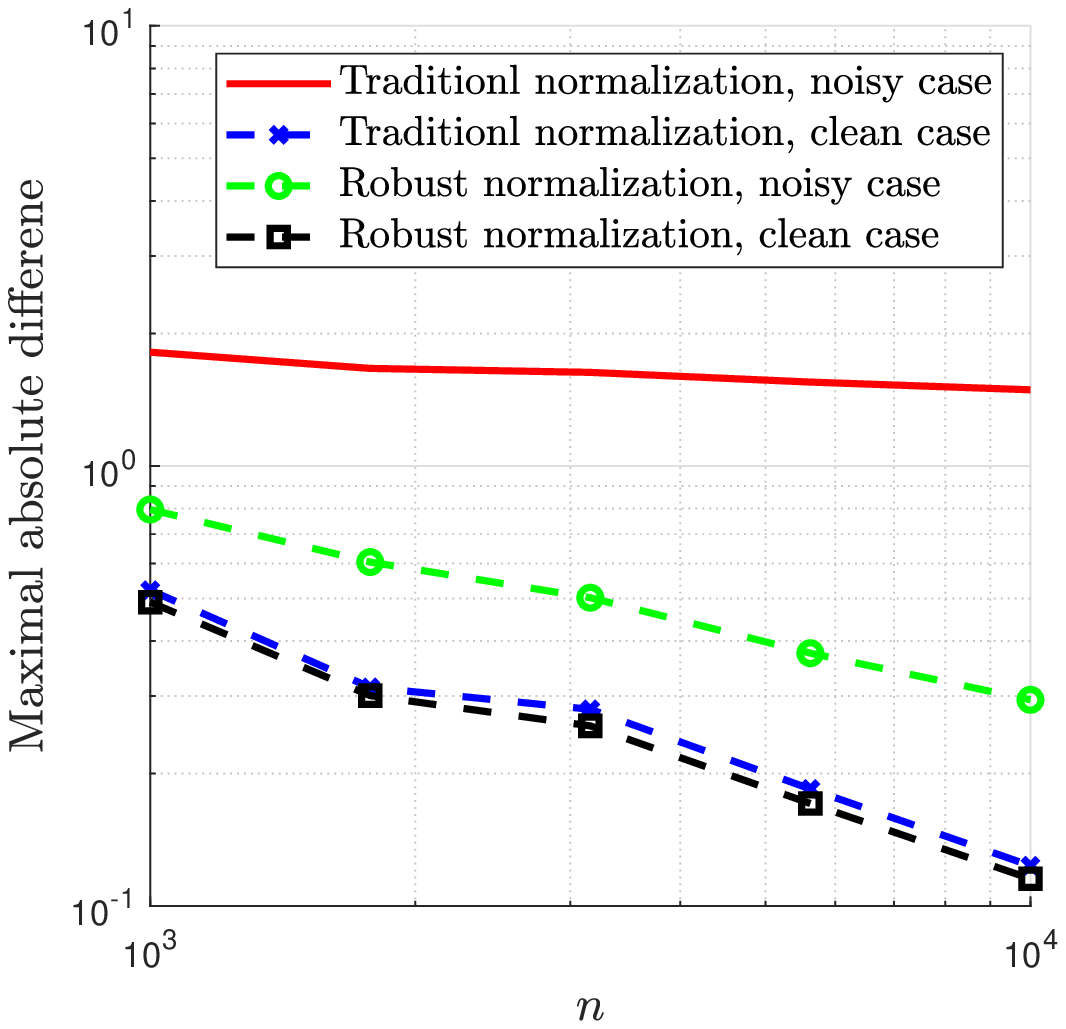} \label{fig: robust Laplacian vs n}
    }
    \hspace{25pt}
    \subfloat[Error versus $\epsilon$]  
  	{
    \includegraphics[width=0.35\textwidth]{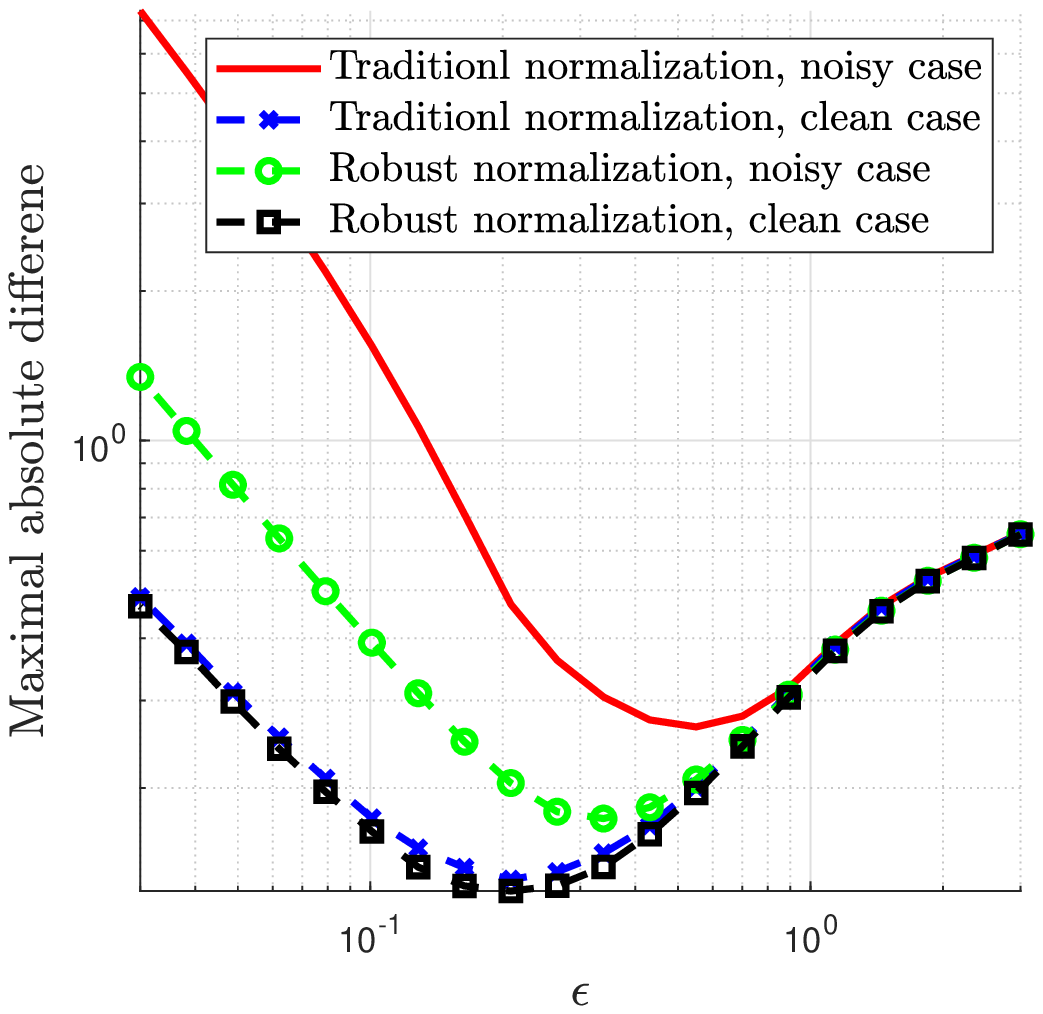} \label{fig: robust Laplacian vs epsilon}
    }
    }
    \caption
    {Maximal absolute difference between the Laplace-Beltrami operator $\{\Delta_\mathcal{M}f\}(x_i)$ and its approximation $\sum_{j=1}^n L^{(\alpha)}_{i,j} f(x_j)$ using the robust graph Laplacian normalization~\eqref{eq: robust graph Laplacian normalizations} with $\alpha=1$ and its traditional counterpart, versus the sample size $n$ and the bandwidth $\epsilon$. The data is sampled according to the setting used for Figures~\ref{fig: density est clean data}, ~\ref{fig: density est example varying noise}, and~\ref{fig: circle with varying noise}, namely a unit circle with non-uniform density where the points are either clean or corrupted by smoothly varying noise in high dimension. } \label{fig: Laplacian estimation example}
    \end{figure} 

\section{Experiments on real single-cell RNA-sequencing data} \label{sec: scRNA-seq example}
In this section, we demonstrate our results using real data from Single-cell RNA-sequencing (scRNA-seq), which is a revolutionary technology for measuring high-dimensional gene expression profiles of individual cells in diverse populations~\cite{tang2009mrna,macosko2015highly}. In this case, each observation $\widetilde{y}_i\in\mathbb{R}^m$ is a vector of nonnegative integers describing the expression levels of $m$ different genes in the $i$'th cell of the sample. The high resolution of the data -- given at the single-cell level -- makes it possible to study the similarities between different cells and to characterize different cell populations, which is of paramount importance in immunology and developmental biology. However, one of the main challenges in analyzing scRNA-seq data is the high levels of noise and its non-uniform nature~\cite{kharchenko2021triumphs,kim2015characterizing,kim2020demystifying}. 

To demonstrate our results, we used the popular dataset of Purified Peripheral Blood Mononuclear Cells (PBMC) by~\cite{zheng2017massively}, where $32733$ genes are sequenced over $94654$ cells that are annotated experimentally according to $10$ known cell types. To preprocess the data, we first randomly subsampled $500$ cells from each of the following types: `b cells', `cd14 monocytes', `cd34', `cd4 helper', `cd56 nk', and `cytotoxic t'. These cell types are fairly distinguishable one from another, thereby simplifying the interpretability of our subsequent results, while the subsampling makes computations more tractable. We then computed the total expression count for each cell, given by $c_i = \sum_{j=1}^m \widetilde{y}_{i,j}$, where $\widetilde{y}_{i,j}$ denotes the $j$'th entry of $\widetilde{y}_i$, and computed the normalized observations $y_i = \widetilde{y}_i/c_i$. This is a standard normalization in scRNA-seq for making the cell descriptors to be probability vectors, thereby removing the influence of technical variability of counts (also known as ``read depth'') across cell populations~\cite{vieth2019systematic,cole2019performance}. The doubly stochastic scaling of the Gaussian kernel is then evaluated using the normalized observations $y_1,\ldots,y_n$, $n=3000$, with a prescribed tolerance of $10^{-6}$ and a maximum of $10^4$ iterations in the algorithm of~\cite{wormell2021spectral}. 

In our first experiment, we set out to investigate the accuracy of noise magnitude estimator $\hat{N}_i$ described in~\eqref{eq: noise magnitude estimator}. To validate our noise estimates, we assume the popular Poisson data model $\widetilde{y}_{i,j} \sim \operatorname{Poisson}(\mu_{i,j})$~\cite{sarkar2021separating}. In this case, we have $\mathbb{E}[\widetilde{y}_{i,j}] = \operatorname{Var}[\widetilde{y}_{i,j}] = \mu_{i,j}$, and by standard concentration arguments,
\begin{equation}
    \Vert \eta_i \Vert_2^2 \sim \frac{\sum_{j=1}^m \left(  \widetilde{y}_{i,j} - \mu_{i,j} \right)^2}{c_i^2} \sim \frac{\sum_{j=1}^m \operatorname{Var}[\widetilde{y}_{i,j}]}{c_i^2} = \frac{\sum_{j=1}^m \mu_{i,j}}{c_i^2} \sim \frac{\sum_{j=1}^m \widetilde{y}_{i,j}}{c_i^2} = \frac{1}{c_i}, \label{eq: predicted noise Poisson model}
\end{equation}
asymptotically as $m\rightarrow \infty$ under appropriate de-localization conditions on the Poisson parameters $\{\mu_{i,j}\}_{j=1}^m$. Therefore, we expect the noise magnitude $\Vert \eta_i \Vert_2^2$ to be close to $1/c_i$, which is the inverse of the total gene expression counts for cell $i$. Figure~\ref{fig: estimated noise vs. Poisson noise scRNAseq} depicts the estimated noise magnitudes $\hat{N}_i$ computed with $\epsilon = 2\cdot 10^{-5}$ and $s=2$, versus $1/c_i$ for a prototypical subsampled dataset (with $3000$ cells total, $500$ from each of six different types). Evidently, the Poisson model suggests that the noise magnitude fluctuates considerably across the data, roughly by an order of magnitude. Of course, the noise magnitude estimator $\hat{N}_i$ is completely oblivious to the Poisson model and does not have access to the total counts $c_i$ (as it is determined solely from the normalized observations). Nonetheless, the estimated noise magnitudes  $\hat{N}_i$ concentrate around the red dashed diagonal line, showcasing good agreement with the Poisson model. Note that there seems to be a slight quadratic trend to the estimated noise magnitudes, which is in line with literature suggesting over-dispersion with respect to the standard Poisson~\cite{svensson2020droplet} (e.g., negative binomial). 
 
 In our second experiment, we employ the cell type annotations to assess the accuracy of $\hat{W}_{i,j}^{(\alpha)}$ from~\eqref{eq: robust graph Laplacian normalizations} and its traditional counterpart $\hat{P}^{(\alpha)}$ from~\eqref{eq: traditional graph Laplacian normalizations}. Since $\hat{W}_{i,j}^{(\alpha)}$ is a transition probability matrix, it describes a random walk over the cells. It is reasonable to assume that a random walk that starts at a certain cell should be unlikely to immediately transition to a cell with a different type. 
Motivated by this reasoning, we show in Figure~\ref{fig: average probability to leave cell type scRNAseq} the probability of a cell to transition to a cell with a different type, averaged over all cells $i=1,\ldots,3000$, according to $\hat{W}_{i,j}^{\alpha}$ from~\eqref{eq: robust graph Laplacian normalizations} as well as is traditional counterpart, where we used $s=2$, $\alpha=0,0.5,1$, and averaged the results over $20$ randomized trials (of subsampling cells), plotted against the bandwidth parameter $\epsilon$. Figure~\ref{fig: worst-case probability to leave cell type scRNAseq} is the same as Figure~\ref{fig: average probability to leave cell type scRNAseq} except that we averaged the aforementioned probabilities over the cells in each cell type separately and took the largest of these, namely the worst-case averaged transition probability over the six cell types.

From Figures~\ref{fig: average probability to leave cell type scRNAseq} and~\ref{fig: worst-case probability to leave cell type scRNAseq} it is evident that for large bandwidth parameters $\epsilon$, all normalizations provide similarly undesirable behavior in the form of large probabilities of transition errors, namely probabilities to transition between different cell types. As we decrease $\epsilon$, the behavior generally improves across all normalizations, but the errors made by the robust normalizations $\hat{W}^{(\alpha)}$ are consistently smaller than the traditional ones. This advantage of our normalizations is particularly evident over the traditional normalization with $\alpha = 0$, whose worst-case error (for one of the cell types) exceeds $0.4$ for all values of $\epsilon$. The traditional normalization with $\alpha=0.5$ seems to be more accurate than $\alpha = 0$ or $\alpha=1$ and provides results very similar to the robust normalizations, albeit slightly larger worst-case errors for small $\epsilon$. Recall that the traditional normalization with $\alpha=0.5$ is obtained by first performing the symmetric normalization $D_i^{-1/2} W_{i,j} D_i^{-1/2}$, where $D_i = \sum_{j=1}^n {K}_{i,j}$, and then performing a row-stochastic normalization. These steps are precisely one iteration of the accelerated scaling algorithm described in~\cite{wormell2021spectral}, which can possibly explain the advantage of $\alpha=0.5$ over the other values of $\alpha$.

\begin{figure} 
  \centering
  	{
  	\subfloat[Estimated noise vs. Poisson noise] 
  	{
    \includegraphics[width=0.29\textwidth]{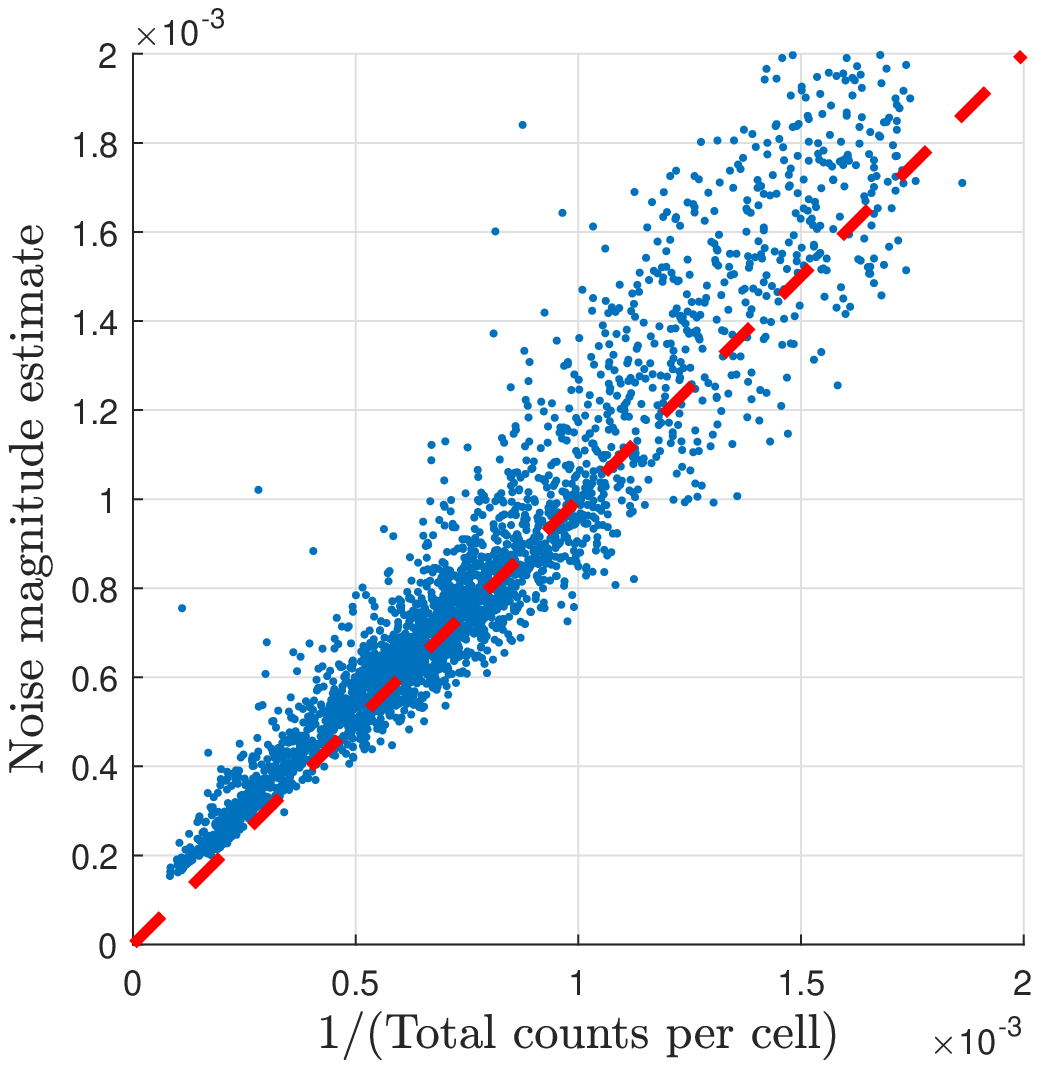}  \label{fig: estimated noise vs. Poisson noise scRNAseq}
    } 
    \hspace{15pt}
    \subfloat[Average transition probability error]  
  	{
    \includegraphics[width=0.29\textwidth]{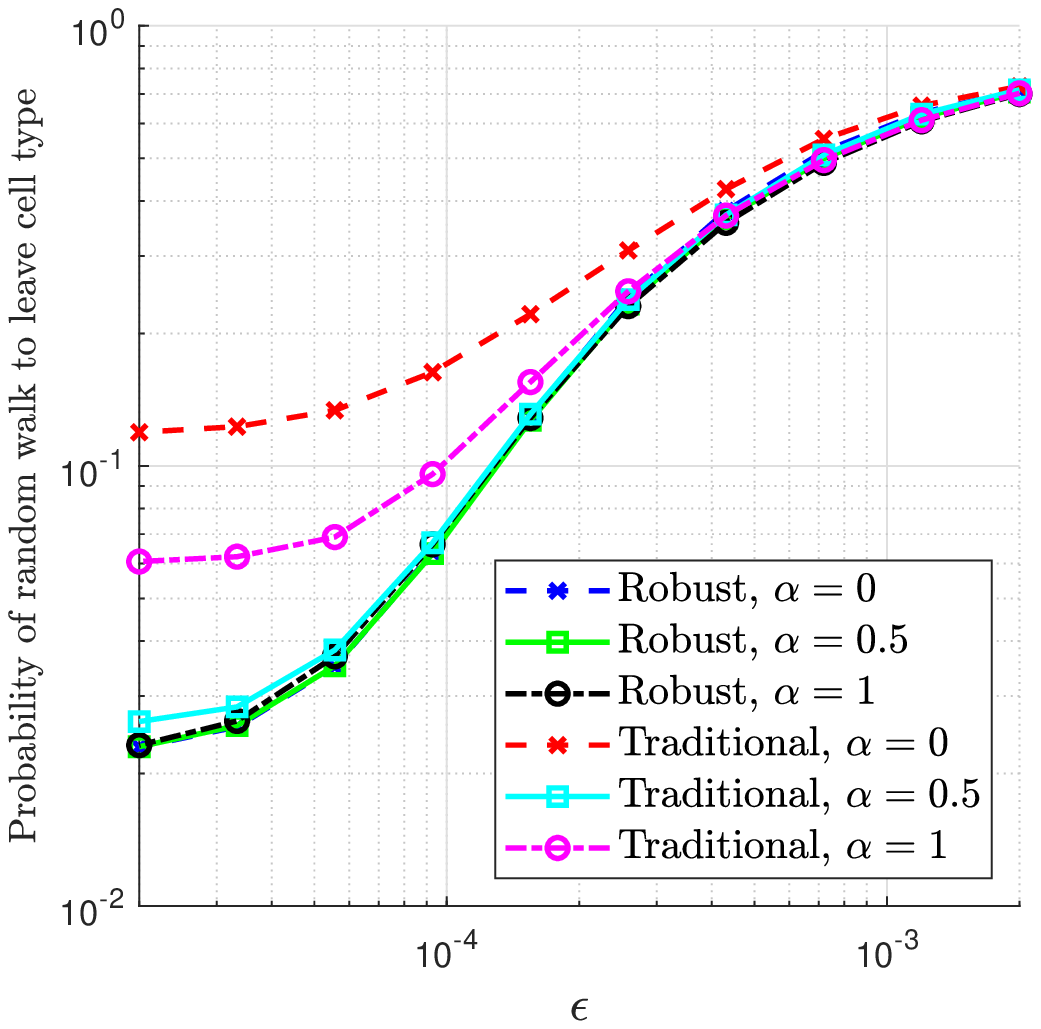} \label{fig: average probability to leave cell type scRNAseq}
    } 
    \hspace{15pt}
    \subfloat[Worst-case transition probability error]  
  	{
    \includegraphics[width=0.29\textwidth]{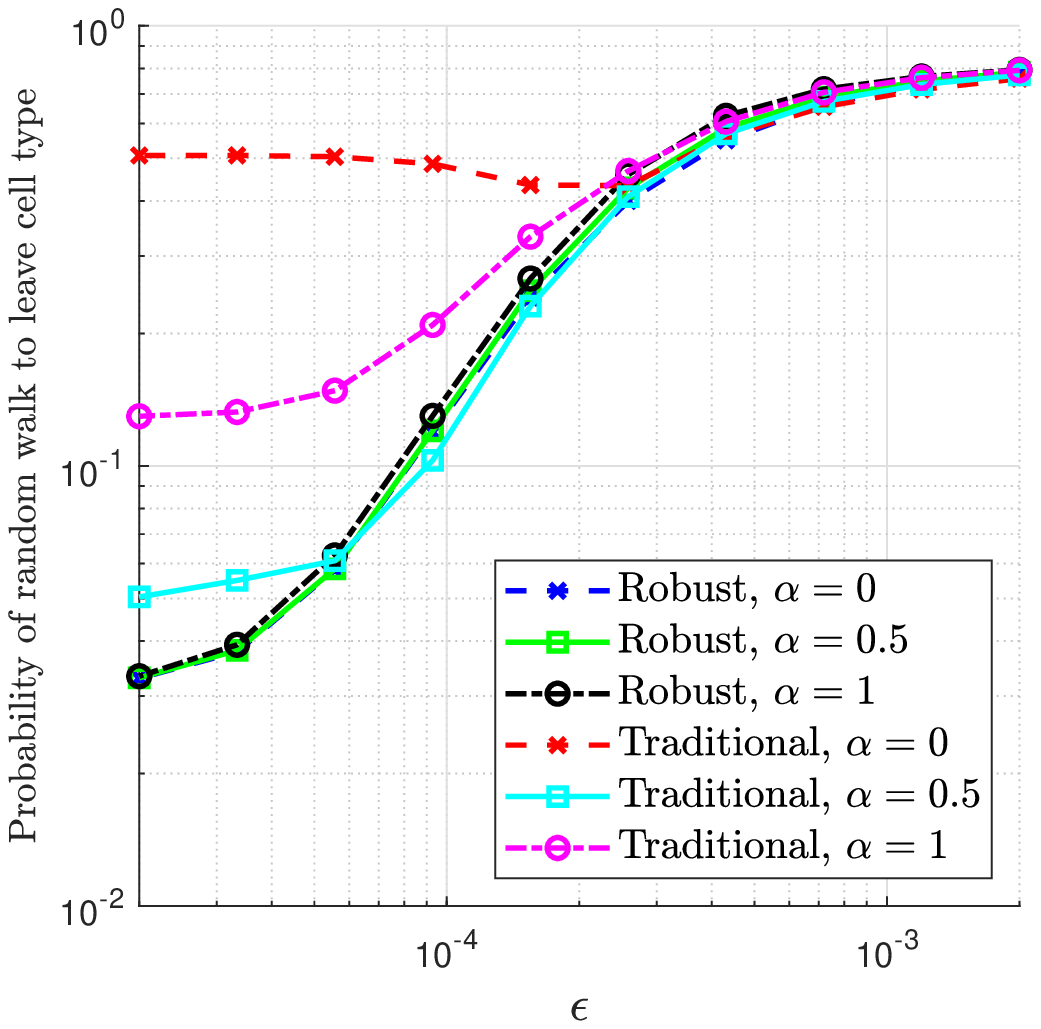} \label{fig: worst-case probability to leave cell type scRNAseq}
    }
    }
    \caption
    {Accuracy of the noise magnitude estimate $\hat{N}_i$ from~\eqref{eq: noise magnitude estimator} (panel (a)) and the Markov matrix $\hat{W}^{(\alpha)}$ from~\eqref{eq: robust graph Laplacian normalizations} as well as its traditional analogue $\hat{P}^{(\alpha)}$ from~\eqref{eq: traditional graph Laplacian normalizations} (panels (b) and (c)), evaluated using the annotated single-cell RNA-sequencing data from~\cite{zheng2017massively}. The noise magnitude estimate $\hat{N}_i$ is compared to the anticipated value from the Poisson model~\eqref{eq: predicted noise Poisson model}, while the error in the Markov matrix $\hat{W}^{(\alpha)}$ and its traditional analogue is assessed by the average probability to transition between distinct cell types (the smaller the better), plotted against the bandwidth parameter $\epsilon$.
    } \label{fig: scRNAseq example} 
    \end{figure} 

Note that all robust normalizations provide nearly identical probabilities of transition errors, which may initially seem strange in light of the different limiting operators for the corresponding graph Laplacians in Theorem~\ref{thm: Laplacian estimator}. However, an important distinction is that the graph Laplacian $L^{(\alpha)}$ only describes the first-order behavior $W^{(\alpha)}$ in $\epsilon$, while its zeroth-order behavior is given by $\sum_{j=1}^n \hat{W}^{(\alpha)}_{i,j} f(x_j) \sim f(x_i)$ for large $m,n$ and small $\epsilon$, regardless of $\alpha$. 
Hence, we should indeed expect the random walk transition probability errors to be very similar across $\alpha$, which is the case for the robust normalization. On the other hand, the random walk transition probability errors for the traditional normalization differ substantially across $\alpha$, which is likely due to the strong variability of the noise in the data (see~Figure~\ref{fig: estimated noise vs. Poisson noise scRNAseq}) and the sensitivity of the standard kernel density estimator to such noise.

\section{Discussion} \label{sec: discussion}
The results in this work give rise to several future research directions. On the practical side, to make the tools we developed in Section~\ref{sec: application to inference} widely applicable, it is desirable to derive procedures for adaptively tuning the bandwidth parameter $\epsilon$ and the parameter $s$ in the DS-KDE~\eqref{eq:density estimator formula}. Moreover, for large experimental datasets, the density can vary considerably across the sample space, where a global bandwidth parameter is unlikely to provide a satisfactory bias-variance trade-off. Hence, it is worthwhile to tune the bandwidth according to the local density around each point. Techniques for adaptive bandwidth selection have been extensively studied for standard kernel density estimation and traditional graph Laplacian normalizations in the clean case (see~\cite{zelnik2005self,berry2016variable} and references therein), e.g., by near neighbor distances. However, the adaptation of these techniques to our setting is nontrivial, as the near-neighbor distances could be too corrupted for determining the local bandwidth. Therefore, this topic requires substantial analytical and empirical investigation that is left for future work. 

On the theoretical side, one important direction is to characterize the constant $C^{'}(\epsilon)$ appearing in~\eqref{eq: variance error probabilistic bound} in terms of $\epsilon$. This would require a more advanced analysis of the stability of the scaling factors $\mathbf{d}$ under perturbations in ${W}$ and the prescribed row sums, which is beyond the scope of this work. In addition, we conjecture that the results in Theorem~\ref{thm: scaling function convergence} can be strengthened, and specifically that $\rho_{\epsilon}(x) = q^{-1/2}(x) F_{\epsilon}(x) + \mathcal{O}(\epsilon^2)$ uniformly over $x\in\mathcal{M}$. Currently, Theorem~\ref{thm: scaling function convergence} only proves an analogous $L^p$ bound for $p\in [1,4/3)$, while the pointwise bound in Theorem~\ref{thm: scaling function convergence} is only for $d\leq 5$ and is dominated by a dimension-dependent error that is worse than $\mathcal{O}(\epsilon^2)$. A useful first step in that direction might be to obtain a tighter characterization of $\rho_\epsilon$ than in our Lemma~\ref{lem: boundedness of rho}. However, this will require different theoretical tools and is left for future work.
Lastly, properly refined versions of Theorems~\ref{thm:scaling factors asym expression} and~\ref{thm: scaling function convergence} can be combined to describe how to tune the bandwidth parameter $\epsilon$ for convergence of the approximation errors described in Sections~\ref{sec: robust density estimation} and~\ref{sec: graph Laplacian estimation}.

\section*{Acknowledgement}
The authors would like to thank Ronald Coifman and Yuval Kluger for their useful comments and suggestions.
The work is supported by NSF DMS-2007040. 
The two authors acknowledge support by NIH grant R01GM131642.
B.L. acknowledges support by NIH grants UM1DA051410, U54AG076043, and U01DA053628.
X.C. is also partially supported by NSF (DMS-1818945, DMS-1820827, DMS-2134037) and the Alfred P. Sloan Foundation.

\begin{appendices}

\section{Supporting lemmas} \label{appendix: supporting lemmas}
\subsection{Concentration of inner products} \label{appendix: lemma on concentration of inner products}
\begin{lemma} 
\label{lem:noise scalar product concentration}
Under Assumptions~\ref{assump: manifold} and~\ref{assump: noise magnitude}, there exist universal constants $c,m_0,t_0>0$, such that for any $t>t_0,m>m_0$, and $i\neq j$ we have
\begin{equation}
    \operatorname{Pr} \left\{ \left\vert  \langle y_i,y_j \rangle - \langle {x}_i,{x}_j \rangle \right\vert > ct \sqrt{\log m} \cdot \operatorname{max}\{E,E^2\sqrt{m}\} \right\} \leq m^{-t}.
\end{equation}
\end{lemma}
\begin{proof}
Let us write
\begin{equation}
    \langle y_i,y_j \rangle = \langle {x}_i,{x}_j \rangle + \langle {\eta}_i,{x}_j \rangle + \langle {x}_i,\eta_j \rangle + \langle \eta_i,\eta_j \rangle. \label{eq:noisy inner product decomposition}
\end{equation}
Conditioning on $x_1,\ldots,x_n$, the random variable $\langle {\eta}_i,{x}_j \rangle$, for each $i=1,\ldots,n$, is sub-Gaussian with
\begin{equation}
    \Vert \langle {\eta}_i,{x}_j \rangle \Vert_{\psi_2}
    \leq \Vert x_j \Vert_2 \sup_{\Vert y \Vert_2=1} \Vert \langle {\eta}_i,y \rangle \Vert_{\psi_2} \leq E,
\end{equation}
where we also used the definition of sub-Gaussian norm~\cite{vershynin2018high} and the fact that $\Vert x_i \Vert_2 \leq 1$. Therefore, according to Proposition 2.5.2 in~\cite{vershynin2018high}, we have
\begin{equation}
    \operatorname{Pr}\{ \vert \langle {\eta}_i,{x}_j \rangle \vert > \tau \} 
    \leq 2\operatorname{exp}\left\{ - \frac{\tau^2}{c^2 \Vert \langle {\eta}_i,{x}_j \rangle \Vert_{\psi_2}^2} \right\} 
    \leq 2\operatorname{exp}\left\{ -\frac{\tau^2}{c^2 E^2} \right\},
\end{equation}
where $c$ is a universal constant. Taking $\tau = c E  \sqrt{ t \log m}$, for any $t>0$, shows that 
\begin{equation}
    \operatorname{Pr}\{ \vert \langle {\eta}_i,{x}_j \rangle \vert > c E  \sqrt{ t \log m} \} 
    \leq 2 m^{-t}. \label{eq: langle eta, x rangle bound }
\end{equation}
Analogously, the above holds when replacing $\langle {\eta}_i,{x}_j \rangle$ with $\langle {x}_i,\eta_j \rangle$. Since~\eqref{eq: langle eta, x rangle bound } holds conditionally on any realization of $x_1,\ldots,x_{n}$, it also holds unconditionally. Next, conditioning on $\eta_i$ for any $i\in \{1,\ldots,n\}$, the random variable $\langle \eta_i,\eta_j \rangle$ for $j\neq i$ is also sub-Gaussian, where
\begin{equation}
    \bigg \Vert \left( \langle \eta_i,\eta_j \rangle \; \bigg \vert \; \eta_i = h \right) \bigg \Vert_{\psi_2} \leq \Vert h \Vert_2 \sup_{\Vert y \Vert_2=1} \Vert \langle y, {\eta}_j \rangle \Vert_{\psi_2} \leq {E \Vert h \Vert_2}.
\end{equation}
Hence, by the definition of a sub-Gaussian random variable
\begin{equation}
    \operatorname{Pr}\left\{ \vert \langle {\eta}_i,{\eta}_j \rangle \vert > \tau \; \bigg\vert \; \Vert \eta_i \Vert_2 \leq r \right\}
    \leq 2\operatorname{exp}\left\{ -\frac{\tau^2}{c^2 \bigg\Vert \left( \langle {\eta}_i,\eta_j \rangle \; \bigg\vert \; \Vert \eta_i \Vert_2 \leq r \right) \bigg\Vert_{\psi_2}^2} \right\} 
    \leq 2\operatorname{exp}\left\{ -\frac{\tau^2}{c^2 E^2 r^2} \right\},
\end{equation}
and taking $\tau = c r E \sqrt{ t \log m }$, for any $t>0$, gives
\begin{equation}
    \operatorname{Pr}\left\{ \vert \langle {\eta}_i,{\eta}_j \rangle \vert > c r E \sqrt{ {t \log m}} \; \bigg\vert \; \Vert \eta_i \Vert_2 \leq r \right\}
    \leq 2 m^{-t}.
\end{equation}
We can now write
\begin{align}
    \operatorname{Pr}\left\{ \vert \langle {\eta}_i,{\eta}_j \rangle \vert >  c r E \sqrt{ {t \log m}} \right\} 
    &= \operatorname{Pr}\left\{ \vert \langle {\eta}_i,{\eta}_j \rangle \vert > c r E \sqrt{ {t \log m}} \; \bigg\vert \; \Vert \eta_i \Vert_2 \leq r \right\} \operatorname{Pr}\{  \Vert \eta_i \Vert_2 \leq r \} \nonumber \\
    &+ \operatorname{Pr}\left\{ \vert \langle {\eta}_i,{\eta}_j \rangle \vert > c r E \sqrt{ {t \log m}} \; \bigg\vert \; \Vert \eta_i \Vert_2 > r \right\} \operatorname{Pr}\{  \Vert \eta_i \Vert_2 > r \} \nonumber \\
    &\leq 2m^{-t} + \operatorname{Pr}\{  \Vert \eta_i \Vert_2 > r \}.
\end{align}
Therefore, we need a probabilistic bound on $\Vert \eta_i \Vert_2$. Applying Theorem 2.1 in~\cite{hsu2012tail} to the sub-Gaussian random vector $\eta_i$ tells us that
\begin{equation}
    \operatorname{Pr}\left\{ \Vert \eta_i \Vert_2  > \Vert \eta_i \Vert_{\psi_2} \left( c_1 \sqrt{m} + c_2 \sqrt{ \log(1/\delta)}\right)\right\} \leq \delta,
\end{equation}
for any $\delta\in (0,1)$, where $c_1,c_2>0$ are universal constants. Taking $\delta = m^{-t}$ and letting $r = E ({c_1 \sqrt{m}} + c_2\sqrt{ t \log m})$, we arrive at
\begin{equation}
    \operatorname{Pr}\left\{ \vert \langle {\eta}_i,{\eta}_j \rangle \vert > c E^2 \sqrt{{t \log m}} \left( {c_1 \sqrt{m}} + c_2 \sqrt{{ t \log m}} \right)\right\} \leq 3m^{-t}, \label{eq:bound for langle eta_i , eta_j, rangle}
\end{equation}
for any $t>0$. Combining~\eqref{eq:bound for langle eta_i , eta_j, rangle},~\eqref{eq: langle eta, x rangle bound },~\eqref{eq:noisy inner product decomposition}, and applying the union bound, provides the required result.
\end{proof}

\subsection{Stability of scaling factors}
\begin{lemma} 
\label{lem:closeness of scaling factors}
Let ${A}\in\mathbb{R}^{n\times n}$, $n\geq 3$, be a symmetric nonnegative matrix with zero main diagonal and strictly positive off-diagonal entries. Denote by $\mathbf{v} = [v_1,\ldots,v_n]>0$ the unique vector for which $D(\mathbf{v}) A D(\mathbf{v})$ is doubly stochastic, and suppose that there exist $\varepsilon\in (0,1)$ and a vector $\hat{\mathbf{v}} = [\hat{v}_1,\ldots,\hat{v}_n] > 0$ such that
\begin{equation}
\left\vert \sum_{j=1}^n \hat{v}_i {A}_{i,j} \hat{v}_j - 1 \right\vert \leq \varepsilon, \label{eq:approximate scaling condition in lemma}
\end{equation}
for all $i=1,\ldots,n$. Then, 
\begin{align}
\sqrt{\frac{1 - \delta}{1+\varepsilon}} \leq \frac{{v}_i}{\hat{v}_i} \leq \sqrt{\frac{1}{(1-\delta)(1-\varepsilon)}},
\end{align}
for all $i=1,\ldots,n$, where 
\begin{equation}
    \delta = \min \left\{ 1, \frac{4 \varepsilon }{ (n-2) \min_{i,j = 1,\ldots,n, i\neq j} \left\{ \hat{v}_i A_{i,j} \hat{v}_j \right\} } \right\}. \label{eq: delta def}
\end{equation}
\end{lemma}
\begin{proof}
Let us define
\begin{equation}
    P_{i,j} = v_i A_{i,j} v_j, \qquad
    \hat{P}_{i,j} = \hat{v}_i A_{i,j} \hat{v}_j, \qquad  
    u_i = \frac{{v}_i}{\hat{v}_i}, \qquad \ell = \operatorname{argmin}_i u_i, \qquad k = \operatorname{argmax}_i u_i. \label{eq:u def}
\end{equation}
Using~\eqref{eq:approximate scaling condition in lemma} we can write
\begin{equation}
    1 = \sum_{j=1}^N {P}_{\ell,j} = \sum_{j=1}^N u_\ell \hat{P}_{\ell,j} u_j \leq \min_i u_i \max_i u_i \sum_{j=1}^N \hat{P}_{\ell,j} \leq (1+\varepsilon) \min_i u_i \max_i u_i. \label{eq:min u max u lower bound}
\end{equation}
Similarly, we have
\begin{equation}
    1 = \sum_{j=1}^N {P}_{k,j} = \sum_{j=1}^N u_k \hat{P}_{k,j} u_j \geq \max_i u_i \min_i u_i \sum_{j=1}^N \hat{P}_{k,j} \geq (1-\varepsilon) \max_i u_i \min_i u_i. \label{eq: max u min u upper bound}
\end{equation}
Combining~\eqref{eq:min u max u lower bound} and~\eqref{eq: max u min u upper bound}, we obtain
\begin{equation}
    \frac{1}{1+\varepsilon} \leq \max_i u_i \min_i u_i \leq \frac{1}{1-\varepsilon}. \label{eq:max u min u upper and lower bounds}
\end{equation}
Continuing, employing~\eqref{eq:max u min u upper and lower bounds} and the fact that $u_k = \max_i u_i$, we can write
\begin{equation}
    1 = \sum_{j=1}^n {P}_{k,j} = \sum_{j=1}^n u_k \hat{P}_{k,j} u_j \geq \frac{1}{(1+\varepsilon) }\sum_{j=1}^n \hat{P}_{k,j} \frac{u_j}{\min_i u_i}, \label{eq:bounding max v - min v step 1}
\end{equation}
implying that
\begin{equation}
    \sum_{j=1}^n \hat{P}_{k,j} \left( \frac{u_j}{\min_i u_i} -1 \right) \leq {1+\varepsilon} - \sum_{j=1}^n \hat{P}_{k,j} \leq {2\varepsilon}. \label{eq:preliminary bound 1}
\end{equation}
Multiplying~\eqref{eq:preliminary bound 1} by $\min_i u_i /\min_{j\neq k}  \hat{P}_{k,j}$ and using the fact that $\hat{P}_{k,k} = 0$, it follows that
\begin{equation}
  \sum_{j=1,\; j\neq k}^n ({u_j} - \min_i u_i ) \leq  \sum_{j=1}^n \frac{\hat{P}_{k,j}}{\min_{j\neq k} \hat{P}_{k,j} } ({u_j} - \min_i u_i ) \leq \frac{2\varepsilon\min_i u_i}{\min_{j\neq k} \hat{P}_{k,j} } \leq  \frac{2\varepsilon\max_i u_i}{\min_{j\neq k} \hat{P}_{k,j} }. \label{eq:sum u_j - min u_i bound}
\end{equation}
We next provide a derivation analogous to~\eqref{eq:bounding max v - min v step 1}--\eqref{eq:sum u_j - min u_i bound} to obtain a bound for $\sum_{j=1,\; j\neq \ell}^n (\max_i u_i - u_j )$. 
Using~\eqref{eq:max u min u upper and lower bounds} and the fact that $u_\ell = \min_i u_i$, we can write
\begin{equation}
    1 = \sum_{j=1}^n {P}_{\ell,j} = \sum_{j=1}^n u_\ell \hat{P}_{\ell,j} u_j \leq \frac{1}{(1-\varepsilon) }\sum_{j=1}^n \hat{P}_{\ell,j} \frac{u_j}{\max_i u_i}, 
\end{equation}
implying that
\begin{equation}
    \sum_{j=1}^n \hat{P}_{\ell,j} \left( \frac{u_j}{\max_i u_i} -1 \right) \geq {1-\varepsilon} - \sum_{j=1}^n \hat{P}_{\ell,j} \geq {-2\varepsilon}.
\end{equation}
Multiplying the above by $-\max_i u_i /\min_{j\neq \ell}  \hat{P}_{\ell,j}$ and using the fact that $\hat{P}_{\ell,\ell} = 0$, it follows that
\begin{equation}
  \sum_{j=1,\; j\neq \ell}^n (\max_i u_i - {u_j}  ) \leq  \sum_{j=1}^n \frac{\hat{P}_{\ell,j}}{\min_{j\neq \ell} \hat{P}_{\ell,j} } (\max_i u_i - {u_j}  ) \leq \frac{2\varepsilon\max_i u_i}{\min_{j\neq \ell} \hat{P}_{\ell,j} }. \label{eq:sum max v_i - u_j bound}
\end{equation}
Next, summing~\eqref{eq:sum u_j - min u_i bound} and~\eqref{eq:sum max v_i - u_j bound} gives
\begin{equation}
     (n-2) \left(\max_i u_i - \min_i u_i \right) \leq \frac{4 \varepsilon \max_i u_i }{\min_{i\neq j} \hat{P}_{i,j}},
\end{equation}
while it is also clear that
\begin{equation}
    (n-2) \left(\max_i u_i - \min_i u_i \right) \leq (n-2) \max_i u_i,
\end{equation}
since $\min_i u_i > 0$. Therefore, we have
\begin{equation}
    0 \leq 1 - \frac{\min_i u_i}{\max_i u_i} \leq \delta,
\end{equation}
or equivalently,
\begin{equation}
    1-\delta \leq \frac{\min_i u_i}{\max_i u_i} \leq 1, \label{eq: minu maxu ratio bound}
\end{equation}
where $\delta$ is from~\eqref{eq: delta def}. 
Finally, combining~\eqref{eq: minu maxu ratio bound} with~\eqref{eq:max u min u upper and lower bounds}, after some manipulation tells us that
\begin{equation}
    \sqrt{\frac{1 - \delta}{1+\varepsilon}} \leq \min_i u_i \leq u_j \leq \max_i u_i \leq \sqrt{\frac{1}{(1-\delta)(1-\varepsilon)}},
\end{equation}
for all $j=1,\ldots,n$. 
\end{proof}

\subsection{Boundedness of the scaling function $\rho_{\epsilon}(x)$}
\begin{lemma} 
\label{lem: boundedness of rho}
Under assumption~\ref{assump: manifold}, there exist constants $\epsilon_0,C>0$ that depend only on $\mathcal{M}$ and $q$, such that $\vert \{ x\in\mathcal{M}: \; \rho_{\epsilon}(x)>t \} \vert \leq C/t^2$ for all $\epsilon\leq \epsilon_0$ and $t>0$.
\end{lemma}
\begin{proof}
Let $\Omega_\epsilon^t := \{ x\in\mathcal{M}: \; \rho_{\epsilon}(x)>t \}$. For any $x\in \Omega_\epsilon^t$, we can write
\begin{align}
    \frac{1}{t} &> \frac{1}{\rho_{\epsilon}(x)}  = \frac{1}{(\pi \epsilon)^{d/2}}\int_{\mathcal{M}} \mathcal{K}_{\epsilon}(x,y) \rho_{\epsilon}(y) q(y) d\mu(y) \geq \frac{1}{(\pi \epsilon)^{d/2}}\int_{\Omega_\epsilon^t} \mathcal{K}_{\epsilon}(x,y) \rho_{\epsilon}(y) q(y) d\mu(y) \nonumber \\
    &\geq \frac{c t}{(\pi \epsilon)^{d/2}} \int_{\Omega_\epsilon^t} \mathcal{K}_{\epsilon}(x,y) d\mu(y), \label{eq:local kernel intergal on unbounded domain}
\end{align}
where $c := \min_{x\in\mathcal{M}}q(x) > 0$ by the positivity and continuity of $q(x)$. Next, let $B_r(x)$ be a ball in $\mathbb{R}^m$ of radius $r$ and center $x\in\mathbb{R}^m$, and consider a sequence of points $x_1^{'},\ldots,x_{\kappa_\epsilon^t}^{'} \in \Omega_\epsilon^t$ such that the balls $\{ B_{\sqrt{\epsilon}}(x_i^{'}) \}$ are disjoint. We take $\kappa_\epsilon^t$ to be the largest integer that allows for this property, noting that such a maximum exists since $\Omega_\epsilon^t \subseteq \mathcal{M} \subset B_1(0)$. Then, $\cup_{i=1}^{\kappa_\epsilon^t} B_{2\sqrt{\epsilon}}(x_i^{'})$ is a covering of $\Omega_\epsilon^t$, as otherwise it would be a contradiction to maximality of $\kappa_\epsilon^t$. We have
\begin{align}
    \vert \Omega_\epsilon^t \vert &= \int_{\Omega_\epsilon^t} d\mu(y) \leq \sum_{i=1}^{\kappa_\epsilon^t} \int_{\Omega_\epsilon^t \cap B_{2\sqrt{\epsilon}}(x_i^{'})} d\mu(y) 
    \leq e^4 \sum_{i=1}^{\kappa_\epsilon^t} \int_{\Omega_\epsilon^t \cap B_{2\sqrt{\epsilon}}(x_i^{'})} \mathcal{K}_{\epsilon}(x,y) d\mu(y) \nonumber \\
    &\leq e^4 \sum_{i=1}^{\kappa_\epsilon^t} \int_{\Omega_\epsilon^t} \mathcal{K}_{\epsilon}(x,y) d\mu(y) \leq \frac{e^4 (\pi \epsilon)^{d/2}}{c t^2} \kappa_\epsilon^t, \label{eq:Omega set upper bound}
\end{align} 
where we used~\eqref{eq:local kernel intergal on unbounded domain} in the last inequality. Let $\overline{\kappa}_\epsilon$ be the largest integer for which there is a sequence of points $\overline{x}_1,\ldots,\overline{x}_{\overline{\kappa}_\epsilon} \in \mathcal{M}$ such that $\{ B_{\sqrt{\epsilon}}(\overline{x}_i) \}$ are disjoint. Certainly, we have $\kappa_\epsilon^t \leq \overline{\kappa}_\epsilon$. Since $\mathcal{M}$ is a smooth and compact Riemannian manifold with no boundary and intrinsic dimension $d$, there exist constants $\overline{c},\epsilon_0>0$ that depend only on $\mathcal{M}$, such that 
\begin{equation}
    \kappa_\epsilon^t \leq \overline{\kappa}_\epsilon \leq \frac{\vert \mathcal{M} \vert}{\min_{x\in \mathcal{M}}\vert \mathcal{M} \cap B_{\sqrt{\epsilon}}(x) \vert } 
    \leq \frac{\vert \mathcal{M} \vert}{ \overline{c} \epsilon^{d/2} },
\end{equation}
for all $\epsilon \leq \epsilon_0$. Plugging the above into~\eqref{eq:Omega set upper bound} completes the proof.

\end{proof}

\subsection{Gaussian kernel integral asymptotic expansion} \label{appendix: diffusion maps lemma}
\begin{lemma} [Lemma 8 in~\cite{coifman2006diffusionMaps}] \label{lem: diffusion maps lemma}
Let $f\in \mathcal{C}^3(\mathcal{M})$. Under Assumption~\ref{assump: manifold}, there exists a smooth function $\omega(x):\mathcal{M}\rightarrow \mathbb{R}$ that depends only on the geometry of $\mathcal{M}$, such that for all $x\in \mathcal{M}$,
\begin{equation}
    \frac{1}{(\pi \epsilon)^{d/2}}\int_\mathcal{M} \mathcal{K}_{\epsilon}(x,y) f(y) d\mu(y) = f(x) + \frac{\epsilon}{4}\left[ \omega(x)f(x) - \Delta_\mathcal{M}\{f\}(x) \right] +\mathcal{O}(\epsilon^2).  
\end{equation}
\end{lemma}

\section{Proof of Theorem~\ref{thm:scaling factors asym expression} } \label{appendix: proof of theorem for concentration of scaling factors and scaled matrix}
For brevity, we will make use of the following definition.
\begin{definition} \label{def: order with high probability}
 Let $X$ be a random variable. We say that $X = \mathcal{O}_{m,n}^{(\epsilon)}\left(f(m,n)\right)$ if there exist $t_0,m_0(\epsilon),n_0(\epsilon), C^{'}({\epsilon})>0$, such that for all $m>m_0(\epsilon)$ and $n>n_0(\epsilon)$,
\begin{equation}
    |X| \leq {t} C^{'}({\epsilon}) f(m,n),
\end{equation}
with probability at least $1-n^{-t}$ for all $t>t_0$.
\end{definition}
Note that by Definition~\ref{def: order with high probability}, if we have random variables $X_i = \mathcal{O}_{m,n}^{(\epsilon)}(f(m,n))$ for $i=1,\ldots,P(n)$, where $P(n)$ is a polynomial in $n$, then we immediately get (by applying the union bound $P(n)$ times)
\begin{equation}
    \max_{i=1,\ldots,P(n)} \vert X_i \vert = \mathcal{O}_{m,n}^{(\epsilon)}(f(m,n)). \label{eq: definition order with high probability property 1}
\end{equation}
Additionally, if $X = \mathcal{O}_{m,n}^{(\epsilon)}(f(m,n))$ and $Y = g^{(\epsilon)}(X)$, where $g^{(\epsilon)}\in\mathcal{C}^1(\mathbb{R})$ for all $\epsilon>0$, and $\lim_{m,n\rightarrow \infty} f(m,n) = 0$, then by a Taylor expansion of $g^{(\epsilon)}(x)$,
\begin{equation}
    Y = g^{(\epsilon)}(0) + \mathcal{O}_{m,n}^{(\epsilon)}(f(m,n)). \label{eq: definition order with high probability property 2}
\end{equation}
We will use properties~\eqref{eq: definition order with high probability property 1} and~\eqref{eq: definition order with high probability property 2}  of Definition~\ref{def: order with high probability} seamlessly throughout the remaining proofs.

Since $\mathcal{M}$ is compact and $\rho_{\epsilon}(x)$ from~\eqref{eq:integral scaling eq with density} is positive and continuous, then it is also bounded from above and from below away from zero (by constants that may depend on $\epsilon$).
Now, let us define
\begin{equation}
    \widetilde{d}_i = \frac{\rho_{\epsilon}(x)}{ \sqrt{(n -1) (\pi \epsilon)^{d/2} }}, \qquad \widetilde{h}_i = \widetilde{d}_i \operatorname{exp}\{ -\Vert x_i \Vert_2^2/\epsilon \}, \qquad h_i = d_i \operatorname{exp}\{ -\Vert y_i \Vert_2^2/\epsilon \}, \label{eq: d_i_tilde, h_i_tilde, and h_i def}
\end{equation}
for all $i = 1,\ldots,n$, and
\begin{align}
    {H}_{i,j} = 
    \begin{dcases}
    \operatorname{exp}\{ 2\langle y_i,y_j \rangle / \epsilon \}, & i\neq j,\\
    0, & i=j,
    \end{dcases}, \qquad\qquad
    \widetilde{H}_{i,j} = 
    \begin{dcases}
    \operatorname{exp}\{ 2\langle {x}_i,{x}_j \rangle / \epsilon \}, & i\neq j,\\
    0, & i=j,
    \end{dcases},
\end{align}
for all $i,j = 1,\ldots,n$. Observe that 
\begin{equation}
    1 = \sum_{j=1}^{n} {d}_i K_{i,j} {d}_j 
    = \sum_{j=1}^{n} {h}_i {H}_{i,j} {h}_j,
\end{equation}
where the latter is a system of scaling equations in $\mathbf{h} = [h_1,\ldots,h_{n}] > 0$. Let us denote
\begin{equation}
    \mathcal{E}(m) = \sqrt{\log m} \cdot \operatorname{max}\{E,E^2\sqrt{m}\}. \label{eq: E caligraphic def}
\end{equation}
Applying Lemma~\ref{lem:noise scalar product concentration} and utilizing Definition~\ref{def: order with high probability} together with the fact that $ m \geq n^\gamma$ and $E \leq C/(m^{1/4} \sqrt{\log m})$ (so $\lim_{m\rightarrow\infty}\mathcal{E}(m) = 0$), we have
\begin{align}
    {H}_{i,j} &= \operatorname{exp}\{ 2\langle y_i,y_j \rangle / \epsilon \} = \operatorname{exp}\{ 2 \langle {x}_i,{x}_j \rangle / \epsilon \} \left( 1 + \mathcal{O}_{m,n}^{(\epsilon)} \left({\mathcal{E}(m)}\right) \right)
    = \widetilde{H}_{i,j}  + \mathcal{O}_{m,n}^{(\epsilon)} \left({\mathcal{E}(m)}\right), \label{eq:H_tilde approx}
\end{align}
for all $i\neq j$, where we also used the fact that $\widetilde{H}_{i,j} \leq \operatorname{exp}\{ 2 / \epsilon \}$.
Therefore, we can write
\begin{align}
    \sum_{j=1,\; j\neq i}^{n} \widetilde{h}_i {H}_{i,j} \widetilde{h}_j 
    =&\sum_{j=1,\; j\neq i}^{n} \widetilde{h}_i \widetilde{H}_{i,j} \widetilde{h}_j +  \mathcal{O}_{m,n}^{(\epsilon)} \left(\mathcal{E}(m)\right) \sum_{j=1,\; j\neq i}^{n} \widetilde{h}_i \widetilde{h}_j \nonumber \\
    &= \sum_{j=1,\; j\neq i}^{n} \widetilde{d}_i \mathcal{K}_{\epsilon}(x_i,x_j) \widetilde{d}_j +  \mathcal{O}_{m,n}^{(\epsilon)} \left(\mathcal{E}(m)\right) \sum_{j=1,\; j\neq i}^{n} \widetilde{h}_i \widetilde{h}_j \nonumber \\
    &= \frac{1}{(n-1) (\pi \epsilon)^{d/2}}\sum_{j=1,\; j\neq i}^{n} \rho_{\epsilon}(x_i) \mathcal{K}_{\epsilon}(x_i,x_j) \rho_{\epsilon}(x_j) +  \mathcal{O}_{m,n}^{(\epsilon)} \left(\mathcal{E}(m)\right) \sum_{j=1,\; j\neq i}^{n} \widetilde{h}_i \widetilde{h}_j \nonumber \\
    &= \frac{1}{(n-1) (\pi \epsilon)^{d/2}}\sum_{j=1,\; j\neq i}^{n} \rho_{\epsilon}(x_i) \mathcal{K}_{\epsilon}(x_i,x_j) \rho_{\epsilon}(x_j) +  \mathcal{O}_{m,n}^{(\epsilon)} \left(\mathcal{E}(m)\right), \label{eq:noise in density est bound}
\end{align}
for all $i=1,\ldots,n$, where we used the fact that
\begin{equation}
    \left\vert \sum_{j=1,\; j\neq i}^{n} \widetilde{h}_i \widetilde{h}_j \right\vert \leq (n-1)\max_{i=1,\ldots,n} \{\widetilde{h}_i^2\} \leq \frac{\max_{x\in\mathcal{M}} \rho_{\epsilon}^2(x) }{(\pi \epsilon)^{d/2}} = \mathcal{O}_{m,n}^{(\epsilon)} \left(1\right),
\end{equation}
since $\Vert x \Vert \leq 1$ for all $x\in\mathcal{M}$ and $\rho_{\epsilon}(x)$ is bounded.
Continuing, when conditioning on the value of $x_i$, Hoeffding's inequality asserts that
\begin{align}
    &\frac{1}{(n-1) (\pi \epsilon)^{d/2}}\sum_{j=1,\; j\neq i}^{n} \rho_{\epsilon}(x_i) \mathcal{K}_{\epsilon}(x_i,x_j) \rho_{\epsilon}(x_j) \nonumber \\
    &= \frac{1}{(\pi \epsilon)^{d/2}} \left[ \int_\mathcal{M}  \rho_{\epsilon}(x_i) \mathcal{K}_{\epsilon}(x_i,y) \rho_{\epsilon}(y) q(y) d\mu(y) + \mathcal{O}_{m,n}^{(\epsilon)} \left(\sqrt{\frac{{\log n}}{n}}\right) \right] 
    = 1 + \mathcal{O}_{m,n}^{(\epsilon)} \left(\sqrt{\frac{{\log n}}{n}}\right), \label{eq:sample to population bound in density est}
\end{align}
where we used~\eqref{eq:integral scaling eq with density}. Since~\eqref{eq:sample to population bound in density est} holds conditionally on any value of $x_i$, it also holds unconditionally.
Overall, combining~\eqref{eq:noise in density est bound} with~\eqref{eq:sample to population bound in density est} and applying the union bound tells us that
\begin{align}
    \max_{i=1,\ldots,n} \left\vert \sum_{j=1}^{n} \widetilde{h}_i {H}_{i,j} \widetilde{h}_j - 1\right\vert = \mathcal{O}_{m,n}^{(\epsilon)} \left( \widetilde{\mathcal{E}}(m,n) \right), \label{eq: sum of h_tilde_i H_tilde_i_j h_tilde_j approx}
\end{align}
where we defined
\begin{equation}
    \widetilde{\mathcal{E}}(m,n) = \operatorname{max}\left\{E \sqrt{\log m} ,E^2\sqrt{m \log m},\sqrt{\frac{\log n}{n}}\right\}. \label{eq:widetilde E def}
\end{equation}

Next, we aim to apply Lemma~\ref{lem:closeness of scaling factors} with $A = {H}$ and $\hat{v}_i = \widetilde{h}_i$. To that end, we first need an upper bound on $1/((n-2) \min_{i\neq j}\{\widetilde{h}_i {H}_{i,j} \widetilde{h}_j \} )$. According to Lemma~\ref{lem:noise scalar product concentration} and the fact that $\Vert x \Vert \leq 1$ for all $x\in\mathcal{M}$, we have
\begin{equation}
    \langle y_i, y_j \rangle = \mathcal{O}_{m,n}^{(\epsilon)} \left(1\right),
\end{equation}
for all $i\neq j$, and it follows that
\begin{equation}
    \frac{1}{{H}_{i,j}} \leq \max_{i\neq j} \left\{ \operatorname{exp}\{ 2 |\langle y_i, y_j \rangle| / \epsilon  \} \right\} =  \mathcal{O}_{m,n}^{(\epsilon)} \left(1\right),
\end{equation}
for all $i \neq j$. Consequently,
\begin{align}
    \frac{1}{(n-2) \widetilde{h}_i {H}_{i,j} \widetilde{h}_j } &= \left( \frac{n-1}{n-2} \right)  \frac{ (\pi \epsilon)^{d/2} \operatorname{exp}\{(\Vert x_i \Vert_2^2 + \Vert x_j\Vert_2^2)/\epsilon \}  }{\rho_{\epsilon}(x_i) \rho_{\epsilon}(x_j) {H}_{i,j}} 
    =  \mathcal{O}_{m,n}^{(\epsilon)} \left(1\right), \label{eq:bound on smallest entry for lemma}
\end{align}
for all $i \neq j$, since $\rho_{\epsilon}(x)$ is bounded from below away from zero. Hence, utilizing~\eqref{eq: sum of h_tilde_i H_tilde_i_j h_tilde_j approx} and~\eqref{eq:bound on smallest entry for lemma}, Lemma~\ref{lem:closeness of scaling factors} asserts that
\begin{equation}
    \sqrt{\frac{1-\mathcal{O}_{m,n}^{(\epsilon)} \left( \widetilde{\mathcal{E}}(m,n) \right)}{1 + \mathcal{O}_{m,n}^{(\epsilon)} \left( \widetilde{\mathcal{E}}(m,n) \right)}} \leq \frac{h_i}{\widetilde{h}_i} \leq \sqrt{\frac{1}{\left[1-\mathcal{O}_{m,n}^{(\epsilon)} \left( \widetilde{\mathcal{E}}(m,n) \right) \right]\left[ 1 - \mathcal{O}_{m,n}^{(\epsilon)} \left( \widetilde{\mathcal{E}}(m,n) \right) \right]}}, \label{eq: h_i/h_i_tilde lower and upper bound}
\end{equation}
for all $i=1,\ldots,n$. Therefore, by properties~\eqref{eq: definition order with high probability property 1} and~\eqref{eq: definition order with high probability property 2} of Definition~\ref{def: order with high probability} (utilizing the fact that $\widetilde{\mathcal{E}}(m,n) \rightarrow 0$ as $m,n\rightarrow \infty$), we conclude that
\begin{equation}
    \max_{i=1,\ldots,n}\left\vert \frac{h_i}{\widetilde{h}_i} - 1 \right\vert = \mathcal{O}_{m,n}^{(\epsilon)} \left(\widetilde{\mathcal{E}}(m,n)\right). \label{eq: h_i/h_i_tilde bound}
\end{equation}
Next, using~\eqref{eq:H_tilde approx} and~\eqref{eq: h_i/h_i_tilde bound} we can write
\begin{align}
    W_{i,j} &= {d}_i {K}_{i,j} {d}_j
    = {h}_i {H}_{i,j} {h}_j 
    = h_i \widetilde{H}_{i,j} h_j \left(1 + \mathcal{O}_{m,n}^{(\epsilon)} \left({\widetilde{\mathcal{E}}(m,n)}\right) \right) \nonumber \\
    &= \left( \frac{h_i}{\widetilde{h}_i} \right) \widetilde{h}_i \widetilde{H}_{i,j} \widetilde{h}_j \left( \frac{h_j}{\widetilde{h}_j} \right) \left(1 + \mathcal{O}_{m,n}^{(\epsilon)} \left(\widetilde{\mathcal{E}}(m,n)\right) \right) \nonumber \\
    &= \widetilde{d}_i \mathcal{K}_{\epsilon}(x_i,x_j) \widetilde{d}_j \left(1 +  \mathcal{O}_{m,n}^{(\epsilon)} \left(\widetilde{\mathcal{E}}(m,n)\right) \right)^3  
    = \frac{\rho_{\epsilon}(x_i) \mathcal{K}_{\epsilon}(x_i,x_j) \rho_{\epsilon}(x_j)}{(n-1) (\pi \epsilon)^{d/2}} \left(1 + \mathcal{O}_{m,n}^{(\epsilon)} \left(\widetilde{\mathcal{E}}(m,n)\right) \right),
\end{align}
and applying the union bound gives~\eqref{eq: W_ij variance error}.
Lastly, by~\ref{eq: d_i_tilde, h_i_tilde, and h_i def} and~\eqref{eq: h_i/h_i_tilde bound}, we have
\begin{align}
    1 + \mathcal{O}_{m,n}^{(\epsilon)} \left(\widetilde{\mathcal{E}}(m,n)\right) &= \frac{h_i}{\widetilde{h}_i} 
    = \frac{d_i \sqrt{(n-1)(\pi \epsilon)^{d/2}}}{\rho_{\epsilon}(x)} \operatorname{exp}\left\{ - (\Vert y_i \Vert_2^2 - \Vert {x}_i \Vert_2^2) / \epsilon \right\} \nonumber \\
    &= \frac{d_i \sqrt{(n-1)(\pi \epsilon)^{d/2}}}{\rho_{\epsilon}(x)} \operatorname{exp}\left\{ - \Vert \eta_i \Vert_2^2/\epsilon \right\} \operatorname{exp}\left\{ - 2\langle x_i,\eta_i \rangle) / \epsilon \right\}, \label{eq: h_i over h_i_tilde expression}
\end{align}
where analogously to~\eqref{eq: langle eta, x rangle bound }, 
\begin{equation}
    \langle x_i,\eta_i \rangle = \mathcal{O}_{m,n}^{(\epsilon)}(\mathcal{E}(m)). \label{eq: inner product between x_i and eta_i bound}
\end{equation}
Combining~\eqref{eq: h_i over h_i_tilde expression} and~\eqref{eq: inner product between x_i and eta_i bound} while utilizing Definition~\ref{def: order with high probability} and applying the union bound establishes~\eqref{eq: d_i variance error}.

\section{Proof of Theorem~\ref{thm: scaling function convergence}} \label{appendix: proof of scaling function convergence}
Define the function $e(x): \mathcal{M} \rightarrow (-1,\infty)$ via 
\begin{equation}
    \rho_{\epsilon}(x) = \frac{F_{\epsilon}(x)}{\sqrt{q(x)}}(1+e(x)), \label{eq: e error from rho def}
\end{equation}
where $F_\epsilon$ is from~\eqref{eq: F_eps def}.
According to~\eqref{eq:integral scaling eq with density}, we have
\begin{multline}
    1 = \frac{F_{\epsilon}(x)}{(\pi \epsilon)^{d/2}  \sqrt{q(x)}} \bigg( \int_{\mathcal{M}} {\mathcal{K}_{\epsilon}(x,y)} F_{\epsilon}(y) \sqrt{q(y)} d\mu(y) 
    + \int_{\mathcal{M}} {\mathcal{K}_{\epsilon}(x,y)} F_{\epsilon}(y) e(y) \sqrt{q(y)} d\mu(y) \\
    + e(x) \int_{\mathcal{M}} {\mathcal{K}_{\epsilon}(x,y)} F_{\epsilon}(y) (1 + e(y)) \sqrt{q(y)} d\mu(y) \bigg),
\end{multline}
for all $x\in\mathcal{M}$. Using the fact that $q(x)\in \mathcal{C}^6(\mathcal{M})$, applying Lemma~\ref{lem: diffusion maps lemma}, and after some manipulation,
\begin{equation}
    \frac{F_{\epsilon}(x)}{(\pi \epsilon)^{d/2} \sqrt{q(x)}} \int_{\mathcal{M}} {\mathcal{K}_{\epsilon}(x,y)} F_{\epsilon}(y) \sqrt{q(y)} d\mu(y) = 1 + \mathcal{O}(\epsilon^2), \label{eq: int k sqrt(q) asymptotic expansion in epsilon}
\end{equation}
for all $x\in\mathcal{M}$. The reason that we need $q(x)\in \mathcal{C}^6(\mathcal{M})$ is that the above derivation requires us to apply Lemma~\ref{lem: diffusion maps lemma} to a function $f(x)$ involving $\Delta_{\mathcal{M}}\{q^{-1/2}\}$. Hence, we need $q(x)\in \mathcal{C}^5(\mathcal{M})$ so that $\Delta_{\mathcal{M}}\{q^{-1/2}\} \in \mathcal{C}^3(\mathcal{M})$. The stronger requirement $q(x)\in \mathcal{C}^6(\mathcal{M})$ is to make sure that the constant in the $\mathcal{O}(\epsilon^2)$ term in Lemma~\ref{lem: diffusion maps lemma} can be bounded uniformly for all $x\in\mathcal{M}$.

Next, according to~\eqref{eq: e error from rho def} and~\eqref{eq:integral scaling eq with density},
\begin{equation}
    \frac{e(x) F_{\epsilon}(x) }{(\pi \epsilon)^{d/2} \sqrt{q(x)}} \int_{\mathcal{M}} {\mathcal{K}_{\epsilon}(x,y)} F_{\epsilon}(y) (1+e(y)) \sqrt{q(y)} d\mu(y) = \frac{e(x) F_{\epsilon}(x)}{\rho_{\epsilon}(x) \sqrt{q(x)}},
\end{equation}
for all $x\in\mathcal{M}$. Therefore, we have that
\begin{equation}
    \mathcal{O}(\epsilon^2) = \frac{F_{\epsilon}(x)}{(\pi \epsilon)^{d/2} \sqrt{q(x)}} \int_{\mathcal{M}} {\mathcal{K}_{\epsilon}(x,y)} F_{\epsilon}(y) e(y) \sqrt{q(y)} d\mu(y) + \frac{e(x) F_{\epsilon}(x) }{\rho_{\epsilon}(x) \sqrt{q(x)}},
\end{equation}
uniformly for all $x\in\mathcal{M}$. Multiplying the above by $e(x)q(x)$, integrating over $x\in\mathcal{M}$ with respect to $d\mu(x)$, and using the fact that $q(x) \leq \overline{C}$ and $0.5 \leq F_{\epsilon}(x)$ for all sufficiently small $\epsilon$, we obtain
\begin{align}
    &\mathcal{O}(\epsilon^2) \int_{\mathcal{M}} \vert e(x) \vert d\mu(x) 
    \geq \int_{\mathcal{M}} \mathcal{O}(\epsilon^2)  e(x) q(x) d\mu(x) \nonumber \\
    &\geq  \frac{1}{(\pi \epsilon)^{d/2}} \int_{\mathcal{M}} \int_{\mathcal{M}} F_{\epsilon}(x) \sqrt{q(x)} e(x) {\mathcal{K}_{\epsilon}(x,y)} e(y) \sqrt{q(y)} F_{\epsilon}(y) d\mu(y) d\mu(x) + 2\sqrt{\overline{C}}\int_{\mathcal{M}} \frac{e^2(x) }{\rho_{\epsilon}(x)} d\mu(x), \nonumber \\
    &\geq 2\sqrt{\overline{C}}\int_{\mathcal{M}} \frac{e^2(x) }{\rho_{\epsilon}(x)} d\mu(x), \label{eq: bound 1}
\end{align}
where we also used the fact that the operator $Gf = \int_{\mathcal{M}} {\mathcal{K}_{\epsilon}(x,y)} f(y) d\mu(y)$ is positive definite with respect to the standard $L^2_{(\mathcal{M},d\mu)}$ inner product $\left\langle f , g \right\rangle_{L^2_{(\mathcal{M},d\mu)}} = \int_{\mathcal{M}} f(x) g(x) d\mu(x)$. 

Next, we provide a lower bound for $\int_\mathcal{M} e^2(x)/\rho_{\epsilon}(x) d\mu(x)$ using Lemma~\ref{lem: boundedness of rho}. For any pair of positive and continuous functions $f,g$ on $\mathcal{M}$, Holder's inequality asserts that
\begin{equation}
    \int_{\mathcal{M}} f(x) d\mu(x) = \int_{\mathcal{M}} \frac{f(x)}{g(x)} g(x) d\mu(x) 
    \leq \left( \int_{\mathcal{M}} \left[ \frac{f(x)}{g(x)} \right]^\ell d\mu(x) \right)^{1/\ell} \left( \int_{\mathcal{M}} [g(x)]^q d\mu(x) \right)^{1/q}, 
\end{equation}
for $\ell,q\in [1,\infty]$, $1/\ell+1/q=1$. Taking $f(x) = \vert e(x) \vert^{2/\ell}$, $g(x) = [\rho_{\epsilon}(x)]^{1/\ell}$, and using the fact that $q = \ell/(\ell-1)$, we obtain
\begin{equation}
    \int_{\mathcal{M}} \frac{e^2(x) }{\rho_{\epsilon}(x)} d\mu(x) \geq \frac{\left( \int_{\mathcal{M}} \vert e(x) \vert^{2/\ell} d\mu(x) \right)^\ell}{\left( \int_{\mathcal{M}} [\rho_{\epsilon}(x)]^{1/(\ell-1)} d\mu(x) \right)^{(\ell-1)}}. \label{eq: bound 2}
\end{equation}
Lemma~\ref{lem: boundedness of rho} asserts that $\vert \{x\in\mathcal{M}: \; [\rho_\epsilon(x)]^{1/(\ell-1)} > t \}\vert \leq C/t^{2(\ell-1)}$ for all $\epsilon \leq \epsilon_0$, and taking $\ell > 3/2$, we have
\begin{equation}
\int_{\mathcal{M}} [\rho_{\epsilon}(x)]^{1/(\ell-1)} d\mu(x) \leq \vert \{x:\mathcal{M}: \; \rho_{\epsilon}(x) < 1  \} \vert + \int_{1}^\infty \frac{C}{t^{2(\ell-1)}} dt \leq \vert \mathcal{M} \vert + \frac{C}{2p-3} := C_\ell < \infty. \label{eq: bound 3}
\end{equation}
Plugging~\eqref{eq: bound 3} and~\eqref{eq: bound 2} into~\eqref{eq: bound 1}, we arrive at
\begin{equation}
    \left( \int_{\mathcal{M}} \vert e(x) \vert^{2/\ell} d\mu(x) \right)^\ell \leq {\mathcal{O}(\epsilon^2)} C_\ell^{\ell-1} \int_{\mathcal{M}} \vert e(x) \vert d\mu(x) \leq {\mathcal{O}(\epsilon^2)} C_\ell^{'} \left( \int_{\mathcal{M}} \vert e(x) \vert^{2/\ell} d\mu(x) \right)^{\ell/2},
\end{equation}
where we used Holder's inequality for $2/\ell \geq 1$ in the last transition, and denoted $C_\ell^{'} = C_\ell^{\ell-1} \vert \mathcal{M} \vert^{\ell/2}$. Since $\ell>3/2$ and $2/\ell\geq 1$, we have for $p := 2/\ell \in [1,4/3)$,
\begin{equation}
    \left( \int_{\mathcal{M}} \vert e(x) \vert^{p} d\mu(x) \right)^{1/p} \leq {\mathcal{O}(\epsilon^2)} C_{2/p}^{'}. \label{eq: e(x) L_ell bound in proof}
\end{equation}
Plugging $e(x) = (\sqrt{q(x)}/F_{\epsilon}(x)) (\rho_{\epsilon}(x) - F_{\epsilon}(x)/\sqrt{q(x)})$ into the above and using the fact that $q(x)$ is bounded from below away from zero and $F_{\epsilon}(x)$ is upper bounded (for all sufficiently small $\epsilon$) proves~\eqref{eq: L_p convergence}.

Next, we use~\eqref{eq:integral scaling eq with density},~\eqref{eq: e error from rho def}, and~\eqref{eq: int k sqrt(q) asymptotic expansion in epsilon} to write
\begin{align}
    \frac{1}{\rho_{\epsilon}(x)} &= \frac{1}{(\pi \epsilon)^{d/2}}\int_{\mathcal{M}} \mathcal{K}_{\epsilon}(x,y) F_{\epsilon}(y) (1 + e(y)) \sqrt{q(y)} d\mu(y) \nonumber \\
    &= \frac{\sqrt{q(x)}}{F_{\epsilon}(x)}[1 + \mathcal{O}(\epsilon^2)] +  \frac{1}{(\pi \epsilon)^{d/2}}\int_{\mathcal{M}} \mathcal{K}_{\epsilon}(x,y) F_{\epsilon}(y) e(y) \sqrt{q(y)} d\mu(y), \label{eq: 1/rho convergence in epsilon}
\end{align}
universally for all $x\in\mathcal{M}$. Since $q(x) \leq \overline{C}$ and $F_{\epsilon}(x)\leq 2$ for all sufficiently small $\epsilon$, Holder's inequality and~\eqref{eq: e(x) L_ell bound in proof} imply that
\begin{align}
    \left\vert \frac{1}{(\pi \epsilon)^{d/2}}\int_{\mathcal{M}} \mathcal{K}_{\epsilon}(x,y) F_{\epsilon}(y) e(y) \sqrt{q(y)} d\mu(y) \right\vert 
    &\leq 2\sqrt{\overline{C}} \frac{1}{(\pi \epsilon)^{d/2}} \left(\int_{\mathcal{M}} [\mathcal{K}_{\epsilon}(x,y)]^\ell d\mu(y) \right)^{1/\ell} \left(\int_{\mathcal{M}} \vert e(y) \vert^p d\mu(y) \right)^{1/p} \nonumber \\
    &\leq 2\sqrt{\overline{C}} C_{2/p}^{'} \frac{\mathcal{O}(\epsilon^2)}{(\pi \epsilon)^{d/2}} \left(\int_{\mathcal{M}} [\mathcal{K}_{\epsilon}(x,y)]^\ell d\mu(y) \right)^{1/\ell} \nonumber \\
    &\leq 2\sqrt{\overline{C}} C_{2/p}^{'} \frac{\mathcal{O}(\epsilon^2)}{(\pi \epsilon)^{d/2}} \left(\frac{\pi \epsilon}{\ell}\right)^{d/(2\ell)}[1+\mathcal{O}(\epsilon)]^{1/\ell}, \label{eq: int K e q bound}
\end{align}
for $p \in [1, 4/3)$ and $\ell = p/(p - 1)$, where we used~\eqref{eq: int k sqrt(q) asymptotic expansion in epsilon} with $q(x)\equiv 1$ and $[\mathcal{K}_{\epsilon}(x,y)]^\ell = \mathcal{K}_{\epsilon/\ell}(x,y)$. If we take $p$ arbitrarily close to $4/3$, then $\ell$ approaches $4$, and we see that for any $d \leq 5$, the right-hand side of~\eqref{eq: int K e q bound} converges to zero as $\epsilon \rightarrow 0 $. Therefore, together with~\eqref{eq: 1/rho convergence in epsilon} and the definition of $F_{\epsilon}(x)$ in~\eqref{eq: e error from rho def}, it implies that $1/\rho_{\epsilon}(x)$ converges to $\sqrt{q(x)}$ uniformly for all $x\in\mathcal{M}$ as $\epsilon\rightarrow 0$. Consequently, $1/\rho_{\epsilon}(x)$ is bounded from below away from zero by a constant for all $x\in\mathcal{M}$ and all sufficiently small $\epsilon$. Using this observation in~\eqref{eq: bound 1}, we obtain
\begin{equation}
    \int_{\mathcal{M}} \vert e(x) \Vert_2^2 d\mu(x) \leq \mathcal{O}(\epsilon^2) \int_{\mathcal{M}} \vert e(x) \vert d\mu(x) = \mathcal{O}(\epsilon^4), \label{eq: L2 convergence in proof}
\end{equation}
where we used~\eqref{eq: e(x) L_ell bound in proof} with $p=1$ in the last transition. Overall, we establishes that~\eqref{eq: e(x) L_ell bound in proof} holds for $p\in [4/3, 2]$ if $d \leq 5$. Lastly, we can now use~\eqref{eq: L2 convergence in proof} in~\eqref{eq: int K e q bound} for $\ell = 2$, $p=2$, which together with~\eqref{eq: 1/rho convergence in epsilon} establishes~\eqref{eq: pointwise convergence with d <= 2}.

\section{Proof of Theorem~\ref{thm: density estimator}} \label{appendix: proof of robust density estimator}
According to Theorem~\ref{thm:scaling factors asym expression}, we have
\begin{align}
    \frac{1}{\hat{q}_i} 
    &= (n-1) \left( \sum_{j=1,\; j\neq i}^{n} \left[ {d}_i {K}_{i,j} {d}_j \right]^s \right)^{1/(s-1)} \nonumber \\ 
    &= {(\pi \epsilon)^{(ds/2)/(1-s)}} \left[ \frac{1}{n-1}\sum_{j=1,\; j\neq i}^{n} \rho_{\epsilon}^s(x_i) \mathcal{K}_{\epsilon}^s(x_i,x_j) \rho_{\epsilon}^s(x_j) \left(1 + \mathcal{O}_{m,n}^{(\epsilon)} \left(\widetilde{\mathcal{E}}(m,n)\right) \right)^s \right]^{1/(s-1)} \nonumber \\ 
    &= {(\pi \epsilon)^{(ds/2)/(1-s)}} \left[ \frac{1}{n-1}\sum_{j=1,\; j\neq i}^{n} \rho_{\epsilon}^s(x_i) \mathcal{K}_{\epsilon}^s(x_i,x_j) \rho_{\epsilon}^s(x_j) \right]^{1/(s-1)} \left(1 + \mathcal{O}_{m,n}^{(\epsilon)} \left(\widetilde{\mathcal{E}}(m,n)\right) \right),
\end{align}    
where $\widetilde{\mathcal{E}}(m,n)$ is from~\eqref{eq:widetilde E def}. When conditioning on the value of $x_i$, Hoeffding's inequality tells us that
\begin{align}
    &{(\pi \epsilon)^{(ds/2)/(1-s)}} \left[ \frac{1}{n-1}\sum_{j=1,\; j\neq i}^{n} \rho_{\epsilon}^s(x_i) \mathcal{K}_{\epsilon}^s(x_i,x_j) \rho_{\epsilon}^s(x_j) \right]^{1/(s-1)} \nonumber \\ 
    &= {(\pi \epsilon)^{(ds/2)/(1-s)}} \left[ \int_{\mathcal{M}} \rho_{\epsilon}^s(x_i) \mathcal{K}_{\epsilon}^s(x_i,y) \rho_{\epsilon}^s(y) q(y) d\mu(y) \right]^{1/(s-1)} + \mathcal{O}_{m,n}^{(\epsilon)} \left(\sqrt{\frac{{\log n}}{n}}\right).
  \label{eq: quadratic sum density est variance error}
\end{align}
Let us define
\begin{equation}
    g(x) = {(\pi \epsilon)^{(ds/2)/(1-s)}} \left[ \int_{\mathcal{M}} \rho_{\epsilon}^s(x) \mathcal{K}_{\epsilon}^s(x,y) \rho_{\epsilon}^s(y) q(y) d\mu(y) \right]^{1/(s-1)}, \label{eq: g def}
\end{equation}
which is bounded for all $x\in\mathcal{M}$ by a constant that depends on $\epsilon$.
Consequently, we obtain that
\begin{equation}
    \frac{1}{\hat{q}_i} = \left( g(x_i) + \mathcal{O}_{m,n}^{(\epsilon)} \left(\sqrt{\frac{{\log n}}{n}}\right) \right) \left(1 + \mathcal{O}_{m,n}^{(\epsilon)} \left(\widetilde{\mathcal{E}}(m,n)\right) \right) =  g(x_i) + \mathcal{O}_{m,n}^{(\epsilon)} \left(\widetilde{\mathcal{E}}(m,n)\right), \label{eq: 1/q_hat_i bound}
\end{equation}
for all $i=1,\ldots,n$.

Next, we turn to establish the bias error. 
According to Assumption~\ref{assump: pointwise convergence}, for $s> 0$, $s\neq 1$, we can write
\begin{align}
    g(x) &= (\pi \epsilon)^{(ds/2)/(1-s)} \left[{q^{-s/2}(x)} \int_{\mathcal{M}} \mathcal{K}_{\epsilon}^s(x,y) q^{1-s/2}(x) d\mu(y) \left( 1 + \mathcal{O}(\epsilon) \right)^{2s} \right]^{1/(s-1)} \nonumber \\
    &=\frac{1 + \mathcal{O}(\epsilon)}{s^{d/(2(s-1))}(\pi \epsilon)^{d/2} q(x)} , \label{eq: g(x) asym approx}
\end{align}
where we also used the fact that, according to Lemma~\ref{lem: diffusion maps lemma}, for $s\neq 0$ and $x\in\mathcal{M}$,
\begin{equation}
    \int_{\mathcal{M}} \mathcal{K}_{\epsilon}^s(x_i,y) q^{1-s/2}(y) d\mu(y) = \int_{\mathcal{M}} K_{\epsilon/s}(x_i,y) q^{1-s/2}(y) d\mu(y) = \left(\frac{\pi \epsilon}{s}\right)^{d/2} q^{1-s/2}(x)\left[  1 + \mathcal{O}(\epsilon) \right].
\end{equation}
Combining~\eqref{eq: g(x) asym approx} with~\eqref{eq: 1/q_hat_i bound} and using the fact that $q(x)$ is bounded gives 
\begin{equation}
    \frac{s^{d/(2(s-1))}(\pi \epsilon)^{d/2} q(x_i)}{\hat{q}_i} = 1 + \mathcal{O}(\epsilon) + \mathcal{O}_{m,n}^{(\epsilon)} \left(\widetilde{\mathcal{E}}(m,n)\right),
\end{equation}
for all $i=1,\ldots,n$, where $\widetilde{\mathcal{E}}(m,n)$ is from~\eqref{eq:widetilde E def}. Lastly, applying the union bound over all $i=1,\ldots,n$ concludes the proof.

\section{Proof of Proposition~\ref{prop: noise magnitude est}} \label{appenidx: proof of noise magnitude estimation}
Utilizing~\eqref{eq: noise magnitude estimator}, Theorem~\ref{thm:scaling factors asym expression}, Assumption~\ref{assump: pointwise convergence}, and Theorem~\ref{thm: density estimator}, we have
\begin{align}
    \hat{N}_i &= \epsilon \log \left( s^{{d}/{(4(s-1))}} \rho_{\epsilon}(x_i) \sqrt{q(x_i)} \operatorname{exp}\left( \frac{\Vert \eta_i \Vert_2^2}{\epsilon} \right) \left( 1 + \mathcal{E}_{i,i} \right)    \sqrt{ 1 + \mathcal{O}(\epsilon) +  \mathcal{E}^{(1)}_{i} } \right) \nonumber \\
    &= \epsilon \log \left( s^{{d}/{(4(s-1))}} \operatorname{exp}\left( \frac{\Vert \eta_i \Vert_2^2}{\epsilon} \right) \left( 1 + \mathcal{O}(\epsilon) \right)\left( 1 + \mathcal{E}_{i,i} \right)    \sqrt{ 1 + \mathcal{O}(\epsilon) +  \mathcal{E}^{(1)}_{i} } \right) \nonumber \\
    &= \Vert \eta_i \Vert_2^2 + \epsilon \frac{d \log s}{4(s-1)} + \epsilon \log \left( \left( 1 + \mathcal{O}(\epsilon) \right)\left( 1 + \mathcal{E}_{i,i} \right)    \sqrt{ 1 + \mathcal{O}(\epsilon) +  \mathcal{E}^{(1)}_{i} } \right) \nonumber \\
    &= \Vert \eta_i \Vert_2^2 + \epsilon \frac{d \log s}{4(s-1)} + \epsilon \log \left( 1 + \mathcal{O}(\epsilon) + \mathcal{O}_{m,n}^{(\epsilon)} \left(\widetilde{\mathcal{E}}(m,n)\right) \right) \nonumber \\
    &= \Vert \eta_i \Vert_2^2 + \epsilon \frac{d \log s}{4(s-1)} + \mathcal{O}(\epsilon^2) + \mathcal{O}_{m,n}^{(\epsilon)} \left(\widetilde{\mathcal{E}}(m,n)\right),
\end{align}
by Taylor expansion, where $\widetilde{\mathcal{E}}(m,n)$ is from~\eqref{eq:widetilde E def}. Next, according to~\eqref{eq: inner product between x_i and eta_i bound},
\begin{equation}
    \Vert y_i \Vert_2^2 = \Vert x_i \Vert_2^2 + \Vert \eta_i \Vert_2^2 + \mathcal{O}_{m,n}^{(\epsilon)} \left({\mathcal{E}}(m)\right), \label{eq: noisy point magnitude expansion}
\end{equation}
where ${\mathcal{E}}(m)$ is from~\eqref{eq: E caligraphic def}, hence 
\begin{align}
    \hat{S}_i &= \Vert x_i \Vert_2^2 + \Vert \eta_i \Vert_2^2 + \mathcal{O}_{m,n}^{(\epsilon)} \left({\mathcal{E}}(m)\right) - \hat{N}_i \nonumber \\
    &= \Vert x_i \Vert_2^2 - \epsilon \frac{d \log s}{4(s-1)} + \mathcal{O}(\epsilon^2) + \mathcal{O}_{m,n}^{(\epsilon)} \left(\widetilde{\mathcal{E}}(m,n)\right).
\end{align}
Lastly, by Lemma~\ref{lem:noise scalar product concentration} and~\eqref{eq: noisy point magnitude expansion}, 
\begin{align}
    \Vert y_i - y_j \Vert_2^2 &= \Vert y_i \Vert_2^2 + \Vert y_j \Vert_2^2 -2\langle y_i,y_j\rangle \nonumber \\
    &= \Vert x_i \Vert_2^2 + \Vert \eta_i \Vert_2^2 + \Vert x_j \Vert_2^2 + \Vert \eta_j \Vert_2^2 -2\langle {x}_i,{x}_j\rangle + \mathcal{O}_{m,n}^{(\epsilon)}\left({\mathcal{E}}(m)\right) \nonumber \\
    &= \Vert {x}_i - {x}_j \Vert_2^2 + \Vert \eta_i \Vert_2^2 + \Vert \eta_j \Vert_2^2 + \mathcal{O}_{m,n}^{(\epsilon)}\left({\mathcal{E}}(m)\right),
\end{align}
and consequently,
\begin{align}
    \hat{D}_{i,j} &= \Vert y_i - y_j \Vert_2^2 - \hat{N}_i -\hat{N}_j \nonumber \\
    &= \Vert {x}_i - {x}_j \Vert_2^2 + \Vert \eta_i \Vert_2^2 + \Vert \eta_j \Vert_2^2 - \hat{N}_i -\hat{N}_j + \mathcal{O}_{m,n}^{(\epsilon)}\left({\mathcal{E}}(m)\right) \nonumber \\
    &= \Vert {x}_i - {x}_j \Vert_2^2 - \epsilon \frac{d \log s}{2(s-1)} + \mathcal{O}(\epsilon^2) + \mathcal{O}_{m,n}^{(\epsilon)} \left(\widetilde{\mathcal{E}}(m,n)\right),
\end{align}
which concludes the proof.

\section{Proof of Theorem~\ref{thm: Laplacian estimator}} \label{appendix: proof of robust graph Laplacian normalization}
Let us write
\begin{align}
    \sum_{j=1}^n L^{(\alpha)}_{i,j} f(x_j) &= \frac{4}{\epsilon} \left[ f(x_i) - \sum_{j=1}^n \hat{W}^{(\alpha)}_{i,j} f(x_j) \right]
    = \frac{4}{\epsilon} \left[ f(x_i) - \frac{\sum_{j=1}^n \widetilde{W}_{i,j}^{(\alpha)}f(x_j)}{\sum_{j=1}^n \widetilde{W}_{i,j}^{(\alpha)}} \right] \nonumber \\
     &= \frac{4}{\epsilon} \left[ f(x_i) - \frac{\sum_{j=1}^n {W}_{i,j} \hat{q}_j^{-(\alpha-1/2)} f(x_j)}{\sum_{j=1}^n {W}_{i,j} \hat{q}_j^{-(\alpha-1/2)}} \right]. \label{eq: Laplacian expression in proof}
\end{align}
According to Theorem~\ref{thm:scaling factors asym expression} and~\eqref{eq: 1/q_hat_i bound}, we have
\begin{multline}
    \sum_{j=1}^n {W}_{i,j} \hat{q}_j^{-(\alpha-1/2)} f(x_j) = \\  (\pi \epsilon)^{-d/2} \rho_{\epsilon}(x_i)  (n-1)^{-1} \sum_{j=1,\; j\neq i}^n \mathcal{K}_{\epsilon}(x_i,x_j) \rho_{\epsilon}(x_j) [g(x_j)]^{(\alpha-1/2)} f(x_j) \left( 1 + \mathcal{O}_{m,n}^{(\epsilon)} (\widetilde{\mathcal{E}}(m,n)) \right),
\end{multline}
where $\widetilde{\mathcal{E}}(m,n)$ is from~\eqref{eq:widetilde E def}, and applying Hoeffding's inequality (while conditioning on $x_i$), 
\begin{multline}
    \sum_{j=1}^n {W}_{i,j} \hat{q}_j^{-(\alpha-1/2)} f(x_j) = \\ \rho_{\epsilon}(x_i)  (\pi \epsilon)^{-d/2} \int_{\mathcal{M}} \mathcal{K}_{\epsilon}(x_i,y) \rho_{\epsilon}(y) [g(y)]^{\alpha-1/2} f(y) q(y) d\mu(y) \left( 1 + \mathcal{O}_{m,n}^{(\epsilon)} (\widetilde{\mathcal{E}}(m,n)) \right). \label{eq: numerator bias term}
\end{multline}
Utilizing Assumption~\ref{assump: pointwise convergence}, Equation~\eqref{eq: g def}, Lemma~\ref{lem: diffusion maps lemma}, and after some technical manipulation, it follows that
\begin{equation}
    \rho_{\epsilon}(y) [g(y)]^{\alpha-1/2} = c q^{-\alpha}(y) + \epsilon A(y) +\mathcal{O}(\epsilon^{1+\beta}), \label{eq: rho g}
\end{equation}
where $c>0$ is a global constant and $A\in \mathcal{C}^3(\mathcal{M})$ is independent of $\epsilon$. Consequently, by~\eqref{eq: rho g} and Lemma~\ref{lem: diffusion maps lemma} we have
\begin{align}
    &(\pi \epsilon)^{-d/2} \int_{\mathcal{M}} \mathcal{K}_{\epsilon}(x_i,y) \rho_{\epsilon}(y) [g(y)]^{(\alpha-1/2)} f(y) q(y) d\mu(y) \nonumber \\
    &= (\pi \epsilon)^{-d/2} \left[ c \int_{\mathcal{M}} \mathcal{K}_{\epsilon}(x_i,y) q^{1-\alpha}(y) f(y) d\mu(y) + \epsilon \int_{\mathcal{M}} \mathcal{K}_{\epsilon}(x_i,y) A(y) q(y) f(y) d\mu(y) \right] \nonumber \\
    &= c q^{1-\alpha}(x_i) f(x_i) + \frac{c \epsilon}{4}\left[ \omega(x_i) f(x_i) - \Delta_\mathcal{M}\{q^{1-\alpha} f\}(x_i) \right] + \epsilon A(x_i) q(x_i) f(x_i) +\mathcal{O}(\epsilon^{1+\beta}) \nonumber \\
    &= J_i f(x_i) - \frac{c \epsilon}{4} \Delta_\mathcal{M}\{q^{1-\alpha} f\}(x_i) + \mathcal{O}(\epsilon^{1+\beta}),
\end{align}
where we defined
\begin{equation}
    J_i = c q^{1-\alpha}(x_i) + \frac{c \epsilon}{4}\omega(x_i) + \epsilon A(x_i) q(x_i).
\end{equation}
Combining the above with~\eqref{eq: numerator bias term} and taking out the term $J_i$ as a common factor, we have 
\begin{equation}
    \sum_{j=1}^n {W}_{i,j} \hat{q}_j^{-(\alpha-1/2)} f(x_j) = 
    J_i \rho_{\epsilon}(x_i) \left(  f(x_i)  - \epsilon \frac{ \Delta_\mathcal{M}\{q^{1-\alpha} f\}(x_i)}{4q^{1-\alpha}(x_i)} + \mathcal{O}(\epsilon^{1+\beta}) \right) \left( 1 + \mathcal{O}_{m,n}^{(\epsilon)} (\widetilde{\mathcal{E}}(m,n)) \right).
\end{equation}
Similarly, by setting $f\equiv 1$ in the above, 
\begin{equation}
    \sum_{j=1}^n {W}_{i,j} \hat{q}_j^{-(\alpha-1/2)} = 
    J_i \rho_{\epsilon}(x_i) \left( 1 - \epsilon \frac{ \Delta_\mathcal{M}\{q^{1-\alpha}\}(x_i)}{4q^{1-\alpha}(x_i)} + \mathcal{O}(\epsilon^{1+\beta}) \right) \left( 1 + \mathcal{O}_{m,n}^{(\epsilon)} (\widetilde{\mathcal{E}}(m,n)) \right).
\end{equation}
Therefore, according to~\eqref{eq: Laplacian expression in proof} and a Taylor expansion in small $\epsilon$,
\begin{align}
    \sum_{j=1}^n L^{(\alpha)}_{i,j} f(x_j) &= \frac{4}{\epsilon} \left[ f(x_i) - \frac{f(x_i)  - \epsilon \frac{ \Delta_\mathcal{M}\{q^{1-\alpha} f\}(x_i)}{4q^{1-\alpha}(x_i)} + \mathcal{O}(\epsilon^{1+\beta})}{1 - \epsilon \frac{ \Delta_\mathcal{M}\{q^{1-\alpha}\}(x_i)}{4q^{1-\alpha}(x_i)} + \mathcal{O}(\epsilon^{1+\beta})} \right] + \mathcal{O}_{m,n}^{(\epsilon)} (\widetilde{\mathcal{E}}(m,n)) \nonumber\\
    &= \frac{ \Delta_\mathcal{M}\{q^{1-\alpha} f\}(x_i)}{q^{1-\alpha}(x_i)} - \frac{ \Delta_\mathcal{M}\{q^{1-\alpha}\}(x_i)}{q^{1-\alpha}(x_i)} f(x_i) + \mathcal{O}(\epsilon^{\beta}) + \mathcal{O}_{m,n}^{(\epsilon)} (\widetilde{\mathcal{E}}(m,n)),
\end{align}
thereby concluding the proof after applying the union bound over $i=1,\ldots,n$.

\end{appendices}

\small
\bibliographystyle{plain}
\bibliography{mybib}

\end{document}